\documentclass[11pt,a4paper,oneside,reqno]{amsart}

\usepackage{paquet_en}

\title{Volume of maximal representations into $\SO_0(2,3)$}
\author[T. Lemistre]{Timoth\'e Lemistre}
\address{LJAD, Universit\'e C\^ote d'Azur, Nice, France}
\email{tlemistre@unice.fr}
\date{\today}

\begin{document}

\begin{abstract}
	We study the volume of maximal representations from a surface group into $\SO_0(2,3)$. We show that it is bounded from above, uniformly in the genus of the surface. On the Gothen components, we prove that it is bounded from below by a strictly positive constant. 
\end{abstract}

\maketitle
{\setlength{\parskip}{0pt}
\tableofcontents}

\section*{Introduction}

Teichmüller space parametrises marked hyperbolic structures on the closed oriented surface $\Sigma_g$ of genus $g \geq 2$. By the uniformisation theorem and the work of Goldman (see \cite{G88}), it is also a connected component of the representation variety $\chi(\pi_1(\Sigma_g),G) = \Hom(\pi_1(\Sigma_g),G)/G$ for $G = \SO_0(2,1) = \PSL(2,\R)$. This connected component is homeomorphic to a ball and composed only of faithful and discrete representations. The work of Hitchin in the 1990s (see \cite{Hit92}) has given further examples of connected components homeomorphic to balls for other simple Lie groups $G$. Those connected components have been interpreted geometrically by Labourie (see \cite{Lab06}), who also showed that they contain only faithful and discrete representations. The latter property defines what Wienhard calls higher-rank Teichmüller space \cite{Wie18}. 

These generalisations of Teichmüller space rely on different classes of representations. The main examples are Hitchin representations (see \cite{Hit92}), maximal representations (\cite{BIW10}), positive representations (see \cite{FG06} and \cite{BIW10}), $\H^{p,q}$-convex cocompact representations (see \cite{BK25}) and $\Theta$-positive representations (see \cite{GW18}). We study here representations at the intersection of these notions : maximal representations in $\SO_0(2,3)$. 

Maximal representations of $\pi_1(\Sigma_g)$ into $\SO_0(2,q+1)$ ($q \geq 1$) admit a geometric interpretation using the pseudo-hyperbolic spaces $\H^{2,q}$, which are pseudo-Riemannian generalisations of the hyperbolic spaces. Indeed, they leave invariant particular submanifolds of $\H^{2,q}$ (see \cite{CTT19}), which allows associating them geometric quantities (see for example Mazzoli-Viaggi \cite{MV23}). 

We study here such a geometric quantity: the \emph{volume} of a maximal representation $\rho : \pi_1(\Sigma_g) \to \SO_0(2,q+1)$, defined as follows. By the work Danciger-Guéritaud-Kassel (see \cite{DGK18}), $\pi_1(\Sigma_g)$ preserves properly convex closed subsets of $\H^{2,q}$ on which it acts properly discontinuously and cocompactly. Among these, there exists a unique minimal convex, which we denote by $\cal{C}_\rho$. We call \emph{volume} of $\rho$ the volume of the quotient :
\[\Vol(\rho) = \Vol(\cal{C}_\rho/\pi_1(\Sigma_g)). \]

This notion of volume has been studied by Bonsante-Seppi-Tamburelli for $q=1$. Since $\SO_0(2,2)$ is a double cover of $\SO_0(2,1)^2$, a maximal representation of $\pi_1(\Sigma)$ into $\SO_0(2,2)$ corresponds to a pair of maximal representations into $\SO_0(2,1)$ and a representation $\epsilon$ into $\Z/2\Z$, thus its conjugacy class corresponds to a pair of points $(h,h')$ in Teichmüller space of $\Sigma$ and $\epsilon$ : we denote it $[\rho_{h,h',\epsilon}] \in \chi(\Gamma,\SO_0(2,2))$. They then show (see \cite{BST17}) :
\[\Vol(\rho_{h,h',\epsilon}) \stackrel{d(h,h') \to \infty}{\longrightarrow} +\infty, \]
where $d(h,h')$ is the Weil-Petersson distance between $h$ and $h'$. We show an opposite result for $n=2$. 

\begin{theoremeLettre}\label{théorème:borne_sup_cocompacte}
	There exists a constant $C > 0$ such that for every integer $g \geq 2$ and every maximal representation $\rho$ of $\pi_1(\Sigma_g)$ into $\SO_0(2,3)$, we have 
	\[\Vol(\rho) \leq C\, g. \]
\end{theoremeLettre}

Theorem \ref{théorème:borne_sup_cocompacte} is a particular case of Theorem \ref{theoreme:borne_sup}, which is the technical core of this article and concerns maximal surfaces. Let us denote by $\cal{M}^*_{2,2}$ the space of pointed complete maximal surfaces of $\H^{2,2}$ (see Definition \ref{définition:surface_pointée}). We use the following property, which appears in many similar contexts (see \cite{CT90}, \cite{CL15} or \cite{OT20}). A continuous function $f$ from $\cal{M}^*_{2,2}$ to $\R$ is said to be \emph{exponentially decaying} (see Definition \ref{définition:décroissance_exponentielle}) if there exist constants $C,\alpha > 0$ such that for every complete maximal surface $S$ of $\H^{2,2}$ and every ball $B(x,r)$ of $S$ on which the sectional curvature $K$ satisfies $\abs{K} \leq \frac14$, we have $\abs{f(S,x)} \leq C e^{-\alpha r}$. 

\noindent\textbf{Theorem \ref{theoreme:borne_sup}. }\textit{Let $f : \cal{M}_{2,2}^* \to \R$ be a continuous and exponentially decaying function. There exists a constant $C_f$ satisfying the following property. Let $\Gamma$ be a subgroup of $\PO(2,3)$ preserving a complete maximal surface $S$ of $\H^{2,2}$ and whose action on $S$ is properly discontinuous and free. Suppose that the sectional curvature $K$ of $S/\Gamma$ tends to $0$ at infinity. We have}
\[\int_{S/\Gamma} \abs{f} \leq C_f \int_{S/\Gamma} \abs{K}. \]

Theorem \ref{théorème:borne_sup_cocompacte} follows from Theorem \ref{theoreme:borne_sup} together with differential-geometric considerations (Section \ref{subsection:contrôle_volume}). Here are other applications of Theorem \ref{theoreme:borne_sup}. 

\noindent\textbf{Corollary \ref{corollaire:Moriani_volume}. }Moriani shows that a complete maximal surface $S$ of $\H^{2,2}$ has finite total curvature if, and only if, the boundary of $S$ in $\P(\R^{2,3})$ is the union of a finite number $N$ of projective segments. We then have $\int_S K = -\frac{\pi}{2} (N-4)$ \cite[Main Theorem]{Mor25}. Theorem \ref{theoreme:borne_sup} proves that there exists a constant $C_{\Vol}$ independent of $S$ and $N$ satisfying
\[\Vol(\Conv(S)) \leq C_{\Vol} (N-4). \]

\noindent\textbf{Corollary \ref{corollaire:cocompact_alpha}. }By \cite{CTT19}, every maximal representation $\rho$ from $\Gamma = \pi_1(\Sigma_g)$ to $\SO_0(2,3)$ defines a $\rho$-invariant maximal surface of $\H^{2,2}$ denoted by $S_\rho$ (see Section \ref{subsection:représentations_maximales}). The quotient $S_\rho/\rho(\Gamma)$ is then a compact Riemannian manifold whose sectional curvature we denote by $K_{S_\rho/\rho(\Gamma)}$. Theorem \ref{theoreme:borne_sup} proves that for every real $\alpha > 0$, there exists a constant $C_\alpha$ such that for any maximal representation $\rho : \pi_1(\Sigma_g) \to \SO_0(2,3)$, we have
\[\int_{S/\rho(\Gamma)} \abs{K_{S/\rho(\Gamma)}}^\alpha \leq C_\alpha\, g. \]

Let us turn back to representations. In \cite{Got01}, Gothen discovered unexpected connected components of the space of maximal representations into $\SO_0(2,3)$. On the latter, called Gothen components, the volume $\Vol$ does never vanish. These components are classified by an integer topological invariant denoted by $\deg(\rho)$, satisfying $0 < \deg(\rho) \leq 4g-4$. 
 
\noindent\textbf{Theorem \ref{theoreme:borne_inf}. }\textit{There exists a constant $D > 0$ such that for every integer $g \geq 2$ and every Gothen representation $\rho$ from $\pi_1(\Sigma_g)$ into $\SO_0(2,3)$, we have}
\[\Vol(\rho) \geq D\, \deg(\rho). \]

Hitchin representations, defined by $\deg(\rho) = 4g-4$, thus satisfy bounds of the type $D g \leq \Vol(\rho) \leq C g$. The bounds of Theorems \ref{théorème:borne_sup_cocompacte} and \ref{theoreme:borne_inf} are therefore optimal in $g$.

\textbf{Strategy of the proof of Theorem \ref{théorème:borne_sup_cocompacte}.}
By the work of \cite{CTT19}, to each maximal representation $\rho$ of $\pi_1(\Sigma)$ into $\SO(2,3)$ is associated a maximal surface $S_\rho$ in $\H^{2,2}$ (see Section \ref{subsection:représentations_maximales}). The convex hull of this surface is equal to the convex  $\cal{C}_\rho$ used to define $\Vol(\rho)$. This allows reducing the study of the volume $\Vol(\rho)$ to a problem of analytical nature on $S_\rho$ to which Theorem \ref{theoreme:borne_sup} gives an answer, as follows. First, we define functions on $S_\rho$ reflecting its geometry and satisfying elliptic differential equations (Section \ref{subsection:décomposition}). Using elliptic estimates, we then show the rapid decay of these functions when the curvature is small (\ref{subsection:décroissance_exponentielle}). We define a domain $D$ where the sectional curvature is far from $0$. This definition allows controlling the volume of a slice of $\cal{C}_\rho$ above a point of $S_\rho$, by a function of the distance to $D$ (Section \ref{subsection:contrôle_volume}). We then integrate this estimate (Section \ref{subsection:contrôle_Dt} and appendix \ref{annexe:contrôle_Dt}). We thus prove Theorem \ref{theoreme:borne_sup}, hence Theorem \ref{théorème:borne_sup_cocompacte} (Section \ref{subsection:contrôle_Dt}). 

\textbf{Organisation.}
We recall the classical results and techniques in Section \ref{section:préliminaires}. Section \ref{section:borne_inférieure} is dedicated to the proof of Theorem \ref{theoreme:borne_inf} : one finds there in germ the main ideas of the proof of Theorem \ref{theoreme:borne_sup}. In Section \ref{section:décroissance_exponentielle}, we prove the exponential decay of functions depending only on the germ of the considered surface. We then infer global results in Section \ref{section:borne_supérieure} (Theorem \ref{theoreme:borne_sup}, from which we deduce Theorem \ref{théorème:borne_sup_cocompacte}). 

\textbf{Acknowledgements.} I would like to thank my advisors, Francesco Bonsante and Jérémy Toulisse, for their patience, help and support. I would also like to thank Enrico Trebeschi, Alex Moriani, Andrea Tamburelli, Gabriele Viaggi and Andrea Seppi for certain discussions and the interest they showed for this work.

\section{Preliminaries}
\label{section:préliminaires}

\subsection{Pseudo-hyperbolic space}
\label{subsection:définition_Hpq}

A pseudo-Riemannian metric on a smooth manifold $M$ is a section of the bundle $\Sym^2(\Tan^* M)$ of symmetric bilinear forms on $M$, nowhere degenerate and of constant signature. It allows defining, as for a Riemannian metric, the Levi-Civita connection $\nabla^M$ : it is the unique connection $\nabla$ which is metric ($\nabla g = 0$) and torsion-free. We denote by $\courb \in \Omega^2(\Hom(\Tan M)$ its curvature tensor, defined by $\courb(a,b)X = \nabla_a \nabla_b X - \nabla_b \nabla_a X - \nabla_{[a,b]} X$. Given a non-degenerate plane $P$ for which we choose an orthonormal basis $(e_1,e_2)$, we define the sectional curvature $\kappa(P) = \frac{\scal{\courb(e_1,e_2)e_2}{e_1}}{\scal{e_1}{e_1}\scal{e_2}{e_2}}$. If $M$ is oriented, we can define a volume form $\vol^M$, which is equal to $1$ on any direct orthonormal basis. The corresponding measure and volume functional, denoted by $\mu_M$ and $\Vol^M$, are defined independently of the orientation. 

\textbf{Pseudo-hyperbolic space.}
The study of pseudo-hyperbolic spaces was initiated by the pioneering work of Mess \cite{Mess90} linking geometry of $\H^{2,1}$ and Teichmüller theory (in this direction, see also \cite{BS10} and \cite{BS20}). Those spaces allow interpreting geometrically a large class of representations into $\PO(p,q+1)$, called $\H^{p,q}$-convex-cocompact (see \cite{Mess90}, \cite{Barbot13}, \cite{BM12} and \cite{DGK18}). 

Let henceforth be integers $p > 0$ and $q \geq 0$. We construct the pseudo-hyperbolic space of signature $(p,q)$, which generalises the usual hyperbolic space (obtained with $q=0$) and the anti-de Sitter space (obtained with $q=1$). Let us denote by $\R^{p,q+1}$ the oriented vector space $\R^{p+q+1}$ endowed with the quadratic form of signature $(p,q+1)$ defined by
\[\quadra(x_1,\dots,x_{p+q+1}) = x_1^2 + \dots + x_p^2 - x_{p+1}^2 - \dots - x_{p+q+1}^2. \]
We denote by $(x,y) \mapsto \scal{x}{y}$ the associated bilinear form, which defines a pseudo-Riemannian metric on $\R^{p,q+1}$, and by $\O(p,q+1)$ its orthogonal group. The group $\PO(p,q+1) = \O(p,q+1)/\{\pm \Id\}$ acts on the projective space $\P(\R^{p,q+1})$ preserving an open set, the \emph{pseudo-hyperbolic space}, defined by
\[\H^{p,q} = \{[x] \in \P(\R^{p,q+1}) \mid \quadra(x) < 0\}. \]
Let us define a double covering of $\H^{p,q}$ by
\[\Hc^{p,q} = \{x \in \R^{p,q+1} \mid \quadra(x) = -1\}. \]
We have $\Hc^{p,q}/\{\pm \Id\} = \H^{p,q}$. A tangent space $\Tan_x \Hc^{p,q}$ is identified with $\Vect{x}^\perp$, thus inherits the metric of $\R^{p,q+1}$. This endows $\Hc^{p,q}$ with a metric of signature $(p,q)$ which descends to a metric on $\H^{p,q}$. The metrics of $\Hc^{p,q}$ and $\H^{p,q}$ have constant curvature $-1$ and respective isometry groups $\O(p,q+1)$ and $\PO(p,q+1)$. Their Levi-Civita connections are geodesically complete, the geodesics of $\H^{p,q}$ being its intersections with projective lines. 

\textbf{Connection.}
The submanifold $\Hc^{p,q} \hookrightarrow \R^{p,q+1}$ is umbilic : given a vector $\sf{u}$ in $\Tan_x \H^{p,q}$ and a vector field $\sf{V}$ on $\Hc^{p,q}$, denoting by $\nabla^\R$ and $\nabla^\H$ the connections of $\R^{p,q+1}$ and $\Hc^{p,q}$ respectively, we have
\begin{equation}\label{éq:Hpq_ombilique}
	\nabla^\R_{\sf{u}} \sf{V} = \nabla^\H_{\sf{u}} \sf{V} + \scal{\sf{u}}{\sf{V}} x. 
\end{equation}

\textbf{Geodesics.}
Consider a maximal geodesic of $\H^{p,q}$. Its image is the intersection of $\H^{p,q}$ with a projective line. More precisely, it can be lifted to a complete geodesic of $\Hc^{p,q}$, i.e. of the form $\gamma : \R \to \Hc^{p,q}$. If $\dot{\gamma}(0) = \sf{v} \in \Tan_x \Hc^{p,q}$, then
\begin{itemize}
	\item $\gamma(t) = \cosh(t) x + \sinh(t) \sf{v}$ if $\scal{\sf{v}}{\sf{v}} > 0$ ($\gamma$ is said to be spacelike) ;
	\item $\gamma(t) = x + t \sf{v}$ if $\scal{\sf{v}}{\sf{v}} = 0$ ($\gamma$ is said to be lightlike) ;
	\item $\gamma(t) = \cos(t) x + \sin(t) \sf{v}$ if $\scal{\sf{v}}{\sf{v}} < 0$ ($\gamma$ is said to be timelike). 
\end{itemize}
Two distinct points $u$ and $v$ in $\Hc^{p,q}$ are joined by a geodesic if, and only if, $\scal{u}{v} < 1$. If it exists, this geodesic is unique. 
\begin{itemize}
	\item If $\scal{u}{v} < -1$, it is spacelike, of length $\cosh^{-1}(-{\scal{u}{v}})$. 
	\item If $\scal{u}{v} = -1$, it is lightlike. 
	\item If $-1 < \scal{u}{v} < 1$, it is timelike, of length $\cos^{-1} {\scal{u}{v}}$. 
\end{itemize}
We deduce that two distinct points $u,v \in \H^{p,q}$ are joined by a geodesic, unique up to reparametrisation.

\subsection{Pseudo-Riemannian geometry}
\label{subsection:geo_diff_Hpq}

Let $M$ be a pseudo-Riemannian manifold of signature $(p,q)$ and $\nabla$ its Levi-Civita connection. Let $S$ be a submanifold of $M$ whose induced metric is non-degenerate. Let us consider the restriction of the affine bundle $(\Tan M,\nabla)$ to $S$ : it decomposes as the orthogonal sum $\Tan S \oplus \Nor S$. The second fundamental form of $S$ is a bundle morphism $\II : \Tan S \otimes \Tan S \to \Nor S$, defined by $\II(\sf{U},\sf{V}) = (\nabla_{\sf{U}} \sf{V})^{\Nor S}$ (normal component of $\nabla_{\sf{U}} \sf{V}$). The shape operator is a bundle morphism $B : \Tan S \otimes \Nor S \to \Tan S$, defined by $B(\sf{U},\sf{N}) = (\nabla_{\sf{U}} \sf{N})^{\Tan S}$ (tangent component of $\nabla_{\sf{U}} \sf{N}$). Therefore, we have $\scal{\II(\sf{U},\sf{V})}{\sf{N}} + \scal{\sf{V}}{B(\sf{U},\sf{N})} = 0$ and the restriction of $\nabla$ to $(\Tan M)_{|S}$ decomposes as
\[\nabla = 
\begin{pmatrix}
	\nabla^{\Tan} & -B\\
	\II & \nabla^{\Nor}
\end{pmatrix}. \]

$S$ is said to be spacelike if its induced metric is Riemannian. Let us now suppose that $S$ is spacelike, of dimension $p$. Let $x$ be a point of $S$ and $(e_i)$ ($1 \leq i \leq p$) an orthonormal frame of $\Tan_x S$. We define the norm $\norm{\II}_2$ and the mean curvature $H$ of $S$ by
\begin{align*}
	\norm{\II}_2^2 &= - \sum_{i,j} \scal{\II(e_i,e_j)}{\II(e_i,e_j)} \tag*{and} \\
	H &= \frac{1}{p} \sum_i \II(e_i,e_i). 
\end{align*}
$S$ is said to be maximal if $H = 0$. If $M$ is Riemannian, this corresponds to the classical notion of minimal submanifold. Note that the metrics of $(\Tan S)^2$ and $\Nor S$ allow defining the adjoint $\II^* : \Nor S \to \Tan S \otimes \Tan S$ and that we have $\norm{\II}_2^2 = - \tr [\II^* \II]$. 

\textbf{Sasaki metric.} 
Given a pseudo-Riemannian manifold $(M,g)$, let us construct a metric $g^{\Tan M}$ on the tangent space $\Tan M$ seen as a manifold, called the Sasaki metric, which will be used in Section \ref{subsection:Chern-Weil} and appendix \ref{annexe:nu}. We have two distributions on the total space $\Tan M$ of the bundle $\pi : \Tan M \to M$ : the vertical distribution $\bf{vert} = \ker(\d\pi)$ and the horizontal distribution $\bf{hor}$ given by the Levi-Civita connection (seen as an Ehresmann connection). At a point $v_x$ pf $\Tan_x M$, we define $g^{\Tan M}_{v_x}$ as follows. 

Given $\sf{u} \in \Tan_x M$, let us denote by $t \mapsto \psi_{\sf{u},v_x}(t)$ the path in $\Tan M$ defined by parallel transporting $v_x$ in $\Tan M$ along the path $t \mapsto \exp^M_x(t \sf{u})$ (defined for $t$ close to $0$). We define the Sasaki metric $g^{\Tan M}$ by imposing that $\bf{vert}_{v_x}$ and $\bf{hor}_{v_x}$ are orthogonal ($g^{\Tan M}(\sf{v},\sf{h}) = 0$ for $\sf{v} \in \bf{vert}_{v_x}$, $\sf{h} \in \bf{hor}_{v_x}$), and that the following maps are isometries:
\[\begin{array}{rrcl}
	\iota_{\bf{hor}} : & \Tan_x M & \to & \bf{hor}_{v_x} \\
	\ & \sf{u} & \mapsto & \dtzero \psi_{\sf{u},v_x}(t),
\end{array}\]
\[\begin{array}{rrcl}
	\iota_{\bf{vert}} : & \Tan_x M & \to & \bf{vert}_{v_x} \\
	\ & \sf{u} & \mapsto & \dtzero v_x+t\sf{u}.
\end{array}\]
We deduce the following lemma, which justifies the use of the Sasaki metric in this article.
\begin{lemme}\label{lemme:Sasaki_par_parties}
	Let $(M,g)$ be a pseudo-Riemannian manifold and $S$ a submanifold of $M$ in restriction to which the metric $g$ is non-degenerate. Let us restrict the Sasaki metric to the normal bundle $\Nor S$. Let $E$ be a measurable subset of $\Nor S$. We have
	\[\Vol^{\Nor S}(E) = \int_S \left( \Vol^{\Nor_x S}(E \cap \Nor_x S) \right) \d\mu_S(x). \]
\end{lemme}
\begin{proof}
	Let us choose a simply connected open set $U$ of $S$ and orientations of $U$ and of the bundle $\Nor U = (\Nor S)_{|U}$. The volume form $\vol^{\Nor U}$ of the total space $\Nor U$ satisfies
	\[\vol^{\Nor U} = \pi^*(\vol^S) \wedge \phi, \]
	where the restriction of the differential form $\phi$ to a fibre $\Nor_x U$ is the volume form $\vol^{\Nor_x U}$ of this fibre. We have, for $\chi$ the indicator function of $E$ :
	\[\int_{\Nor U} \chi\ \d\!\vol^{\Nor U} = \int_U \left( \int_{\Nor_x U} \chi \d\!\vol^{\Nor_x S} \right) \d\!\vol^S(x), \]
	whence the result if $E$ is contained in $\Nor U$, thus in general. 
\end{proof}

\subsection{Complete submanifolds of pseudo-hyperbolic space}
\label{subsection:geom_projective_hpq}

We shall use the following lemmas, which describe the complete maximal surfaces of $\H^{2,q}$. We refer to \cite{CTT19} for many proofs. 

\begin{proposition}[see {\cite[Proposition 3.5]{CTT19}}]\label{proposition:coordonnées_Fermi}
	Let $\R^{2,q+1} = E \oplus F$ be an orthogonal decomposition of $\R^{2,q+1}$ into two subspaces, $E$ being positive definite and $F$ negative definite. Let us denote by $\D$ the open unit ball of $E$ and by $\S^q$ the unit sphere of $F$. Denoting by $g_E$ the metric of $E$, we endow $\D$ with the metric $(g_{\D})_u = \left( \frac{2}{1+\norm{u}^2} \right)^2 (g_E)_u$ and $\S^q$ with the spherical metric $g_\S$ induced by the metric of $F$. The map
\[\begin{array}{rrcl}
	\psi : & \D \times \S^q & \longrightarrow & \Hc^{2,q} \\
	\ & (u,v) & \longmapsto & \frac{2}{1-\norm{u}^2} u + \frac{1+\norm{u}^2}{1-\norm{u}^2} v
\end{array}\]
is a diffeomorphism and we have
\[\psi^* g_{\H^{2,q}} = \left( \frac{1+\norm{u}^2}{1-\norm{u}^2} \right)^2 (g_\D \oplus - g_{\S}). \]
\end{proposition}
\begin{proposition}[see {\cite[Proposition 3.8]{CTT19}}]\label{proposition:complète_implique_entière}
	Let $S$ be a spacelike and complete surface of $\H^{2,q}$. With the notations of Proposition \ref{proposition:coordonnées_Fermi}, $S$ is of the form
	\[S = \{\psi(\xi,\phi(\xi)) \mid \xi \in \D\} \]
	with $\phi : \D \to \S^q$ a strictly $1$-Lipschitz map. 
\end{proposition}

\begin{corollaire}\label{corollaire:sous-variété_simplement_connexe}
	Let $S$ be a surface spacelike and complete in $\Hc^{2,q}$. $S$ is simply connected. Moreover, for all points $x,y \in S$, we have
	\[d(x,y) \geq \cosh(-\scal{x}{y}). \]
\end{corollaire}
\begin{proof}
	$S$ is simply connected by Proposition \ref{proposition:complète_implique_entière}. For the inequality, see \cite[Lemma 3.8]{CTT19}. 
\end{proof}

The following technical lemma will be useful in Section \ref{subsection:contrôle_volume}. We recall (Section \ref{subsection:définition_Hpq}) that the type of the geodesic joining two points of $\Hc^{2,q}$ is deduced from their scalar product. 
\begin{lemme}\label{lemme:cône_1_Lipschitz}
	Let $a$ be a point in $\Hc^{2,q}$, $\Omega = \{y \in \Hc^{2,q} \mid \scal{a}{y} < -1\}$ be the set of points joined to $a$ by a spacelike geodesic, $S$ be a spacelike and complete surface of $\Hc^{2,q}$ and $x$ be a point in $S$ joined to $a$ by a timelike geodesic (that is $\abs{\scal{a}{x}} < 1$). Let us denote by $\d$ the distance on $S$ induced by its Riemannian metric. Let $t > 0$ be a real number such that for every point $y$ in $S$ satisfying $\d(x,y) = t$, we have $y \in \Omega$. For every $y$ in $S$ satisfying $\d(x,y) \geq t$, we have $y \in \Omega$. 
\end{lemme}
\begin{proof}
	Proposition \ref{proposition:complète_implique_entière} identifies $\Hc^{2,q}$ with $\D \times \S^q$ in such a way that the surface $S$ is the graph of a $1$-Lipschitz map $\phi : \D \to \S^q$. Let us choose this identification in such a way that $a = (0,\sigma_a)$ and $x = (0,\sigma_x)$ (for points $\sigma_a, \sigma_x \in \S^q$). Denote by $\d^\D$ and $\d^\S$ the distances induced by $g_\D$ and $g_\S$ respectively. Since $\Omega$ is the set of points joined to $a$ by a spacelike geodesic, we have
	\[\Omega = \{(u,\sigma) \mid \d^\S(\sigma_a,\sigma) < \d^{\D}(0,u)\}. \]
	Let $y$ be a point in $S$ satisfying $\d(x,y) \geq t$ : it is of the form $y = (\xi,\phi(\xi))$ for some $\xi$ in $\D$. We define a path $\gamma : [0;1] \to S$ by $\gamma(s) = (s\xi,\phi(s\xi))$. Since $S$ is complete, there exists a real number $s$, $0 < s \leq 1$, such that $\d(\gamma(s),x) = t$. By hypothesis, we have $\gamma(s) \in \Omega$, that is $\d^\S(\phi(s\xi),\sigma_a) < \d^\D(0,s \xi)$. Since $\phi$ is $1$-Lipschitz, we deduce that $\d^\S(\phi(\xi),\sigma) < \d^\D(0,\xi)$, hence $y$ is in $\Omega$. 
\end{proof}
\begin{lemme}[see {\cite[Proposition 3.8]{CTT19}} and {\cite[remark 3.3]{CTT19}}]\label{lemme:relèvement_unique}
	Let $S$ be a spacelike and complete surface of $\H^{2,q}$. The inverse image of $S$ under the projection $\Hc^{2,q} \to \H^{2,q}$ is the disjoint union of two submanifolds $\hat{S}$ and $-\hat{S}$. The convex hull $\Conv(\hat{S})$ of $\hat{S}$ in $\R^{2,q+1}$ does not contain $0$. 
\end{lemme}

\begin{definition}[convex hull]\label{définition:enveloppe_convexe}
	Let $S$ be a spacelike type and complete surface of $\H^{2,q}$. Let $\hat{S}$ be a lift of $S$ to $\Hc^{2,q}$ (see Lemma \ref{lemme:relèvement_unique}). We define the \emph{convex hull} $\Conv(S) = \P \Conv(\hat{S})$. 
\end{definition}
Lemma \ref{lemme:relèvement_unique} shows that the convex hull is well defined and independent of the choice of lift. 

\begin{lemme}[see {\cite[Lemma 3.11]{CTT19}}]\label{lemme:épaisseur_pi2}
	Let $S$ be a spacelike and complete surface of $\H^{2,q}$ and $x$ a point in $S$. The convex hull $\Conv(S)$ does not intersect the hyperplane $\Vect{x}^\perp$. 
\end{lemme}
\begin{proof}
	Let $\hat{S} \subset \Hc^{2,q}$ be a lift of $S$ and let $\hat{x}, \hat{y} \in \hat{S}$. We see from Propositions \ref{proposition:complète_implique_entière} and \ref{proposition:coordonnées_Fermi} that $\hat{x}$ and $\hat{y}$ are connected by a spacelike geodesic, hence satisfy $\scal{\hat x}{\hat y} \leq -1$ (see \cite[Lemma 3.7]{CTT19}). A convex combination of such points $\hat{y}$ therefore cannot be orthogonal to $\hat{x}$ : $\Conv(\hat{S})$ does not intersect $\Vect{\hat{x}}^\perp$. 
\end{proof}

We establish another two lemmas in signature $(p,q)$. They will be used in the proof of Proposition \ref{proposition:Seppi}. 
\begin{lemme}\label{lemme:u_harmonique}
	Let $S$ be a maximal $p$-dimensional submanifold of $\Hc^{p,q}$, whose Laplacian we denote by $\Delta$. Choose a point $h$ in $\R^{p,q+1}$ and define a function $u : S \to \R$ by $u(x) = \scal{h}{x}$. We have $(\Delta - p) u = 0$. 
\end{lemme}
\begin{proof}
	Define a function $U : \Hc^{p,q} \to \R,\ x \mapsto \scal{h}{x}$. We denote by $\grad^\H$ and $\grad^S$ the respective gradients of $\Hc^{p,q}$ and $S$. Define a section $\sf{n}$ of $\Nor S$ as the orthogonal projection of $\grad^{\H} U$ onto $\Nor S$ : we thus have $\sf{n} = \grad^{\H} U - \grad^{S} u$. Choose an orthogonal frame $(e_1,\dots,e_p)$ of $\Tan_x S$ and extend it to vector fields satisfying $\nabla e_k = 0$ at $x$. At the point $x$, we have
	\begin{align*}
		\grad^S\, u &= \grad^\H\, U - \sf{n} \\
		&= h + \scal{h}{x} x - \sf{n}, \tag*{hence} \\
		(\Hess u)(e_i,e_j) &= e_i \cdot (e_j \cdot u) \\
		&= e_i \cdot \scal{h + \scal{h}{x} x - \sf{n}}{e_j} \\
		&= \scal{h}{x} \scal{e_i}{e_j} - \scal{\nabla_{e_i} \sf{n}}{e_j} \\
		&= u(x)\, \scal{e_i}{e_j} - \scal{B(e_i,\sf{n})}{e_j}. 
	\end{align*}
	$S$ being maximal, we have $\sum_k \scal{B(e_i,\sf{n})}{e_i} = 0$ hence $\Delta u = \tr(\Hess  u) = pu$. 
\end{proof}

\begin{lemme}\label{lemme:surface_dans_hyperplan}
	Let $S$ be a maximal complete pointed $p$-dimensional submanifold of $\H^{p,q}$. If there exists a point $x$ in $S \cap \partial \Conv(S)$, then $S$ is contained in a hyperplane of $\R^{p,q+1}$ of signature $(p,q)$. 
\end{lemme}
\begin{proof}
	Under this hypothesis, we have a supporting hyperplane of $\Conv(S)$ at $x$, denoted by $H$. Since $H$ contains $\Tan_x S$, it has signature $(p,q)$. Lift everything to $\Hc^{p,q}$ : we thus obtain a point $h$ in $\R^{p,q+1}$ such that $H = \{h\}^\perp$, hence the function $u = \scal{h}{\cdot} : S \to \R$ satisfies $u \geq 0$. We have $(\Delta^S - p) u = 0$ by Lemma \ref{lemme:u_harmonique} and $u(x) = 0$, hence $u = 0$ (see \cite{CTT19}, Lemma A.3), in other words $S \subset H$. 
\end{proof}

\subsection{Complexification of a real vector bundle of rank two}
\label{subsection:fibrés_holomorphes}

This section describes how to decompose the complexification of a real vector bundle over a Riemann surface, of rank $2$ and endowed with a metric and an orientation. These results will be used in section \ref{subsection:décomposition}. 

\begin{definition}\label{définition:complexification}
	Let $E \to S$ be a vector bundle over a Riemann surface. Suppose $E$ is endowed with a metric $g$. Endow the bundle $\cal{E} = E \otimes_\R \C$, called complexification of $E$, with the $\C$-linear extension of $g$ and with the $\C$-antilinear involution equal to the identity on $E$, denoted by $x \mapsto \bar{x}$. We define a hermitian metric $h$ on $\cal{E}$ by
	\[h (\sf{u},\sf{v}) = g(\sf{u},\bar{\sf{v}}), \]
	as well as the norm $\norm{x} = \sqrt{h(x,x)}$. If moreover $E$ is of rank $2$ and endowed with an almost complex structure $J \in \O(E)$, define the volume form $\vol^E$ of $\cal{E}$ by
	\[\vol^E(\sf{u},\sf{v}) = g(J \sf{u},\sf{v}). \]
\end{definition}

\begin{fait}\label{fait:involutions}
	If $E,F$ are real bundles, the involutions of the complexifications of $E \otimes F$ and $E^*$ are given, for $x \in E^\C$, $y \in F^\C$ and $f \in (E^*)^\C$, by
	\begin{align*}
		\bar{x \otimes y} &= \bar{x} \otimes \bar{y} \tag*{and} \\
		\bar{f(x)} &= \bar{f}(\bar{x}). 
	\end{align*}
\end{fait}

Suppose now that $E$ is endowed with a metric connection $\nabla$. The $\C$-linear extension of $\nabla$ to $\cal{E}$ is still a metric connection, whose $(0,1)$ part, denoted by $\nabla^{(0,1)}$, defines a structure of holomorphic vector bundle on $\cal{E}$. Suppose moreover that $E$ is endowed with an almost complex structure $J \in \O(E)$ : this is equivalent to the data of a principal $\SO(2,\R)$-bundle $P$, such that $E$ is the associated bundle $P[\R^2]$. Let $\nabla$ be a metric connection on $E$ : this data is equivalent to that of a principal connection on $P$. 

\begin{lemme}[decomposition of an oriented metric vector bundle of rank $2$]\label{lemme:décomposition_fibré_rang_2}
	We have a canonical decomposition $\cal{E} = \ker(J - i \Id) \oplus \ker(J + i \Id)$, where $\ker(J - i \Id) = \cal{L}$ and $\ker(J + i \Id)$ are holomorphic subbundles of $\cal{E}$ orthogonal to each other and preserved by $\nabla$. Endow these bundles with the metrics and connections inherited from $\cal{E}$. $\cal{L} \otimes \ker(J + i \Id)$ is trivial as a hermitian holomorphic vector bundle endowed with a metric connection (hence an identification of hermitian holomorphic vector bundles endowed with a metric connection between $\cal{L}^{-1}$ and $\ker(J + i \Id)$). 
\end{lemme}
\begin{proof}
	All the geometric data of $E = P[\R^2]$ and $\cal{E} = P[\C^2]$ depend on the action of $\SO(2,\R)$ on $\R^2$ and $\C^2$ respectively : the lemma thus follows from the properties of this action. 
	The canonical orientation of $\R^2$ is defined by $j = \left( \begin{smallmatrix} 0 & -1 \\ 1 & 0 \end{smallmatrix} \right) \in \frak{so}(2,\R)$. Notice that $j$ is diagonalisable over $\C$ and its eigenspaces $\ker(j \pm i \Id)$ are the isotropic lines of $\C^2$ for the canonical quadratic form, denoted by $L^\pm$. The action of $e^{i \theta} \in \SO(2,\R)$ on $\C^2$ preserves each of these two lines, acting on $\ker(j \pm i \Id)$ by multiplication by $e^{\pm i \theta}$. We deduce that the cocycles defining $E$ preserve two line subbundles of $\cal{E} = E \otimes_\R \C$ defined by $\ker(J \pm i \Id) = \cal{L}^\pm$, and that the cocycle defining $\cal{L}^+ \otimes \cal{L}^-$ is trivial, hence the identification $\cal{L}^{-1} \approx \ker(J+i \Id)$ as holomorphic vector bundles. 
	
	The lines $L^\pm$ are orthogonal for the canonical hermitian metric of $\C^2$ : the corresponding subbundles are therefore orthogonal for the hermitian metric on $\cal{E}$ deduced from $g$. Similarly, a principal connection on $P$ takes values in $\frak{so}(2,\R)$, which is generated by $j$, whose action on $\C^2$ preserves each of the lines $\ker(j \pm i \Id)$. In particular, the $\ker(J \pm i \Id)$ are preserved by $\nabla$ thus by $\nabla^{(0,1)}$ : these are holomorphic subbundles. As $tj$ acts by $\pm i t$ on $\ker(j \pm i \Id)$, passing to the associated bundles, we see that the connection $\nabla$ preserves the $\ker(J \pm i \Id) = \cal{L}^{\pm 1}$ and that the connection $\cal{L} \otimes \cal{L}^{-1}$ is trivial. 
\end{proof}

\begin{remarque}\label{remarque:changement orientation}
	The other possible orientation on the metric bundle $E$ is $-J$. The previous construction, carried out with $-J$, would exchange the bundles $\cal{L}$ and $\cal{L}^{-1}$ of the decomposition into lines. 
\end{remarque}

The decomposition $(\Tan U)^\C = \cal{T} \oplus \cal{T}^{-1}$ is classical : $\cal{T}$ is often called the holomorphic tangent bundle and denoted by $\Tan^{(1,0)} S$. Similarly, the complexification of $\Tan^* S$ (oriented by $-\transp{J_{\Tan}}$) decomposes as $\canon \oplus \canon^{-1}$, where $\canon$ is the dual of $\cal{T}$, called the canonical bundle of $S$. A tensor $\phi : (\Tan S)^{\otimes d} \to \C$ thus extends $\C$-linearly into a section of $(\canon \oplus \canon^{-1})^{\otimes d} \cong \bigoplus_{p+q=d} (\canon^{p-q})^{\oplus \binom{d}{p}}$, the component corresponding to $(p,q)$ being denoted by $\phi^{(p,q)}$. The involution $x \mapsto \bar{x}$ of $(\Tan S)^\C$ sends $\canon$ onto $\canon^{-1}$. 

Let $(\cal{L},h)$ be a holomorphic hermitian line bundle over a Riemann surface $S$. We denote by $\nabla$ its Chern connection and $\courb \in \Omega^2(S,\End(\cal{L})) = \Omega^2(S,\C)$ its curvature. Let moreover $s$ be a holomorphic nowhere vanishing section of $\cal{L})$ : this defines a function $u$ on $S$ by $e^u = h(s,s)$. 
\begin{lemme}\label{lemme:courbure_fibré_droites}
	Let $z = x+iy$ be a local holomorphic coordinate and define $\partial_z = \frac12 (\partial_x - i\partial_y)$, $\partial_{\bar z} = \frac12 (\partial_x + i\partial_y)$. With the previous notations, we have $\nabla_{\partial_z} s = u_z\, s$ and $\courb(\partial_z, \partial_{\bar{z}}) = -u_{z \bar{z}}\ \Id_\cal{L}$. In particular, if $S$ is endowed with a metric (inducing the complex structure) defining a Laplacian $\Delta$ and a volume form $\omega$, then $\courb = \frac{i}{2} \Delta u\ \Id_{\cal{L}} \otimes \omega$. 
\end{lemme}
\begin{proof}
	Let us define (locally) $\alpha$ by $\nabla_{\partial_z} s = \alpha\ s$. Since $\nabla_\sf{v} h = 0$ for $\sf{v}$ in $\Tan S$, we have for every complex tangent vector $\sf{v}$ in $\Tan^\C S$ : $\sf{v} \cdot h(X,Y) = h(\nabla_\sf{v} X,Y) + h(X,\nabla_{\bar{\sf v}} Y)$. We deduce $\partial_z \cdot e^u = h(\nabla_{\partial_z} s,s) + h(s,\nabla_{\partial_{\bar{z}}} s) = \alpha e^u$ since $\nabla_{\partial_{\bar{z}}} s = 0$, hence $\alpha = \partial_z \cdot u$ and the desired expression for $\nabla_{\partial_z} s$. Finally, 
	\begin{align*}
		\courb(\partial_z,\partial_{\bar{z}}) s &= \nabla_{\partial_z} \nabla_{\partial_{\bar{z}}} s - \nabla_{\partial_{\bar{z}}} \nabla_{\partial_z} s - \nabla_{[\partial_z,\partial_{\bar{z}}]} s \\
		&= -\nabla_{\partial_{\bar{z}}} \nabla_{\partial_z} s \tag{since $s$ is holomorphic} \\
		&= -\nabla_{\partial_{\bar{z}}} (u_z\ s) \\
		&= -u_{z \bar{z}}\ s.  
	\end{align*}
	Let us assume $S$ endowed with a metric. The Laplacian of the flat metric $\abs{\d z^2}$ is given by $\Delta_0 u = u_{xx} + u_{yy} = 4 u_{z \bar{z}}$, thus the Laplacian $\Delta$ is given by $\Delta u = \frac12 \norm{\partial_z}^{-2} \Delta_0 u = 2 \norm{\partial_z}^{-2} u_{z \bar{z}}$. We deduce the equality $\courb = \frac{i}{2} \Delta u\ \Id_{\cal{L}} \omega$. 
\end{proof}

\subsection{Gau\ss, Ricci and Codazzi equations}
\label{subsection:équations_fondamentales}

We establish three fundamental equations relating the connection of a spacelike submanifold $S$ of $\H^{p,q}$ to the ambient connection, which will be useful in Section \ref{subsection:décomposition}. The sectional curvature of $\H^{p,q}$ is constant, equal to $-1$, from which we classically deduce the expression of its curvature $\courb$ : $\courb(\sf{u},\sf{v})X = \scal{\sf{u}}{X} \sf{v} - \scal{\sf{v}}{X} \sf{u}$ (see \cite{Petersen16}). In particular, for all $\sf{u},\sf{v} \in \Tan_x S$ and $X,Y \in \Tan_x \H^{p,q}$, we obtain
\begin{equation*}
	\scal{\sf{u}}{X} \scal{\sf{v}}{Y} - \scal{\sf{v}}{X} \scal{\sf{u}}{Y} = \scal{\nabla_\sf{u} \nabla_\sf{v} X - \nabla_\sf{v} \nabla_\sf{u} X - \nabla_{[\sf{u},\sf{v}]} X}{Y}. 
\end{equation*}
Using the decomposition of $\nabla$ (section \ref{subsection:geo_diff_Hpq}) and taking each of the vectors $X$ and $Y$ in $\Tan S$ or $\Nor S$, the previous equation simplifies using the curvature $\courb^{\Nor}$ of the normal bundle and the curvature of the tangent bundle $\Tan S$, equivalent to the sectional curvature $K$ of $S$. When $\sf{u} = Y$ and $\sf{v} = X$ form an orthonormal basis of a plane $P$ in some fiber $\Tan_x S$, we obtain the Gau\ss\ equation :
\begin{align}
	-1 &= K(P) - \scal{\II(\sf{v},\sf{v})}{\II(\sf{u},\sf{u})} + \scal{\II(\sf{u},\sf{v})}{\II(\sf{v},\sf{u})}. \tag{Gau\ss} \label{éq:Gauss_0}
\end{align}
By considering $X,Y$ in the normal bundle $\Nor S$, we obtain the Ricci equation :
\begin{align}
	0 &= \scal{\courb^{\Nor}(\sf{u},\sf{v}) X}{Y} - \scal{B(\sf{v},X)}{B(\sf{u},Y)} + \scal{B(\sf{u},X)}{B(\sf{v},Y)}. \tag{Ricci} \label{éq:Ricci_0}
\end{align}
Finally, taking $X \in \Tan S$ and $Y \in \Nor S$ and denoting by $\nabla \II$ the derivative of $\II$ as a section of $(\Tan^* S)^{\otimes 2} \otimes \Nor S$ (where $\nabla$ is induced by the connections on $\Tan S$ and $\Nor S$), we obtain the Codazzi equation :
\begin{align}
	0 &= \scal{\II(\sf{u},\nabla_\sf{v}^{\Tan} X) + \nabla_\sf{u}^{\Nor} (\II(\sf{v},X)) - \II(\sf{v},\nabla_\sf{u}^{\Tan} X) - \nabla_\sf{v}^{\Nor}(\II(\sf{u},X)) - \II([\sf{u},\sf{v}],X)}{Y} \nonumber \\
	&= \scal{(\nabla_\sf{u} \II)(\sf{v},X) - (\nabla_\sf{v} \II)(\sf{u},X)}{Y}. \tag{Codazzi} \label{éq:Codazzi_0}
\end{align}

\textbf{Case of a maximal surface.}
Let us specialise these formulas by assuming that $S$ is a maximal surface in $\H^{2,2}$. Suppose the vector bundles $\Tan S$ and $\Nor S$ are oriented : this defines complex structure $J_{\Tan} \in \End(\Tan S)$ and $J_{\Nor} \in \End(\Nor S)$ (we shall apply the following results to a maximal complete surface in $\H^{2,2}$, in which case this latter hypothesis is indeed verified). The curvature $\courb^{\Tan}$ of the tangent bundle is a section of the bundle $\Lambda^2(\Tan^* S) \otimes \frak{so}(\Tan S)$. As the volume form $\omega$ of $S$ generates at every point the line bundle $\Lambda^2(\Tan S)$, we can define $\frac{\courb^{\Tan}}{\omega} \in \frak{so}(\Tan S)$ by the formula $\courb^{\Tan} = \omega \frac{\courb^{\Tan}}{\omega}$. We define similarly $\frac{\courb^{\Nor}}{\omega} \in \frak{so}(\Nor S)$. Extend $\II$ $\C$-linearly into a morphism from the complexification of $(\Tan S)^{\otimes 2}$ to the complexification of $\Nor S$. Denote by $\Tan^1 S$ the unit tangent bundle of $S$, fix a vector $\sf{u} \in \Tan^1_x S$ and define vectors in the complexification of $\Tan S$ by $\partial_z = \frac{1}{\sqrt 2} (\sf{u} - iJ_{\Tan} \sf{u})$ and $\partial_{\bar{z}} = \frac{1}{\sqrt 2} (\sf{u} + iJ_{\Tan} \sf{u})$. The mean curvature of $S$ (see Section \ref{subsection:geo_diff_Hpq}) is then given by
\[H = \frac12 [\II(\sf{u},\sf{u}) + \II(J_{\Tan} \sf{u},J_{\Tan} \sf{u})] = \II(\partial_z,\partial_{\bar{z}}). \]
In other words, the maximality of $S$ is equivalent to $\II^{(1,1)} = 0$. Consider the basis $(b_1,b_2) = (\partial_z,\partial_{\bar z})$ of the complexification of $\Tan S$ : it is orthonormal with respect to the hermitian metric deduced from the metric on $\Tan S$. The $b_i \otimes b_j$ ($1 \leq i,j \leq 2$) hence form an orthonormal basis of the complexification of $(\Tan S)^{\otimes 2}$, so that :
\begin{align*}
	\norm{\II}^2 &= -\tr (\II^* \II) \\
	&= \sum_{i,j} h(\II(b_i \otimes b_j),\II(b_i \otimes b_j)) \tag*{by definition of $\II^*$} \\
	&= \sum_{i=j} h(\II(b_i \otimes b_j),\II(b_i \otimes b_j)) \tag*{because $\II(\partial_z \otimes \partial_{\bar z}) = 2H = 0$} \\
	&= \norm{\II(\partial_z,\partial_z)}^2 + \norm{\II(\partial_{\bar z},\partial_{\bar z})}^2. 
\end{align*}
The $\C$-linear extension of $\II$ (still denoted by $\II$) comes from a morphism of real vector bundles, so it satisfies $\bar{\II} = \II$. By Fact \ref{fait:involutions}, we have
\begin{equation}\label{éq:II_symétrie_conjuguée}
	\II(\partial_{\bar{z}}^2) = \II(\bar{\partial_z^2}) = \bar{\II(\partial_z^2)}. 
\end{equation}
The previous equality thus rewrites as
\[\norm{\II}^2 = 2 \norm{\II(\partial_z,\partial_z)}^2. \]
By summing equation \eqref{éq:Gauss_0} for $\sf{u}$ and $\sf{v}$ in an orthonormal basis, we obtain a reformulation of the Gau\ss{} equation: 
\begin{align}
	K &= -1 + \frac{1}{2} \norm{\II}_2^2 \nonumber \\
	&= -1 + \norm{\II(\partial_z,\partial_z)}^2. \label{éq:Gauss}
\end{align}
To simplify the Codazzi equation, let us take $\sf{u} = X = \partial_z$ and $\sf{v} = \partial_{\bar{z}}$ in equation \eqref{éq:Codazzi_0}. Taking into account $\II^{(1,1)} = 0$ (thus $(\nabla \II)^{(1,1)} = 0$), we obtain that $\II^{(2,0)}$ satisfies
\begin{align}
	0 &= \nabla_{\partial_{\bar{z}}} \II^\C(\partial_z,\partial_z), \tag* {that is} \nonumber \\
	0 &= \bar{\partial} \II^{(2,0)}, \label{éq:Codazzi}
\end{align}
where $\bar{\partial}$ is the operator $\nabla^{(0,1)}$ on $\canon^2 \otimes (\Nor S)^\C$. Let $x$ be a point in $S$ and $(e_1,e_2)$ be an orthonormal basis of $\Tan_x S$. The metric of $(\Tan S)^{\otimes 2}$ restricts to the space of symmetric traceless $2$-tensors, denoted by $\Sym^2_0(\Tan S)$. An orthogonal basis of this space is given by
\[(f_1,f_2) = \left( \frac{1}{\sqrt{2}} (e_1 \otimes e_1 - e_2 \otimes e_2),\frac{1}{\sqrt{2}} (e_1 \otimes e_2 + e_2 \otimes e_1) \right). \]
We define this basis as direct, which orients $\Sym^2_0(\Tan S)$. The symmetry of $\II$ and the fact that $H=0$ imply that $\II$ vanishes on the orthogonal of $\Sym^2_0 (\Tan S)$ in $(\Tan S)^{\otimes 2}$. All the information of $\II$ is therefore contained in its restriction to $\Sym^2_0 (\Tan S)$, which we will denote by $\II_0$ or simply $\II$. $\Sym^2_0(\Tan S)$ and $\Nor S$ now have volume forms, which allows defining $\det(\II_0)$ by $\II_0^* \vol^{\Nor S} = \det(\II_0)\vol^{\Sym^2_0(\Tan S)}$. Recall that $\sf{u}$ is a unitary vector in $\Tan_x S$ defining vectors $\partial_z = \frac{1}{\sqrt 2} (\sf{u} - iJ_{\Tan} \sf{u})$ and $\partial_{\bar{z}} = \frac{1}{\sqrt 2} (\sf{u} + iJ_{\Tan} \sf{u})$. 
\begin{lemme}\label{lemme:Ricci}
	The Ricci formula can be rewritten as
	\begin{equation}\label{éq:Ricci}
		\frac{\courb^{\Nor}}{\omega} = (\det \II_0) J_{\Nor}.
	\end{equation}
	Moreover, we have $\det \II_0 = -i h(J \II(\partial_z,\partial_z),\II(\partial_z,\partial_z))$. 
\end{lemme}
\begin{remarque}
	In the equation \eqref{éq:Ricci}, everything in the left-hand term is independent of the choice of $J_{\Nor}$, whereas $\det \II$ changes sign when $J_{\Nor}$ is reversed. 
\end{remarque}
\begin{proof}
	We have $\II(f_1) = \sqrt{2} \II(e_1,e_1)$ and $\II(f_2) = \sqrt{2} \II(e_1,e_2)$ so according to \cite[Proposition 4.6]{LT22}, equation \eqref{éq:Ricci_0} can be rewritten as
	\[\scal{\frac{\courb^{\Nor}}{\omega}(\II(f_1))}{\II(f_2)} = -\norm{\II(f_1) \wedge \II(f_2)}^2. \]
	Since $\frac{\courb^{\Nor}}{\omega}$ is a section of $\frak{so}(\Nor S)$, we have $\frac{\courb^{\Nor}}{\omega} = \lambda J_{\Nor}$ for some function $\lambda$. This latter equality can be rewritten using the volume form $\vol^{\Nor S} = -\scal{\cdot}{J \cdot}$ :
	\begin{align*}
		\lambda \vol^{\Nor S}(\II(f_1),\II(f_2)) &= \norm{\II(f_1) \wedge \II(f_2)}^2 \\
		&= [\vol^{\Nor S}(\II(f_1),\II(f_2))]^2. 
	\end{align*}
	We thus obtain $\lambda = \vol^{\Nor S}(\II(f_1),\II(f_2))$. Since we also have $\vol(f_1,f_2) = 1$ by definition, this implies $\lambda = \det \II_0$. 
	
	Let us move to the expression of $\det \II_0$. We use the notations of Definition \ref{définition:complexification}. The vectors $\partial_z$ and $\partial_{\bar z}$ are unitary and satisfy $\bar{\partial_z} = \partial_{\bar{z}}$. The vectors $\partial_z^2 = \partial_z \otimes \partial_z$ and $\partial_{\bar{z}}^2 = \partial_{\bar{z}} \otimes \partial_{\bar{z}}$ still satisfy $\bar{\partial_z^2} = \partial_{\bar{z}}^2$ (fact \ref{fait:involutions}) and form a basis of the complexification of $\Sym^2_0(\Tan S)$. We deduce (see Definition \ref{définition:complexification}) :
\begin{align*}
	\vol^{\Sym^2_0(\Tan S)}(\partial_z^2,\partial_{\bar{z}}^2) &= h_{\Sym^2_0(\Tan S)}(J \partial_z^2,\bar{\partial_{\bar{z}}^2}) \\
	&= h_{\Sym^2_0(\Tan S)}(i \partial_z^2, \partial_z^2) = i. 
\end{align*}
By equation \eqref{éq:II_symétrie_conjuguée}, we have $\II(\partial_{\bar{z}}^2) = \bar{\II(\partial_z^2)}$, hence
\begin{align*}
	\vol^{\Nor S}(\II(\partial_z^2),\II(\partial_{\bar{z}}^2)) &= \vol^{\Nor S}(\II(\partial_z^2),\bar{\II(\partial_z^2)}) \\
	&= h_{\Nor S}(J \II(\partial_z,\partial_z),\II(\partial_z,\partial_z)) \tag*{by definition of $\vol^{\Nor S}$.} 
\end{align*}
We recall that $\det \II_0$ is defined by $\II_0^* \vol^{\Nor S} = \det(\II_0) \vol^{\Sym^2_0(\Tan S)}$ (equality between $2$-forms). Evaluate this equality at $(\partial_z^2,\partial_{\bar{z}}^2)$ : the previous volume computations show that $\det \II_0 = -i h_{\Nor S}(J \II(\partial_z,\partial_z),\II(\partial_z,\partial_z))$. 
\end{proof}

\subsection{Maximal representations in $\SO_0(2,n)$ and volume}
\label{subsection:représentations_maximales}

Let $\Sigma$ be a closed oriented surface of genus $g \geq 2$. Several properties of the Lie group $\SO_0(2,q+1)$ allow defining particular connected components of the space of representations of $\pi_1(\Sigma)$ into this group. 
For $q>1$, $\SO_0(2,q+1)$ is simple and hermitian : this allows defining maximal representations. These have been studied from the point of view of Higgs bundles (by Garc\'{i}a-Prada, Bradlow and Gothen, see \cite{BGPG06} and \cite{Got01}) and more geometrically (see in particular \cite{BIW10} and \cite{BLIW05}). This group is also that of (orientation-preserving) isometries of $\H^{2,q}$, so that one can consider $\H^{2,q}$-convex cocompact representations (see \cite{DGK18}). When $q=2$, this group is also real split, so that there exists a Hitchin component of $\chi(\Gamma,\SO_0(2,3))$ (see \cite{Hit92}). Hitchin representations are maximal and maximal representations are $\H^{2,2}$-convex-cocompact.

We will mainly use the geometric interpretation of a maximal representation given by the following theorem. 
\begin{theoreme}[{\cite[Theorem 8]{CTT19}}]\label{theoreme:existence_unicité_CTT}
	For every maximal representation $\rho : \pi_1(\Sigma_g) \to \SO_0(2,q+1)$ ($g \geq 2$), there exists a unique $\rho$-equivariant embedding $\tilde{\Sigma} \to \H^{2,q}$ whose image is a maximal surface. 
\end{theoreme}

The set of maximal representations of $\Gamma = \pi_1(\Sigma)$ is a union of connected components of the character variety $\chi(\Gamma,\SO_0(2,q+1))$, denoted by $\chi^{max}(\Gamma,\SO_0(2,q+1))$. For $q=0$, we thus recover Teichmüller space $\rm{Teich}(\Sigma)$. For $q=1$, since $\SO_0(2,2)$ is a double cover of $\SO_0(2,1)^2$, a maximal representation of $\pi_1(\Sigma)$ into $\SO_0(2,2)$ corresponds to a pair of maximal representations into $\SO_0(2,1)$ and a representation $\epsilon$ into $\Z/2\Z$. We deduce that
\[\chi^{max}(\Gamma,\SO_0(2,2)) = \rm{Teich}(\Sigma)^2 \times (\Z/2\Z)^{2g}. \]

For $q \geq 2$, given a maximal representation $f : \Gamma \to \SO_0(2,1)$ and a representation $\alpha : \Gamma \to \O(n)$, we define a representation of $\Gamma$ into $\SO_0(2,q+1)$ by $(f \otimes \det(\alpha)) \oplus \alpha$ (see \cite{CTT19}) : 
\begin{equation}\label{diagramme:représentation_fuchsienne}
	\begin{tikzcd}[row sep=tiny]
		&& {\O(2,1)} && \\
		{\Gamma} && \bigoplus && {\SO_0(2,q+1)} \\
		&& {\O(q)}
		\arrow["{f\ \otimes\ \det(\alpha)}", from=2-1, to=1-3]
		\arrow["\alpha"', from=2-1, to=3-3]
		\arrow[from=2-3, to=2-5]
	\end{tikzcd}
\end{equation}
Given a maximal representation $\rho$, we denote by $S_\rho$ the maximal surface given by Theorem \ref{theoreme:existence_unicité_CTT}. The normal bundle $\Nor S_\rho$ descends to a vector bundle $\Nor S_\rho/\Gamma$ over the quotient $\Sigma = S_\rho/\Gamma$. For example, if $\rho$ is given by Diagram \eqref{diagramme:représentation_fuchsienne}, then this bundle is the flat bundle over $\Sigma$ with holonomy $\alpha$. Let us denote by $\rm{sw}_1(\rho) \in \cohom^1(\Sigma,\Z/2\Z)$ and $\rm{sw}_2(\rho) \in \cohom^2(\Sigma,\Z/2\Z)$ the Stiefel--Whitney classes of this bundle. 

When $q \geq 3$, every connected component of $\chi^{max}(\Gamma,\SO_0(2,q+1))$ contains a unique connected component of representations of the form \eqref{diagramme:représentation_fuchsienne} (see \cite{BGPG06}). The work of \cite{CTT19} gives the following geometric interpretation : two maximal representations $\rho_1$ and $\rho_2$ are in the same connected component if, and only if, the normal bundles $\Nor S_{\rho_1}/\Gamma$ and $\Nor S_{\rho_2}/\Gamma$ are isomorphic as vector bundles over $\Sigma$. We deduce that when $q \geq 3$, the connected components of $\chi^{max}(\Gamma,\SO_0(2,q+1))$ are classified by $\rm{sw}_1$ and $\rm{sw}_2$ : this space hence has $2^{2g+1}$ connected components. 

In the case $q=2$, when $\rm{sw}_1(\rho) = 0$, the bundle $\Nor S_\rho/\Gamma$ is orientable. Its Euler class (well defined modulo choice of orientation) defines an element of $\cohom^2(\Sigma,\Z) = \Z$, called degree of $\rho$ and denoted by $\deg(\rho)$ (we shall choose the orientation to have $\deg(\rho) \geq 0$), whose reduction in $\cohom^2(\Sigma,\Z/2\Z)$ is $\rm{sw}_2(\rho)$. The space $\chi^{max}(\Gamma,\SO_0(2,q+1))$ then has $2(2^{2g}-1)$ connected components on which $\rm{sw}_1$ does not vanish, and $4g-3$ connected components on which it vanishes. The latter are classified by the degree, which satisfies $0 \leq \deg \leq 4g-4$ : we denote $\chi^{max}_d(\Gamma,\SO_0(2,3))$ the connected component on which the degree is $d$. The space $\chi^{max}(\Gamma,\SO_0(2,3))$ then has $2(2^{2g} -1) + (4g-3)$ connected components :
\[\bigsqcup_{\rm{sw}_1 \neq 0,\ \rm{sw}_2} \chi^{max}_{\rm{sw}_1,\ \rm{sw}_2}(\Gamma,\SO_0(2,3))\ \sqcup \bigsqcup_{0 \leq d \leq 4g-4} \chi^{max}_d(\Gamma,\SO_0(2,3)). \]
The connected component corresponding to $d = 4g-4$ is called Hitchin component. Those corresponding to $0 < d \leq 4g-4$, discovered by Gothen (see \cite{Got01}), are called Gothen components. They are characterised by the fact that they contain only irreducible representations.

We now define the volume of a maximal representation $\rho$ from $\Gamma$ to $\SO_0(2,q+1)$. There exist properly convex closed subsets of $\H^{2,q}$ preserved by $\Gamma$ and on which $\Gamma$ acts properly discontinuously and cocompactly (see \cite{DGK18}). Among these, there is a unique minimal convex set, which coincides with the convex hull of the limit set of $\rho$. This convex set is also equal to $\Conv(S_\rho)$ : in practice, this is the definition we will use. As $\Gamma$ acts on $\Conv(S_\rho)$ properly discontinuously, cocompactly and by isometries, the quotient is equipped with a volume form. 

\begin{definition}
	We call \emph{volume} of $\rho$ the volume of the quotient :
	\[\Vol(\rho) = \Vol(\Conv(S_\rho)/\Gamma). \]
\end{definition}

Note that a representation $\rho$ given by Diagram \eqref{diagramme:représentation_fuchsienne} stabilises $\R^{2,1}$ and its limit set lies in the projectivisation $\P(\R^{2,1})$ : this projectivisation also contains the convex hull of the limit set, so we have $\Vol(\rho) = 0$. On the contrary, the Gothen components contain only irreducible representations, hence of non-zero volume. These are therefore the only connected components of the $\chi^{max}_d(\Gamma,\SO_0(2,q+1))$ ($q \geq 1$) on which the volume is nowhere vanishing, as mentioned in the introduction. We show in fact that their volume is bounded below by a strictly positive constant (Theorem \ref{theoreme:borne_inf}).

\subsection{Compactness results}
\label{subsection:compacité}

We collect compactness results which will be of regular use. To establish Proposition \ref{proposition:Seppi} in full generality, we state them in signature $(p,q)$. They will however be mostly used in signature $(2,q)$ (in that case, most were proved in \cite{LT22}). 
\begin{definition}\label{définition:surface_pointée}
	Let us endow the set $\mathcal{M}_{p,q}$ of complete maximal $p$-submanifolds $S$ of $\H^{p,q}$ with the $\cal{C}_{loc}^\infty$ topology (see \cite{SST23}, Section 5). We define similarly $\mathcal{M}^*_{p,q}$, the set of pairs $(S,x)$, with $S$ in $\mathcal{M}_{p,q}$ and $x$ in $S$, endowed with the topology inherited from $\mathcal{M}_{p,q} \times \H^{p,q}$. 
\end{definition}
Define a map $\tau$ from $\mathcal{M}^*_{p,q}$ into the Grassmannian $\mathscr{G}_p(\R^{p,q+1})$ of positive definite $p$-subspaces of $\R^{p+q+1}$, associating to $(S,x)$ its tangent space $\Tan_x S$. 
\begin{theoreme}[\cite{SST23}, Theorem 5.3]\label{theoreme:compacité_SST}
	The function $\tau : \mathcal{M}_{p,q}^* \to \mathscr{G}_p(\R^{p,q+1})$ is proper. 
\end{theoreme}
\begin{corollaire}\label{corollaire:compacité_fonction_invariante}
	Any continuous function from $\mathcal{M}_{p,q}^*$ to $\R$ which is invariant under the action of $\O(p,q+1)$ has compact image. 
\end{corollaire}
This compactness result allows obtaining uniform bounds on certain quantities defined on $\mathcal{M}_{p,q}^*$ (see for example Proposition \ref{proposition:courbure_petite_contrôle_uv}). The main ingredient of these compactness results is the following bound on the second fundamental form, due to Ishihara (Theorem 1.2 of \cite{Ish88}), which uses the Omori-Yau maximum principle. The effective bound was established by \cite{Cheng94} in $\H^{2,1}$ and by \cite{MT25} for general $p,q$. The completeness of $S$ is crucial here since it implies that the second fundamental form $\II$ is bounded. 
\begin{theoreme}\label{theoreme:Ishihara}
	Let $S$ be a complete maximal $p$-submanifold of $\H^{p,q}$. Its second fundamental form $\II$ satisfies $\norm{\II}_2^2 \leq p \min(p-1,q)$. 
\end{theoreme}
\begin{theoreme}[Omori-Yau maximum principle]\label{théorème:Omori_Yau}
	Let $M$ be a complete Riemannian manifold whose Ricci curvature is bounded from below. Let $f : M \to \R$ of class $\cal{C}^2$ and bounded from above. For every $\epsilon > 0$, there exists a point $x \in M$ at which
	\[f(x) \geq \sup f - \epsilon, \norm{\nabla f} \leq \epsilon,\ \Delta f \leq \epsilon. \]
\end{theoreme}
\begin{remarque}\label{remarque:Omori_Yau}
	The Omori-Yau maximum principle also applies when taking $f$ with values in $\R \cup \{-\infty\}$, continuous, bounded from above and of class $\cal{C}^2$ on $f^{-1}(\R)$. To see it, apply the classical Omori-Yau maximum principle to the function $g$ defined by $g(x) = f(x)$ if $f(x) \geq \sup f -1$, $g(x) = f(x)$ otherwise, then made $\cal{C}^2$. 
\end{remarque}
Submanifolds realising the case of equality in Theorem \ref{theoreme:Ishihara} play a particular role. We shall only need some of them, studied in \cite{Barbot13} in $\H^{2,1}$, which we call Barbot surfaces (following \cite{LT22}). 
\begin{exemple}[Barbot surfaces]\label{exemple:couronne}
	Let $q$ be an integer, $q \geq 1$. Choose an isometric embedding $f : \R \hookrightarrow \Hc^{1,1} \hookrightarrow \R^{1,1}$ and decompose $\R^{2,q+1}$ into an orthogonal sum:
	\[\R^{2,q+1} = \R^{1,1} \oplus \R^{1,1} \oplus \R^{0,q-1}. \]
	We define an embedding of $\R^2$ into $\Hc^{2,q}$ by
	\[\iota : (t,s) \mapsto \frac{1}{\sqrt{2}} (f(t) \oplus f(s) \oplus 0). \]
	The map $\R^2 \stackrel{\iota}{\to} \Hc^{2,q} \to \H^{2,q}$ is an embedding. Let us denote its image by $S$ and endow it with the metric induced by $\H^{2,q}$. This embedding is an isometry of $\R^2$ onto $S$, which is therefore spacelike and flat. Moreover, $\iota$ conjugates the action of $\R^2$ on itself by isometries to the action of $\R^2$ on $\H^{2,q}$ by $\R^2 \to O(1,1)^2 \hookrightarrow O(2,q+1)$. Using the symmetries of $S$, we show that it is a maximal submanifold. We call Barbot surface a surface of $\H^{2,q}$ equal to $S$ modulo isometry of $\H^{2,q}$. 
\end{exemple}

The following proposition is due to \cite{Barbot13} for $\H^{2,1}$ and \cite[Corollary 5.7]{LT22} for $\H^{2,q}$. Let us recall (Theorem \ref{theoreme:Ishihara}) that the sectional curvature $K$ of a complete maximal surface of $\H^{2,q}$ satisfies $K \leq 0$. 
\begin{proposition}\label{proposition:rigidité_couronnes}
	Let $q \geq 1$ be an integer and $S$ be a complete maximal surface of $\H^{2,q}$. If its sectional curvature vanishes at a point, then $S$ is a Barbot surface (see Example \ref{exemple:couronne}). 
\end{proposition}

Proposition \ref{proposition:rigidité_couronnes} and Corollary \ref{corollaire:compacité_fonction_invariante} lead to the following principle : the closer the curvature at $x$ of a pointed complete maximal submanifold $S$ of $\H^{2,q}$ is to $0$, the more $S$ resembles the Barbot surface in a neighbourhood of $x$. We shall need a quantitative version of this convergence : this is the object of Section \ref{subsection:décroissance_exponentielle}. 

\subsection{Uniform elliptic estimates}
\label{subsection:estimées_elliptiques_uniformes}

In this section, we establish elliptic estimates that will be crucial for the proof of Proposition \ref{proposition:Seppi} (which is why we state them in any signature) and Section \ref{subsection:décroissance_exponentielle}. Let $p \geq 1$ and $q \geq 0$ be integers and $(S,x) \in \cal{M}^*_{p,q}$ denote a maximal complete pointed $p$-submanifold of $\H^{p,q}$ (see Definition \ref{définition:surface_pointée}). We denote by $\nabla^{S}$ the connection of $S$. 
\begin{definition}\label{définition:norme_Ck}
	Given a function $u : S \to \R$, we denote by $\bar u$ the function
	\[\bar u : \Tan_x S \to \R,\, z \mapsto u(\exp_x(z)). \]
	Given an open subset $U$ of an Euclidean space and $f : U \to \R$, we also denote $\norm{f}_{\mathscr{C}^k(U)} = \sum_{\abs {\alpha} \leq k} \norm{\rm{D}^\alpha f}_{\mathscr{C}^0(U)}$. 
\end{definition}
The Laplacian of $S$, denoted by $\Delta = \Delta^S$, allows defining a differential operator on $\cal{C}^{\infty}(\Tan_x S)$ by $L_{(S,x)} = \exp_x^* \Delta^S$. Given a real number $c$, write
\[L_{(S,x)} + c \Id = \sum_{i,j} a_{ij} \rm{D}_{ij} + \sum_i b_i \rm{D}_i + c \Id \]
where the $a_{ij}$ and $b_i$ are $\cal{C}^\infty$ functions from $\Tan_x S$ to $\R$. The matrix $(a_{ij})$ is symmetric, with eigenvalues denoted by $\lambda_1(a_{ij}) \leq \lambda_2(a_{ij})$. To prove uniform elliptic estimates, we shall use a control on the coefficients of $L_{(S,x)} + c \Id$ of the following form:
\begin{equation}\label{éq:uniforme_ellipticité}
	\begin{aligned}
		\lambda \leq \lambda_1(a_{ij}) &&\text{(uniform ellipticity)} \\
		\norm{a_{ij}}_{\cal{C}^1}, \norm{b_i}_{\cal{C}^1}, \norm{c}_{\cal{C}^1} \leq \Lambda &&\text{($\cal{C}^1$ bounded coefficients)}
	\end{aligned}
\end{equation}
\begin{remarque}
In general, $\exp_x$ could not be a local diffeomorphism : we do not know if the sectional curvature of a maximal complete $p$-submanifold of $\H^{p,q}$ is non-positive for $p > 2$. The pullback $\exp_x^* \Delta^S$ is then not well-defined. However, it is always a diffeomorphism from a ball centred at $0$ (whose radius depends only on $\sup \abs{K}$) onto its image, which will suffice for us. We therefore ignore this problem of definition.
\end{remarque}

\begin{lemme}\label{cor:coefficients_laplacien_bornés}
	For all $R > 0$ and all $c$ in $\R$, there exist constants $\lambda, \Lambda > 0$ such that for any pointed $p$-submanifold $(S,x)$ in $\cal{M}^*_{p,q}$, we have the bounds \eqref{éq:uniforme_ellipticité} on the ball $\rm{B}(0,R)$ of $\Tan_x S$. 
\end{lemme}
\begin{proof}
	The various functions to be bounded above and below by constants are continuous functions of $\norm{a_{ij}}_{\cal{C}^1}$, $\norm{b_i}_{\cal{C}^1}$ and $\norm{c}_{\cal{C}^1}$ (the norm $\norm{\cdot}_{\cal{C}^1}$ is taken on the ball $\rm{B}(0,R) \subset \Tan_x S$). These quantities are functions of $(S,x) \in \cal{M}_{p,q}^*$, continuous (by definition of the topology on $\mathcal{M}_{p,q}^*$) and invariant under the action of $\O(p,q+1)$. Corollary \ref{corollaire:compacité_fonction_invariante} therefore asserts that their respective images are compact subsets of $\R$. This proves the second inequality \eqref{éq:uniforme_ellipticité}. For the first it suffices to show that $\lambda_1(a_{ij}) > 0$, which follows from the following computation. \\
	If we consider the Hessian $\Hess f$ as a function from $(\Tan S)^{\otimes 2}$ to $\R$, then it is defined by $(\Hess f) (\sf{u},\sf{v}) = \scal{\nabla_{\sf{u}} \grad f}{\sf{v}}$. In normal coordinates $(e_i)$ whose dual basis is denoted by $(e^i)$, it is hence $\Hess f = \sum_{i,j} (\partial_{ij} f + \sum_k \Gamma_{i,j}^k \partial_k f) e^i \otimes e^j$ (with $\Gamma_{ij}^k$ the Christoffel symbols). 
	Let us use the metric to view $\Hess f$ as a function from $\Tan S \otimes \Tan^* S$ to $\R$. The metric $g$ is the isomorphism from $\Tan S$ to $\Tan^* S$ given by $g(x)(x) = \scal{x}{y}$ : denote by $(g_{ij})$ its matrix in the bases $(e_i)$ and $(e^i)$. Denoting by $(g^{ij})$ the coefficients of $g^{-1}$, we obtain
	\[\Hess f = \sum_{i,j} (\partial_{ij} f + \sum_k \Gamma_{ij}^k \partial_k f) (\sum_k g^{ik} e^k) \otimes e_j. \]
	Taking the trace, we obtain :
	\[\Delta f = \tr_g \Hess f = \sum_{i,j} (\partial_{ij} f + \sum_k \Gamma_{ij}^k \partial_k f) g^{ij}. \]
	We therefore have $a_{ij} = g^{ij}$ and $b_k = \sum_k (\sum_{i,j} \Gamma_{ij}^k g^{ij})$. The matrix $(a_{ij}) = (g^{ji})$ is the inverse of $(g_{ij})$, which is symmetric positive definite, hence $\lambda_1(a_{ij}) > 0$. 
\end{proof}

\begin{proposition}[uniform Schauder estimate]\label{prop:schauder}
	Given integers $p \geq 1$ and $q \geq 0$ and real numbers $R > 0$ and $c$, there exists a constant $C_S > 0$ such that for every $(S,x) \in \mathcal{M}^*_{p,q}$ and every function $u$ in $\mathscr{C}^2(S,\R)$ satisfying $(\Delta^S + c) u = 0$, $\bar u$ satisfies the Schauder estimate
	\[\norm{\bar u}_{\mathscr{C}^2(\rm{B}(0,R))} \leq C_S \norm{\bar u}_{\mathscr{C}^0(\rm{B}(0,2R))}. \]
\end{proposition}
\begin{proof}
	Consequence of the classical Schauder estimate, which requires uniform ellipticity and a $\cal{C}^{0,\alpha}$ control of the coefficients of the operator $L = \exp_x^* (\Delta^S + c)$ for some $\alpha > 0$ (see \cite{GT01}, Theorem 6.2). These controls are given by Lemma \ref{cor:coefficients_laplacien_bornés}. $C_S$ depends only on $(p,\lambda,\Lambda)$, hence on $(p,q,R,c)$ by Lemma \ref{cor:coefficients_laplacien_bornés}. 
\end{proof}
\begin{proposition}[uniform Harnack inequality]\label{prop:harnack}
	Given integers $p \geq 1$ and $q \geq 0$ and real numbers $R > 0$ and $c \leq 0$, there exists a constant $C_H > 0$ such that for every $(S,x) \in \mathcal{M}^*_{p,q}$ and every function $u \in \mathscr{C}^2(S,\R)$ satisfying $u \geq 0$ and $(\Delta^S + c) u = 0$, $\bar u$ satisfies the Harnack inequality
	\[\norm{\bar u}_{\rm{L}^{\infty}(\rm{B}(0,R))} \leq C_H \bar{u}(0). \]
\end{proposition}
\begin{proof}
	As for \ref{prop:schauder}, consequence of the classical Harnack inequality, which requires a $\rm{L}^{\infty}$ control of the coefficients of $L = \exp_x^* (\Delta^S + c)$, that $L$ be uniformly elliptic and that the coefficient of $\Id$ in the expression of $L$, here equal to $c$, satisfies $c \leq 0$ (see \cite{GT01}, Theorem 8.20). The necessary controls are given by Lemma \ref{cor:coefficients_laplacien_bornés}. $C_S$ depends only on $(p,\frac{\Lambda}{\lambda},R)$, hence on $(p,q,R,c)$ by Lemma \ref{cor:coefficients_laplacien_bornés}. 
\end{proof}

The last uniform estimate will be useful only for surfaces ($p=2$), in section \ref{subsection:décroissance_exponentielle}. We hence establish it only in that case. 
\begin{proposition}[elliptic $\rm{W}^{2,\bar{p}}$ estimate and uniform Sobolev inequality]\label{proposition:estimées_W2p}
	Given an integer $q \geq 0$ and real numbers $\bar{p} > 2$ and $c$, there exists $C > 0$ such that for every $(S,x) \in \mathcal{M}^*_{2,q}$ and every functions $u,f$ from the ball $B(x,1)$ of $S$ to $\R$, bounded, of class $\cal{C}^2$ and satisfying $(\Delta^S - c) u = f$, we have
	\[\norm{(\nabla u)_x} \leq C \left[ \norm{\bar{u}}_{\rm{L}^{\bar{p}}(B(0,1))} + \norm{\bar{f}}_{\rm{L}^{\bar{p}}(B(0,1))} \right]. \]
\end{proposition}
\begin{proof}
	Let us first apply an elliptic $\rm{W}^{2,\bar{p}}$ estimate (see \cite{GT01}, Theorem 9.11). The necessary controls are given by Lemma \ref{cor:coefficients_laplacien_bornés} and we indeed have $\exp_x^* u \in \rm{W}^{1,\bar{p}}_{\rm{loc}}(B(0,1)) \cap \rm{L}^{\bar{p}}(B(0,1))$ and $\exp_x^* f \in \rm{W}^{2,p}(B(x,1))$. There therefore exists $C_W$, depending only on $p=2$, $\bar{p}$, $\lambda$, $\Lambda$ in \eqref{éq:uniforme_ellipticité} and on the modulus of continuity of the $a_{ij}$ on $B(0,1)$, hence on $(q,\bar{p},c)$ by Lemma \ref{cor:coefficients_laplacien_bornés}, satisfying
	\[\norm{\bar{u}}_{\rm{W}^{2,\bar{p}}(B(0,\frac12))} \leq C_W \left[ \norm{\bar{u}}_{\rm{L}^{\bar{p}}(B(0,1))} + \norm{\bar{f}}_{\rm{L}^{\bar{p}}(B(0,1))} \right]. \]
	We then apply the Sobolev inequality to $\bar{u}$, which gives, for a constant $C_{\rm{Sob}}$ depending only on $\bar{p}$ :
	\[\norm{\bar{u}}_{\cal{C}^{1,2-\frac{2}{\bar{p}}}(B(0,\frac12))} \leq C_{\rm{Sob}} \norm{\bar{u}}_{\rm{W}^{2,\bar{p}}(B(0,\frac12))}. \]
	Let us further remark that we have:
	\[\norm{(\nabla u)_x} \leq \norm{\bar{u}}_{\cal{C}^1(B(0,\frac12))}. \]
	We conclude the proof by combining these three inequalities. 
\end{proof}

\section{Lower bound on the volume}
\label{section:borne_inférieure}

We prove here Theorem \ref{theoreme:borne_inf}. The idea is the following. Given a complete surface $S$ of $\H^{2,2}$, we begin by relating $\Conv(S)$ to the second fundamental form $\II$ of $S$ (Proposition \ref{proposition:Seppi}). The Ricci equation then relates $\II$ to the curvature of the normal bundle $\courb^{\Nor}$, itself related to $\deg(\rho)$ by the Chern-Weil isomorphism. 

\subsection{Seppi's estimates}

We begin by drawing a geometric consequence (Proposition \ref{proposition:Seppi}) from the elliptic estimates of section \ref{subsection:estimées_elliptiques_uniformes}. Let $h$ be a point in $\R^{p,q+1}$, which defines a function
\[u : S \to \R,\ x \mapsto \scal{x}{h}. \]
\begin{lemme}\label{lemme:majoration_seconde_forme_fondamentale}
	We denote as previously $\bar{u} = u \circ \exp^S_x$. For all $R > 0$ and all unit tangent vector $\sf{v} \in \Tan^1_x S$, we have:
	\[\abs{\scal{\II_x(\sf{v},\sf{v})}{h}} \leq \norm{\bar u}_{\mathscr{C}^2(\rm{B}(0,R))}. \]
\end{lemme}
\begin{proof}
	Let $\gamma : t \mapsto \exp^S_x(t\sf{v})$ be the geodesic of $S$ starting from $x$ with speed $\sf{v}$. Let us study the function $f : t \mapsto \scal{\gamma(t)} h = \bar u (t\sf{v})$. We have $f'(t) = \scal{\dot \gamma(t)} h$ and
	\begin{align*}
		f''(t) &= \scal{\nabla^{\R}_{\dot \gamma} \dot \gamma} h = \scal{\nabla^{\H}_{\dot \gamma} \dot \gamma + \scal{\dot \gamma}{\dot \gamma} \gamma} h \tag*{by equation \eqref{éq:Hpq_ombilique}} \\
		&= \scal{\II(\dot \gamma, \dot \gamma) + \gamma} h = \scal{\II(\sf{v},\sf{v})} h + \bar u (t\sf{v}).
	\end{align*}
	Since $f''(0) = \ddtzero{\bar{u}(t\sf{v})}$, we have $\abs{f''(0)} + \abs{\bar{u}(0)} = \abs{\ddtzero{\bar{u}(t\sf{v})}} + \abs{\bar{u}(0)} \leq \norm{\bar u}_{\mathscr{C}^2(\rm{B}(0,R))}$ (see Definition \ref{définition:norme_Ck}), whence the announced inequality.
\end{proof}

The following proposition is inspired by \cite[Theorem 1.A]{Sep17} and generalises it. 
\begin{proposition}[Seppi's estimates]\label{proposition:Seppi}
	For all integers $p \geq 1$ and $q \geq 0$, there exists a constant $C > 0$ such that for any pointed maximal complete $p$-submanifold $(S,x)$ of $\Hc^{p,q}$ and any $\sf{v} \in \Tan_x^1 S$, we have $\exp(C\, \II(\sf{v},\sf{v})) \in \Conv(S)$. 
\end{proposition}
\begin{proof}
	The idea is the following. The function $u$ defined above satisfies $(\Delta - p) u = 0$ according to Lemma \ref{lemme:u_harmonique}. We can therefore apply to it the uniform Schauder estimate and Harnack inequality (section \ref{subsection:estimées_elliptiques_uniformes}), from which we then draw geometric consequences. 
	
	Let us simplify the problem. If $q=0$, the statement is trivial. Let $\sf{v}$ in $\Tan_x S$ and $\cal{C} = \Conv(S)$. If $x$ is in $\cal{C}$, then $S$ is contained in a hyperplane $H$ of $\R^{p,q+1}$ of signature $(p,q)$ (Lemma \ref{lemme:surface_dans_hyperplan}) so we are reduced to proving the same statement in $H \cap \Hc^{p,q} \cong \Hc^{p,q-1}$ and we conclude by induction on $q$. We now assume $x$ is not in $\partial \cal{C}$. 
	
	Let us write $\II(\sf{v},\sf{v}) = s \sf{n}$ with $s$ in $\R_{\geq 0}$ and $\sf{n}$ in $\Nor_x S$ of norm $\norm{\sf{n}} = 1$. The geodesic $\gamma : t \mapsto \exp_x^\H (t\sf{n})$, starting from $x$ with velocity $\sf{n}$, reaches the boundary of $\cal{C}$ in a time $d > 0$ (since $x \notin \partial \cal{C}$), at a point $y$ where we can find a supporting hyperplane of $\cal{C}$, of the form $H = \{h\}^{\perp}$ : $S$ is therefore on one side of the hyperplane $H$ in $\R^{p,q+1}$. Up to multiplying $h$ by $-1$, the function $u = \scal{h}{\cdot} : S \to \R$ is non-negative. 
	
	The goal is now to find an inequality of the form $C s \leq d$. Since $s$ is bounded ($s \leq \norm{\II}_\infty \leq \norm{\II}_2 < \sqrt{pq}$, see Theorem \ref{theoreme:Ishihara}), it suffices to prove such an equality assuming $d \leq \frac{\pi}{4}$, which we shall do from now on. Lemma \ref{lemme:majoration_seconde_forme_fondamentale}, the Schauder estimates (Proposition \ref{prop:schauder}) and the Harnack inequality (Proposition \ref{prop:harnack}) already allow us to obtain the following inequalities : for all $R > 0$, there exist constants $C_S, C_H > 0$ depending only on $p$, $q$ and $R$ and satisfying
	\begin{align}
		\abs{\scal{\II(\sf{v},\sf{v})}{h}} \stackrel{\ref{lemme:majoration_seconde_forme_fondamentale}}{\leq} \norm{\bar u}_{\mathscr{C}^2(\rm{B}(0,R))} &\stackrel{\ref{prop:schauder}}{\leq} C_S \norm{\bar u}_{\mathscr{C}^0(\rm{B}(0,2R))} \stackrel{\ref{prop:harnack}}{\leq} C_S C_H \bar{u}(0), \tag*{hence} \\
		\abs{\scal{\II(\sf{v},\sf{v})}{h}} &\leq C_S C_H \scal{x}{h}. \label{éq:ineg_seppi}
	\end{align}
Let us now use the definition of $d$ : 
\begin{align}
	0 &= \scal{\gamma(d)}{h} \tag*{hence} \nonumber \\
	0 &= \cos(d)\scal{x}{h} + \sin(d) \scal{\sf{n}}{h}. \label{éq:demo_seppi}
\end{align}
Since $x$ is not in $\partial \cal{C}$, we have $\scal{x}{h} \neq 0$. Since moreover $0 \leq d \leq \frac{\pi}{4}$, we also have $\cos(d) \neq 0$ so equality \eqref{éq:demo_seppi} shows $\scal{\sf{n}}{h} \neq 0$. By definition of $s$, we have
\begin{align*}
	s \abs{\scal{\sf{n}}{h}} &= \abs{\scal{\II(\sf{v},\sf{v})}{h}} \\
	&\leq C_S C_H \scal{x}{h} \tag*{by \eqref{éq:ineg_seppi}} \\
	&\leq C_S C_H \tan(d) \abs{\scal{\sf{n}}{h}} \tag*{by \eqref{éq:demo_seppi}.}
\end{align*}
Since $\scal{\sf{n}}{h}$ is non-zero, we obtain
\begin{align*}
	s &\leq C_S C_H \tan(d) \\
	&\leq \frac{4}{\pi} C_S C_H d \tag*{since $\tan(d) \leq \frac{4d}{\pi}$ for $0 \leq d \leq \frac{\pi}{4}$.}
\end{align*}
We obtain an inequality of the type $C\ s \leq d$, under the hypothesis $d \leq \frac{\pi}{4}$. We had reduced the problem to proving such an inequality, which therefore concludes the proof.
\end{proof}

\subsection{Ricci equations, Chern-Weil isomorphism}
\label{subsection:Chern-Weil}

This section is devoted to the proof of Theorem \ref{theoreme:borne_inf}, for which we set the following definition. 
\begin{definition}\label{définition:volume_tranche}
Let $S$ be a maximal surface in $\H^{p,q}$. Denote by $\Nor_x^0 S \subset \Nor_x S$ the connected component of $0$ in $\{\sf{n} \in \Nor_x S \mid \exp_x(\sf{n}) \in \Conv(S)\}$ and denote
\[\bar{\bf{V}}_S(x) = \Vol \left( \Nor_x^0 S \right). \]
The volume $\Vol$ is that of the Euclidean space $\Nor_x S$. 
\end{definition}

\begin{theoremeLettre}\label{theoreme:borne_inf}
	There exists a constant $C > 0$ such that for every integer $g \geq 2$ and every Gothen representation $\rho$ of $\pi_1(\Sigma_g)$ into $\SO_0(2,3)$, we have $\Vol(\rho) \geq C\, \deg(\rho)$. 
\end{theoremeLettre}
\begin{proof}
Let $\Sigma$ be a closed oriented surface of genus $g \geq 2$ and $\rho$ a Gothen representation from $\Gamma = \pi_1(\Sigma)$ to $\SO_0(2,3)$ (see Section \ref{subsection:représentations_maximales}). By Theorem \ref{theoreme:existence_unicité_CTT}, there exists a unique embedding $u$ from the universal cover $\tilde{\Sigma}$ into $\H^{2,2}$, $\rho$-equivariant and whose image is a maximal surface $S$. 

Let $s_0>0$ be a real number smaller than $\frac{2}{\sqrt3}$ and than the constant $c$ of Lemma \ref{lemme:jacobienne_minorée}. Let us define a subset of the bundle $\Nor S$ by
\[\Nor^{\leq} S = \cup_{x \in S} \{\sf{n} \in \Nor_x^0 S \mid \norm{n} \leq s_0\}. \]
We have $\Conv(S) = \cup_{x \in S} \exp(\Nor^0_x S)$ thus
\begin{align}
	\Vol(\rho) &= \Vol(\Conv(S)/\Gamma) \nonumber \\
	&\geq \Vol \left[ \exp \left( \Nor^{\leq} S \right) / \Gamma \right]. \label{éq:démo_borne_inf_volume}
\end{align}
Let us choose a fundamental domain $\cal{F} \subset S$ for the action of $\Gamma$. Lemma \ref{lemme:nu_injective} states that the restriction of $\exp$ to $\Nor^{\leq} S$ (denoted by $\nu$) is injective. The Sasaki metric on $\Tan \H^{2,2}$, defined in Section \ref{subsection:geo_diff_Hpq}, restricts to a metric on $\Nor S$, whose associated measure we denote by $\mu_{\Nor S}$. We also denote by $\mu_\H$ the measure on $\H^{2,2}$ induced by its metric. Denoting $(\Nor^{\leq} S)_{|\cal{F}} = (\Nor^{\leq} S) \cap (\Nor S)_{|\cal{F}}$, we then have
\[\Vol \left[ \exp \left( \Nor^{\leq} S \right) / \Gamma \right] = \int_{(\Nor^{\leq} S)_{|\cal{F}}} \frac{\nu^* \mu_\H}{\mu_{\Nor S}} \d\mu_{\Nor S}(\sf{n}), \]
the Radon-Nikodym derivative $\frac{\nu^* \mu_\H}{\mu_{\Nor S}}$ being defined in Appendix \ref{annexe:nu}. 
Lemma \ref{lemme:jacobienne_minorée} therefore gives the following upper bound:
\begin{equation}\label{éq:démo_borne_inf_jacobien}
	\Vol \left[ \exp \left( \Nor^{\leq} S \right) / \Gamma \right] \geq \int_{(\Nor^{\leq} S)_{|\cal{F}}} \frac12 \d\mu_{\Nor S}(\sf{n}). 
\end{equation}
Let us now use Proposition \ref{proposition:Seppi} : there exists a universal constant $C > 0$, which we can assume to be less than $s_0$, such that for all $\sf{v} \in \Tan S$ of norm $\norm{\sf{v}} \leq 1$, we have $\exp(C\, \II(\sf{v},\sf{v})) \in \Conv(S)$. Let us denote by $\D_x \subset \Sym^2_0(\Tan_x S)$ the unit disc, so that this proposition states that $C \II(\D_x)$ is included in $\Nor^{\leq} S$, so that we have 
\[\int_{(\Nor^{\leq} S)_{|\cal{F}}} \frac12 \d\mu_{\Nor S}(\sf{n}) \geq \int_{\bigcup_{x \in \cal{F}} C \II(\D_x)} \frac12 \d\mu_{\Nor S}(\sf{n}). \]
Lemma \ref{lemme:Sasaki_par_parties} states that to compute the volume of $\bigcup_{x \in \cal{F}} C \II(\D_x)$, we can integrate the volume of the slice $C \II(\D_x)$ for $x$ in $\cal{F}$. This gives : 
\begin{align}
	\int_{(\Nor^{\leq} S)_{|\cal{F}}} \frac12 \d\mu_{\Nor S}(\sf{n}) &\geq \frac12 \int_{\cal{F}} \Vol^{\Nor_x S} ( C \II(\D_x) ) \d\mu_S(x) \nonumber \\
	&= \frac{C^2}{2} \int_{\cal{F}} \Vol ( \II(\D_x) ) \d\mu_S(x). \label{éq:démo_borne_inf_Sasaki}
\end{align}

We then define the bundle $\Nor \Sigma$ over $\Sigma = S/\Gamma$ as the quotient $(\Nor S)/\Gamma$. This bundle is orientable (see section \ref{subsection:représentations_maximales}) : we choose an almost complex structure $J$, which lifts to an almost complex structure $J_{\Nor}$ on $\Nor S$. This defines a Euler class $\bf{eu}(\Nor \Sigma) \in \cohom^2(\Sigma,\Z)$. The connection of $\Nor S$ descends to $\Nor \Sigma$ and we denote by $\courb^{\Nor S}$ and $\courb^{\Nor \Sigma}$ their respective curvatures. Similarly, the metric of $S$ descends to $\Sigma$. Let us choose an orientation of $S$, which defines a volume $\omega^S$ on $S$. Let us note that $\II(\bb{D}_x)$ is the image of the unit disc of $\Sym^2_0(\Tan_x S)$. Its volume is therefore $\pi \abs{\det \II}$, so that the Ricci formula \eqref{éq:Ricci} writes here
\begin{equation}\label{éq:Ricci_disque}
	\courb^{\Nor S}_x = \pm \pi [\Vol \II(\bb{D}_x)] J_{\Nor} \otimes \omega^S_x. 
\end{equation}
Let us recall that the Pfaffian $\rm{Pf}$ is defined on $\frak{so}_2(\R)$ by
\[\rm{Pf} \begin{pmatrix}
	0 & -a \\
	a & 0
\end{pmatrix} = a\]
and that the Chern-Weil isomorphism states that the class $\bf{eu}(\Nor \Sigma) \in \cohom^2(\Sigma,\R)$ is represented by the $2$-form given at a point $y$ of $\Sigma$ by :
\begin{align*}
	&\rm{Pf}(\frac{1}{2\pi}\courb^{\Nor \Sigma}_y) \\
	= &\rm{Pf}(\frac{1}{2\pi} \courb^{\Nor S}_x) \tag*{for any lift $x \in S$ of $y$.}
\end{align*}
Finally, let us recall that the degree of $\rho$ is defined by
\begin{align*}
	\deg(\rho) &= \abs{\int_\Sigma \bf{eu}(\Nor \Sigma)} \\
	&= \abs{\int_{\cal{F}} \rm{Pf}(\frac{1}{2\pi} \courb^{\Nor S})} \tag*{by the previous equality.}
\end{align*}
Using the expressions of $\courb^{\Nor S}$ (equation \eqref{éq:Ricci_disque}) and of the Pfaffian, we infer
\begin{align}
	\deg (\rho) &\leq \abs{\int_{\cal{F}} \pm \frac12 [\Vol \II(\bb{D}_x)] \omega^S_x} \nonumber \\
	&\leq \int_{\cal{F}} \abs{\frac12 \Vol \II(\bb{D}_x)} \d\mu_S \label{éq:démo_borne_inf_Euler}
\end{align}
by the triangle inequality. Putting together the inequalities \eqref{éq:démo_borne_inf_volume}, \eqref{éq:démo_borne_inf_jacobien}, \eqref{éq:démo_borne_inf_Sasaki} and \eqref{éq:démo_borne_inf_Euler}, we indeed obtain
\[\Vol(\rho) \leq C^2 \deg(\rho), \]
which proves Theorem \ref{theoreme:borne_inf}. 
\end{proof}

\section{Exponential decay}
\label{section:décroissance_exponentielle}

Let $S$ be a maximal surface in $\H^{2,2}$. We shall define a pair of functions on $S$ reflecting the second fundamental form and satisfying a system of elliptic differential equations (section \ref{subsection:décomposition}), which we shall use in section \ref{subsection:décroissance_exponentielle} to quantify the following fact: if the curvature of $S$ at a point $x$ is almost zero, then $S$ looks like a Barbot surface (Example \ref{exemple:couronne}) in a neighbourhood of $x$. This will allow to estimate the volume of the convex hull (Definition \ref{définition:enveloppe_convexe}) above a point $x$ (section \ref{subsection:contrôle_volume}) and then to integrate this estimate over $S$ (section \ref{subsection:contrôle_Dt}) in order to prove Theorem \ref{theoreme:borne_sup}. 

\subsection{Decomposition of the normal bundle and differential equations}
\label{subsection:décomposition}

In this section, we define functions $u$ and $v$ encoding the second fundamental form and use the Gau\ss, Ricci and Codazzi equations to establish a system of elliptic equations on $(u,v)$. 

Let $g$ denote the pseudo-Riemannian metric of $\H^{2,2}$ and $\nabla$ its Levi-Civita connection. We define Riemannian metrics on the tangent bundle $\Tan S$ and the normal bundle $\Nor S$, by restricting $g$ and $-g$ respectively. By projecting $\nabla$, we define metric connections on $\Tan S$ and $\Nor S$, of respective curvatures $\courb^{\Tan}$ and $\courb^{\Nor} \in \Gamma(\Lambda^2 (\Tan^* S) \otimes \End(\Nor S))$. Let henceforth $U$ be an open subset of $S$ and suppose given orientations of the vector bundles $\Tan U$ (thus $\Tan^* U$) and $\Nor U$. Lemma \ref{lemme:décomposition_fibré_rang_2} allows to decompose their complexifications as a sum of line bundles:
\[(\Tan U)^\C = \cal{T} \oplus \cal{T}^{-1},\ (\Tan^* U)^\C = \canon \oplus \canon^{-1},\ (\Nor U)^\C = \cal{N} \oplus \cal{N}^{-1}. \]
$\cal{T}$ and $\canon$ are usually called the holomorphic tangent bundle and the canonical bundle of $S$. These decompositions are orthogonal for the hermitian metrics on these bundles defined by extending sesquilinearly the metrics on $\Tan S$, $\Tan^* S$ and $\Nor S$, and each line bundle is preserved by the complexified ambient connection (see section \ref{subsection:fibrés_holomorphes}). Consider the second fundamental form $\II$ as a section of the bundle $\Tan^* S \otimes \Tan^* S \otimes \Nor S$ and decompose its $\C$-linear extension into types (see section \ref{subsection:fibrés_holomorphes}) :
\[\II^\C = \II^{(2,0)} + \II^{(1,1)} + \II^{(0,2)}, \]
where, for example, $\II^{(2,0)}$ is a section of $\canon^2 \otimes (\Nor S)^\C$. The second fundamental form $\II$ being defined on real bundles, we have $\bar{\II} = \II$ (see section \ref{subsection:fibrés_holomorphes}) : we deduce the equality $\II^{(0,2)} = \bar{\II^{(2,0)}}$.  Since the surface $S$ is maximal, $\II^{(1,1)}$ vanishes (see section \ref{subsection:équations_fondamentales}). $\II$ is therefore determined by $\II^{(2,0)}$ :
\[\II^\C = \II^{(2,0)} + \bar{\II^{(2,0)}}. \]
The decomposition $(\Nor U)^\C = \cal{N} \oplus \cal{N}^{-1}$ allows writing $\II^{(2,0)} = \alpha + \beta$, with $\alpha$ and $\beta$ sections of $\canon^2 \cal{N}$ and $\canon^2 \cal{N}^{-1}$ respectively. The Codazzi equation \eqref{éq:Codazzi} is written $\bar\partial \II^{(2,0)} = 0$ and $\cal{N}$, $\cal{N}^{-1}$ are holomorphic subbundles of $(\Nor U)^\C$ : we deduce that $\alpha$ and $\beta$ are holomorphic sections. Let $U_0$ denote the set of points of $U$ where $\alpha \beta \in \cohom^0(\canon^4)$ does not vanish. We recall that $\alpha \beta$ is a section of $\canon^2 \cal{N} \canon^2 \cal{N}^{-1}$, identified with $\canon^4$ as a hermitian holomorphic bundle since $\cal{N} \otimes \cal{N}^{-1} = \cal{O}$ as hermitian holomorphic bundles (Lemma \ref{lemme:décomposition_fibré_rang_2}). Define $u,v : U_0 \to \R$ by
\begin{align*}
	e^u &= \norm{\alpha}^2, \\
	e^v &= \norm{\beta}^2.
\end{align*}

\begin{remarque}\label{remarque:changement orientation_normal}
By remark \ref{remarque:changement orientation}, the other choice of orientation on $\Nor S$ would exchange the bundles $\cal{N}$ and $\cal{N}^{-1}$, hence would exchange $u$ and $v$. 
\end{remarque}

Let $K$ denote the sectional curvature of $S$, $\omega$ its volume form and $J_{\Tan} \in \End(\Tan S)$, $J_{\Tan^*} = -\transp{J_{\Tan}} \in \End(\Tan^* S)$, $J_{\Nor} \in \End(\Nor S)$ defining the orientations. We have
\begin{align*}
	\frac{\courb^{\Tan}}{\omega} &= -K J_{\Tan} \tag*{by definition, and} \\
	\frac{\courb^{\Nor}}{\omega} &= (\det \II) J_{\Nor} \tag*{by the Ricci equation \eqref{éq:Ricci}.}
\end{align*}

\begin{lemme}
	We have $\det \II = e^u - e^v$ and $K = -1 + e^u + e^v$. 
\end{lemme}
\begin{proof}
	Let $\sf{u} \in \Tan^1_x S$ be a unit tangent vector. We define unit vectors $\partial_z = \frac12 (\sf{u} - iJ\sf{u}) \in \cal{T} = \Tan^{(1,0)}_x S$ and $\partial_{\bar{z}} = \frac12 (\sf{u} + iJ\sf{u}) \in \cal{T}^{-1} = \Tan^{(0,1)}_x S$. We have a decomposition $\II^\C(\partial_z,\partial_z) = \alpha_0 + \beta_0$, with $\alpha_0$ in $\cal{N}_x$ and $\beta_0$ din $\cal{N}^{-1}_x$ verifying $\norm{\alpha_0} = e^u$ and $\norm{\beta_0} = e^v$. Lemma \ref{lemme:Ricci} therefore gives:
	\begin{align*}
		\det \II &= -i h(J \II(\partial_z,\partial_z),\II(\partial_z,\partial_z)) \\
		&= -i h(i \alpha_0 - i \beta_0,\alpha_0 + \beta_0) \tag*{since $\cal{N}^{\pm 1} = \ker(J \mp i)$} \\
		&= h(\alpha_0,\alpha_0) - h(\beta_0,\beta_0) \\
		&= e^u - e^v. 
	\end{align*}
	Let us now turn to the expression of $K$. By equation \eqref{éq:Gauss}, we have
	\begin{align*}
		K &= -1 + \norm{\alpha_0 + \beta_0}^2 \\
		&= -1 + \norm{\alpha_0}^2 + \norm{\beta_0}^2 \tag*{since $h(\alpha_0,\beta_0) = 0$}
	\end{align*}
	as required. 
\end{proof}

Since $\nabla$ preserves $\cal{N}$ and $\cal{T}$ and $J_{\Tan}$, $J_{\Nor}$ act on these bundles by multiplication by $i$, we deduce
\begin{equation}\label{éq:courbures_uv}
	\begin{aligned}
		\frac{\courb^\cal{T}}{i\omega} &= -K \Id_{\cal{T}} = (1 -e^u -e^v) \Id_{\cal{T}}, \\
		\frac{\courb^\cal{N}}{i\omega} &= (\det \II) \Id_{\cal{N}} = (e^u - e^v) \Id_{\cal{N}}.
	\end{aligned}
\end{equation}

We now apply Lemma \ref{lemme:courbure_fibré_droites} to $\alpha \in \cohom^0(\canon^2 \otimes \cal{N})$ and $\beta \in \cohom^0(\canon^2 \otimes \cal{N}^{-1})$ : this gives
\begin{align*}
	\Delta u &= 2\frac{\courb^{\canon^2 \cal{N}}}{i\omega} \tag*{and} \\
	\Delta v &= 2 \frac{\courb^{\canon^2 \cal{N}^{-1}}}{i\omega}. 
\end{align*}
Since $\courb^{\canon^2 \cal{N}^{\pm 1}} = 2\courb^\canon \pm \courb^\cal{N} = -2\courb^\cal{T} \pm \courb^\cal{N}$, we obtain
\begin{equation}\label{éq:Hitchin_uv}
	\begin{aligned}
		\Delta u = 6e^u + 2e^v -4, \\
		\Delta v = 2e^u + 6e^v -4. 
	\end{aligned}
\end{equation}

\subsection{Exponential decay}
\label{subsection:décroissance_exponentielle}

We henceforth assume that $S$ is complete (in order to use the compactness results of section \ref{subsection:compacité}). We begin by deducing from the equations \eqref{éq:Hitchin_uv} that $u$ and $v$ tend rapidly towards a limiting value when the curvature $K$ is close to $0$ (Proposition \ref{proposition:contrôle_mu}), which justifies the title of this section. We then deduce a similar control of $\II$ and of $\nabla \II$ (Corollary \ref{corollaire:contrôle_nablaII}) by means of elliptic estimates. The main ideas leading to Proposition \ref{proposition:contrôle_mu} are adapted from arguments of \cite{CL15} and \cite{OT20}, which use similar results formulated in terms of Higgs bundles. Exponential decay has been used for even longer (for example in \cite{CT90}). We choose here a more geometric formulation. 

Let us begin by rewriting the equations \eqref{éq:Hitchin_uv}. Lemma \ref{corollaire:sous-variété_simplement_connexe} shows that $S$ is contractible. One can therefore choose orientations for $\Tan S$ and $\Nor S$ : $\alpha$ and $\beta$, defined in Section \ref{subsection:décomposition}, are thus \emph{global} holomorphic sections of holomorphic bundles over $S$, and the functions $u$ and $v$ are defined away from their zeroes, that is on $S_0 = \{x \in S \mid (\alpha \beta)_x \neq 0\}$. Define a new couple of functions from $S_0$ to $\R$ by
\[\mu_1 = \frac{u+v}{4}, \mu_2 = \frac{u-v}{2}. \]
They can be interpreted as follows : $\alpha \beta$ is a section of $\canon^4$, which locally admits roots $(\alpha \beta)^{\frac14} \in \cohom^0(\canon)$, which have square norm $e^{\mu_1}$. Similarly, the roots $\alpha^{\frac12} \beta^{-\frac12} \in \cohom^0(\cal{N})$ (which exist locally) have square norm $e^{\mu_2}$. In particular, the equations \eqref{éq:courbures_uv} and Lemma \ref{lemme:courbure_fibré_droites} give
\begin{equation}\label{éq:Hitchin_mu}
	\begin{aligned}
		K = \frac{\courb^{\canon}}{i\omega} &= \frac12 \Delta \mu_1, \\
		\det \II = \frac{\courb^{\cal{N}}}{i\omega} &= \frac12 \Delta \mu_2. 
	\end{aligned}
\end{equation}
The equations \eqref{éq:Hitchin_uv} can be rewritten in terms of the $\mu_k$:
\begin{equation}\label{éq:Hitchin_mu_alt}
	\begin{aligned}
		\Delta \mu_1 &= 4 e^{2\mu_1} \cosh(\mu_2) -2, \\
		\Delta \mu_2 &= 4 e^{2\mu_1} \sinh(\mu_2). 
	\end{aligned}
\end{equation}

\begin{remarque}
	A change of orientation of $\Nor S$ exchanges $u$ and $v$, hence does not change $\mu_1$ but multiplies $\mu_2$ by $-1$: the geometrically relevant quantities are indeed $\mu_1$ and $\abs{\mu_2}$ (see remark \ref{remarque:fonctions_géométriques_axes}) and only the absolute value $\abs{\mu_2}$ will be used. 
\end{remarque}

Recall that we have $K \leq 0$ (Theorem \ref{theoreme:Ishihara}). Moreover, if $K$ vanishes at a point, then $S$ is a Barbot surface (Proposition \ref{proposition:rigidité_couronnes}): we shall in fact show that when $K$ tends to $0$, the geometry of $S$ tends towards that of the Barbot surface. Note already that on the Barbot surface $u$ and $v$ are constant, so equation \eqref{éq:Hitchin_uv} gives $u = v = \ln(\frac12)$. Recall also that for any pointed maximal surface $(S,x) \in \cal{M}_{2,2}^*$, $u(x) = u(S,x)$ and $v(x) = v(S,x)$ are well defined up to permutation (see remark \ref{remarque:changement orientation_normal}). 
\begin{proposition}\label{proposition:courbure_petite_contrôle_uv}
	For every $\epsilon > 0$, there exists $k(\epsilon) > 0$ such that for every pointed complete maximal surface $(S,x) \in \cal{M}_{2,2}^*$ in $\H^{2,2}$ whose sectional curvature $K_S$ satisfies $\abs{K_S(x)} \leq k(\epsilon)$, we have $\abs{u(S,x) - \ln\frac12} \leq \epsilon$ and $\abs{v(S,x) - \ln\frac12} \leq \epsilon$. 
\end{proposition}
\begin{proof}
	For every $(S,x)$ in the space $\mathcal{M}_{2,2}^*$ of pointed complete maximal surfaces in $\H^{2,2}$ (see section \ref{subsection:compacité}), according to remark \ref{remarque:changement orientation_normal}, $(\alpha,\beta) = (\alpha^S,\beta^S)$ is well defined up to permutation. One can therefore define $\phi : \mathcal{M}_{2,2}^* \to \R^2$ by
	\[\phi(S,x) = (\min(\norm{\alpha^S_x}^2,\norm{\beta^S_x}^2),\max(\norm{\alpha^S_x}^2,\norm{\beta^S_x}^2)). \]
	Define $\kappa : \R^2 \to \R,\ (x,y) \mapsto x+y-1$ (compare with the formula $K = e^u + e^v -1$). The map $\phi$ is continuous and invariant under $\SO(2,3)$, hence Corollary \ref{corollaire:compacité_fonction_invariante} asserts that its image is a compact subset of $\R^2$. This compact is contained in $\kappa^{-1}(\R_{\leq 0})$ (Theorem \ref{theoreme:Ishihara}) and intersects $\kappa^{-1}(0)$ only at $(\frac12,\frac12)$ (Proposition \ref{proposition:rigidité_couronnes}). Thus, if a sequence $(S_n,x_n)$ in $\mathcal{M}_{2,2}^*$ satisfies $K(S_n,x_n) = \kappa(\phi(S_n,x_n)) \to 0$, then $\phi(S_n,x_n) \to (\frac12,\frac12)$, which concludes the proof. 
\end{proof}

The aim of this section is to explicitly control the function $k(\epsilon)$ appearing in the statement of Proposition \ref{proposition:courbure_petite_contrôle_uv} ; to this end, let us prove the following technical lemma (compare \cite[Lemma 2.3]{CT90}). Suppose $\mu_1$ bounded on a ball $B(x,r)$ : the formulas \eqref{éq:Hitchin_mu_alt} give an inequality of the type $\Delta \abs{\mu_2} \geq C \sinh(\abs{\mu_2})$ on this ball, which is the hypothesis of the following lemma. Note that in the points where $\mu_2$ cancels, this inequality is only true in the sense of distributions but trivial. 
\begin{lemme}[exponential decay]\label{lemme:décroissance_exponentielle}
	Let $\kappa,c>0$ be real numbers, $M$ a complete simply connected Riemannian $n$-manifold satisfying $\Ric \geq -(n-1) \kappa$, a ball $B(x,r)$ and a function $f : B(x,r)\to \R_{\geq 0}$ of class $\cal{C}^2$, bounded and satisfying the following equation :
	\[(\Delta - c) f \geq 0. \]
	There exists $\alpha > 0$, depending only and explicitly on $(n,\kappa,c)$, such that
	\[f(x) \leq 2 e^{-\alpha r}\ \sup_{B(x,r)} f. \]
\end{lemme}
\begin{proof}
	We may suppose $\sup_{B(x,r)} f = 1$. The function $\rho = d(x,\cdot)$ satisfies (\cite[Lemma 7.1.2]{Petersen16}) :
	\[\Delta \rho \leq (n-1) \sqrt{\kappa} \coth(\sqrt{\kappa} \rho). \]
	We define a barrier function $b = \frac{\cosh(\alpha \rho)}{\cosh(\alpha r)}$, $\alpha > 0$ remaining to be chosen. Since $b = 1$ on $\partial B(x,r)$, we only need to have $(\Delta - c) b \leq 0$ in order to conclude that $b \geq f$ on $B(x,r)$. Now we have
	\begin{align*}
		(\Delta - c) b &= \left( \alpha^2 + \frac{\alpha}{\coth(\alpha \rho)} \Delta \rho - c \right) b \\
		&\leq \left( \alpha^2 + \alpha (n-1) \sqrt{\kappa} \frac{\coth(\sqrt{\kappa} \rho)}{\coth(\alpha \rho)} - c \right) b. 
	\end{align*}
	By imposing $\alpha \leq \sqrt{\kappa}$, we obtain as soon as $\rho > 0$ :
	\[(\Delta - c) b \leq \left( \alpha^2 + \alpha (n-1) \sqrt{\kappa} - c \right) b. \]
	Let us choose $\alpha > 0$ small enough so that the right-hand term is negative : we obtain $(\Delta - c) b \leq 0$, as required. The function $b$ is therefore a barrier : we have $f(x) \leq b(x) = \frac{1}{\cosh(\alpha r)} \leq 2e^{-\alpha r}$. 
\end{proof}

We now wish to apply Lemma \ref{lemme:décroissance_exponentielle} to the functions $f = \mu_k$ (see Proposition \ref{proposition:contrôle_mu}). For this, we shall need a rough control given by the following lemma. 
\begin{lemme}\label{lemme:estimées_mu}
	We have $u,v \leq \ln(\frac23)$ and $\mu_1 = 4(u+v) \leq \frac12 \ln(\frac12)$. 
\end{lemme}
\begin{proof}
	Let us apply the Omori-Yau maximum principle (Theorem \ref{théorème:Omori_Yau} and remark \ref{remarque:Omori_Yau}) to $u$. For every $\epsilon>0$, we thus have at a certain point $x$ such that $u(x) \geq \sup u - \epsilon$ and $\Delta u \geq \epsilon$ :
\[\epsilon \geq 6e^u + 2e^v -4. \]
In particular, using equation \eqref{éq:Hitchin_uv}, we obtain :
\[\epsilon \geq 6e^{\sup u - \epsilon} -4. \]
	Letting $\epsilon$ tend to $0$, we see that $\sup u \leq \ln(\frac23)$. Exchanging $u$ and $v$, we prove the bound on $v$. Similarly, applying the Omori-Yau maximum principle to $\mu_1  = \frac{u+v}{4}$, we obtain for every $\epsilon > 0$ a point $x$ satisfying
	\begin{align*}
		\epsilon &\geq \Delta \mu_1 \\
		&= 4e^{2\mu_1} \cosh(\mu_2) -2 \tag*{by equation \eqref{éq:Hitchin_mu_alt}} \\
		&\geq 4e^{2(\sup \mu_1 - \epsilon)} -2,
	\end{align*}
	thus $\sup \mu_1 \leq \frac12 \ln(\frac12)$ by letting $\epsilon$ tend to $0$. 
\end{proof}

We now wish to quantify the convergence of $S$ towards a Barbot surface when the curvature of $S$ is close to $0$ ; for this purpose, we set the following definition. 
\begin{definition}[exponential decay]\label{définition:décroissance_exponentielle}
	Fix a real number $-\frac13 < k < 0$. For every complete maximal surface $S$ in $\H^{2,2}$, denote $D^k_S = \{x \in S \mid K(x) < k\}$ and define
	\[\rho : \mathcal{M}_{2,2}^* \to \R_+ \cup \{+\infty\},\ (S,x) \mapsto d(x,D^k_S) \]
	the distance to $D^k_S$. We say that a function $f$, defined on a subset of $\mathcal{M}_{2,2}^*$ and with values in $\R$, is \emph{exponentially decaying} if there exist constants $C,\alpha,\rho_0 > 0$ such that $f$ is defined on the set of points $(S,x) \in \mathcal{M}_{2,2}^*$ satisfying $\rho(S,x) \geq \rho_0$ and satisfies on this set
	\begin{equation}\label{éq:décroissance_exponentielle}
		\abs{f(S,x)} \leq C e^{-\alpha \rho}. 
	\end{equation}
\end{definition}
\begin{remarque}\label{remarque:décroissance_exponentielle_partout}
	With the notations of Corollary \ref{corollaire:compacité_fonction_invariante}, suppose that $f$ is a continuous function from $\mathcal{M}_{2,2}^*$ to $\R$, $\O(2,3)$-invariant and exponentially decaying. Then $f$ is bounded by Corollary \ref{corollaire:compacité_fonction_invariante} thus up to modifying $C$, we obtain the inequality $f(x,S) \leq C e^{-\alpha \rho}$ without the hypothesis $\rho \geq \rho_0$. 
\end{remarque}
\begin{remarque}\label{remarque:D_vide}
	If $D^k_S$ is empty (that it $\rho = +\infty$), then $S$ is a Barbot surface (Proposition \ref{proposition:rigidité_couronnes}). In that case, any exponentially decaying function is zero on $S$. 
\end{remarque}
\begin{remarque}
	We will see (see proof of Corollary \ref{corollaire:contrôle_nablaII}) that the sectional curvature $K$ is exponentially decaying. This shows that Definition \ref{définition:décroissance_exponentielle} is independent of the choice of $k$. 
\end{remarque}

\begin{proposition}[control of the $\mu_k$]\label{proposition:contrôle_mu}
	The functions $\tilde{\mu}_1 = \mu_1 - \frac12 \ln(\frac12)$ and $\mu_2$ are exponentially decaying. 
\end{proposition}
\begin{proof}
	We have $\min(u,v) = 2\mu_1 -\abs{\mu_2}$ and Lemma \ref{lemme:estimées_mu} asserts that $\max(u,v) \leq \ln(\frac23)$, thus
	\[-K = 1 -e^u -e^v \geq \frac13 -e^{2\mu_1 -\abs{\mu_2}}. \]
	When $-\frac13 < k < K$, we deduce
	\begin{equation}\label{éq:contrôle_grossier_k_mu}
		\ln(\frac13 + k) \leq 2\mu_1 -\abs{\mu_2}
	\end{equation}
	Let $(S,x)$ be a pointed complete maximal surface of $\H^{2,2}$ with $x$ not belonging to $D_S^k$. 
	On the ball $B(x,\rho(S,x))$, we deduce a rough control of the $\mu_k$ from the inequality $\mu_1 \leq \frac12 \ln(\frac12)$ (Lemma \ref{lemme:estimées_mu}) and equation \eqref{éq:contrôle_grossier_k_mu} :
	\begin{equation}\label{éq:encadrement_grossier_mu}
		\begin{aligned}
			\abs{\mu_2} &\leq \ln(\frac12) - \ln(\frac13 + k), \\
			\frac12 \ln(\frac13 + k) \leq \mu_1 &\leq \frac12 \ln(\frac12). 
		\end{aligned}
	\end{equation}
	Let us insert the first inequality into the equations \eqref{éq:Hitchin_mu_alt}, which gives
	\begin{equation}\label{éq:Delta_mu2}
		\Delta \abs{\mu_2} = 4e^{2\mu_1} \sinh(\abs{\mu_2}) \geq 4C_0 \abs{\mu_2}
	\end{equation}
	with $C_0 = \frac13 + k > 0$. We can therefore apply Lemma \ref{lemme:décroissance_exponentielle} to the function $\abs{\mu_2}$, which gives the desired upper bound for $\abs{\mu_2}$. To control $\tilde{\mu}_1$, define $f = \frac12 \abs{\mu_2} + \tilde{\mu}_1$ and apply the equations \eqref{éq:Hitchin_mu_alt} : we obtain
	\[\Delta f = \frac12 \Delta \abs{\mu_2} - 2e^{-2\tilde{\mu}_1} (\cosh(\mu_2)-1) - 2(e^{-2\tilde{\mu}_1}-1). \]
	From the rough bounds \eqref{éq:encadrement_grossier_mu} and from $\tilde{\mu}_1 \leq 0$, we deduce bounds
	\begin{align*}
		e^{-2\tilde{\mu}_1} (\cosh(\mu_2)-1) &\leq C_1 \mu_2^2, \\
		e^{-2 \tilde{\mu}_1} -1 &\leq -C_2 \tilde{\mu}_1. 
	\end{align*}
	Inserting these bounds and equation \eqref{éq:Delta_mu2} into the expression of $\Delta f$, we find
	\[\frac12 \Delta f \geq C_0 \abs{\mu_2} - C_1 \mu_2^2 + C_2 \tilde{\mu}_1. \]
	We have already shown a formula of the type $\abs{\mu_2} \leq C e^{-\alpha \rho}$ : $\abs{\mu_2}$ is therefore less than $\frac{C_0}{2C_1}$ on a sub-ball $B(x,\rho(S,x)-\rho_0)$ (with $\rho_0$ depending only on the previous constants). On $B(x,\rho(S,x)-\rho_0)$, we thus obtain an inequality of the type $\Delta f \geq C_3 f$. The Lemma \ref{lemme:décroissance_exponentielle} then shows the exponential decay of $f$, thus that of $\tilde{\mu}_1 = f - \frac12 \mu_2$. 
\end{proof}

We now use the Bochner formula to control the derivative of the second fundamental form.
\begin{proposition}[see {\cite[Proposition 4.2]{MT25}}, and {\cite[section 5.1]{LT22}} for $\H^{2,n}$]\label{proposition:Bochner_II}
	Let $M$ be a complete maximal $p$-submanifold of $\H^{p,q}$, whose second fundamental form $\II$ is viewed as a $1$-form with values in $\Omega^1(\Nor M)$, and $\Scal$ the scalar curvature (equal to the sectional curvature $K$ if $p=2$). We have
	\[\frac12 \Delta \norm{\II}^2 + 2p \Scal \geq \norm{\nabla \II}^2. \]
\end{proposition}

\begin{corollaire}[control of $\nabla \II$]\label{corollaire:contrôle_nablaII}
	The function $\norm{\nabla \II}$ is exponentially decaying (see Definition \ref{définition:décroissance_exponentielle}), where $\II$ is considered as a section of $\Tan^S \otimes \Tan^ S \otimes \Nor S$ (or, equivalently, as a $1$-form with values in $\Tan^* S \otimes \Nor S$).
\end{corollaire}
\begin{proof}
	We show successively that the following functions are exponentially decaying (see Definition \ref{définition:décroissance_exponentielle}).
	\begin{itemize}
		\item The functions $\tilde{\mu}_1$ and $\mu_2$: this is Proposition \ref{proposition:contrôle_mu}.
		\item Their Laplacians $\Delta \mu_1 = 2K$ and $\Delta \mu_2$ by the equations \eqref{éq:Hitchin_mu_alt}.
		\item Their derivatives $\norm{\nabla \tilde{\mu}_1}$ and $\norm{\nabla \mu_2}$ by Proposition \ref{proposition:estimées_W2p} (elliptic estimates and Sobolev inequality).
		\item The Laplacian $\Delta \norm{\II}^2$, which is equal by the Gau\ss\ equation \eqref{éq:Gauss} to $\Delta K = (\Delta u + \norm{\nabla u}^2) e^u + (\Delta v + \norm{\nabla v}^2) e^v$.
	\end{itemize}
	We can then conclude using Proposition \ref{proposition:Bochner_II}, which bounds $\norm{\nabla \II}^2$ by quantities whose exponential decay has just been shown.
\end{proof}

\section{Upper bound on the volume}
\label{section:borne_supérieure}

We establish here two results. The first is Proposition \ref{proposition:décroissance_volume}, which states that the function $\bar{\bf{V}}$, which associates to a point of a maximal surface $S$ of $\H^{2,2}$ the volume of the slice of $\Conv(S)$ above this point, is exponentially decaying (see Definition \ref{définition:volume_tranche}). The second is Theorem \ref{theoreme:borne_sup}, which gives an upper bound on the integral of an exponentially decaying function and from which we deduce Theorem \ref{théorème:borne_sup_cocompacte}.

\subsection{Volume above a point}
\label{subsection:contrôle_volume}

The proof of Proposition \ref{proposition:décroissance_volume} goes as follows. Given a point $x$ at large distance $\rho = \rho(S,x)$ from the set $D^k_S$ (which is equivalent to $K(x)$ being close to $0$), we find a unit vector to which the plane $\Tan_y S$ remains almost orthogonal for $y$ in the ball $B(x,\rho)$ (Lemma \ref{lemme:décroissance_angles}). Thus, the convex hull of $S$ remains close to the orthogonal of this vector on this ball. It is therefore almost contained in a totally geodesic subspace of $\H^{2,2}$, hence an upper bound on $\bar{\bf{V}}$.

Let us begin by defining this unit vector. We can consider the second fundamental form $\II$ of $S$ as a bundle morphism from $\Sym^2_0(\Tan S)$ to $\Nor S$ (see section \ref{subsection:équations_fondamentales}). These bundles being of dimension $2$ and endowed with Riemannian metrics, this bundle morphism admits a polar decomposition. In particular, if $\II_x : \Sym^2_0(\Tan_x S) \to \Nor_x S$ is not a conformal map, then the set of unit vectors of $\Sym^2_0(\Tan_x S)$ is sent to an ellipse of $\Nor_x S$ (possibly degenerate), which has a major axis $A = \II(\sf{V},\sf{V})$ and a minor axis $a = \II(\sf{V},J\sf{V})$ (well defined up to sign), these formulas defining $\sf{V} \in \Tan^1_x S$ (up to replacement by $J\sf{V}$, $-\sf{V}$ or $-J\sf{V}$).

\begin{lemme}\label{lemme:longueurs_axes}
	The axes $A,a$ of $\II$ satisfy $\sqrt2 \norm{A} = e^\frac{u}{2} + e^\frac{v}{2}$ and $\sqrt2 \norm{a} = \abs{e^\frac{u}{2} - e^\frac{v}{2}}$. 
\end{lemme}
\begin{proof}
	Let us note that the image of the unit circle by $\II : \Sym^2_0(\Tan_x S) \to \Nor S$ corresponds to the set of $\II(\sf{u},\sf{u})$ with $\sf{u} \in \Tan_x^1 S$ unit tangent (where this time we consider $\II$ as a morphism from $\Tan S \otimes \Tan S$ into $\Nor S$). Let $\sf{u} \in \Tan^1_x S$ be a unit tangent vector. We define unit vectors $\partial_z = \frac12 (\sf{u} - iJ\sf{u}) \in \Tan^{(1,0)}_x S$ and $\partial_{\bar{z}} = \frac12 (\sf{u} + iJ\sf{u}) \in \Tan^{(0,1)}_x S$. In particular, the decomposition $\II^\C(\partial_z,\partial_z) = \alpha_0 + \beta_0$ (with $\alpha_0 \in \cal{N}$ and $\beta_0 \in \cal{N}^{-1}$) satisfies $\norm{\alpha_0} = e^u$ and $\norm{\beta_0} = e^v$. Let us denote, for all $a \in \C$, $\sf{t}_a = a \partial_z + \bar{a} \partial_{\bar{z}}$ : $\Tan_x S$ is the set of $\sf{t}_a$, for $a$ in $\C$. The equality $h(\sf{t}_a,\sf{t}_a) = \abs{a}^2$ shows that the unit tangent $\Tan^1_x S$ is the set of $\sf{t}_u$, for $u$ in the set $\bb{U}$ of complex numbers of norm one. For $u$ in $\bb{U}$, denoting by $h$ the hermitian metric of $(\Nor S)^\C$, we have
	\begin{align*}
		&2 \norm{\II(\sf{t}_u,\sf{t}_u)}^2 = \frac12 \norm{u^2 \II(\partial_z,\partial_z) + \bar{u}^2 \II(\partial_{\bar{z}},\partial_{\bar{z}})}^2 \tag*{since $\II(\partial_z,\partial_{\bar z}) = 0$} \\
		&= h(u^2 \alpha_0 + \bar{u}^2 \beta_0, u^2 \alpha_0 + \bar{u}^2 \beta_0) \\
		&= \norm{\alpha_0}^2 + \norm{\beta_0}^2 + u^4 h(\alpha_0,\beta_0) + \bar{u}^4 h(\beta_0,\alpha_0) \\
		&= \norm{\alpha_0}^2 + \norm{\beta_0}^2 + 2 \Re \left[ u^4 h(\alpha_0,\beta_0) \right]. 
	\end{align*}
	By definition of the axes $A$ and $a$, the maximum (for $u \in \bb{U}$) of the previous expression is $2 \norm{A}^2$ and its minimum is $2 \norm{a}^2$. These extrema are equal to $\norm{\alpha_0}^2 + \norm{\beta_0}^2 \pm 2 \norm{\alpha_0} \norm{\beta_0}$, hence
	\begin{align*}
		2\norm{A}^2 &= \left[ \norm{\alpha_0}^2 + \norm{\beta_0}^2 + 2 \norm{\alpha_0} \norm{\beta_0} \right] = (e^\frac{u}{2} + e^\frac{v}{2})^2 \tag*{and} \\
		2\norm{a}^2 &= \left[ \norm{\alpha_0}^2 + \norm{\beta_0}^2 - 2 \norm{\alpha_0} \norm{\beta_0} \right] = (e^\frac{u}{2} - e^\frac{v}{2})^2, 
	\end{align*}
	as desired. 
\end{proof}

\begin{remarque}\label{remarque:fonctions_géométriques_axes}
	We deduce from Lemma~\ref{lemme:longueurs_axes} that if the axes have the same length at a point $x$, then $u$ or $v$ is not defined : morally, we have $e^u = 0$ or $e^v = 0$. In particular, the sectional curvature satisfies $K(x) = -1 + e^u + e^v \leq -\frac13$, hence $x$ is in $D^k$ : the axes are therefore well defined outside $D^k$. Let us also note that we could define $\mu_1$ and $\abs{\mu_2}$ more geometrically via the equalities $\norm{A} = \sqrt2 e^{\mu_1} \cosh(\frac{\mu_2}{2})$ and $\norm{a} = \sqrt2 e^{\mu_1} \sinh(\frac{\abs{\mu_2}}{2})$. 
\end{remarque}

From now on, let un consider a ball $B(x_0,R)$ of $S$ not intersecting $D^k$. This ball is simply connected (since $K \leq 0$ by Theorem~\ref{theoreme:Ishihara}) and $\II$ cannot be conformal on any point of $B(x_0,R)$ by remark~\ref{remarque:fonctions_géométriques_axes}. The curvature satisfies $K = -1 + \norm{A}^2 + \norm{a}^2$ (Lemma~\ref{lemme:longueurs_axes} and equation \eqref{éq:courbures_uv}) as well as $K \geq k \geq -\frac13$ on $B(x_0,R)$ and we have $\norm{A} \geq \norm{a}$, hence on this ball we have
\begin{equation}\label{éq:borne_inf_A}
	\norm{A}^2 \geq \frac13. 
\end{equation}
We can therefore choose a vector field $\sf{V}$ on $B(x_0,R)$ such that at each point, $\II(\sf{V},\sf{V}) = A$ is the major axis and $\II(\sf{V},J\sf{V}) = a$ the minor axis (both being well defined up to sign). The idea of this section is to show that if the curvature $K$ is close to $0$ around $x_0$, then $S$ remains close to the orthogonal of $a_{x_0}$ (minor axis at $x_0$) around $x_0$. For this, let us lift $S$ to a complete maximal surface in $\Hc^{2,2}$ still denoted by $S$ (this is possible according to Lemma~\ref{lemme:relèvement_unique}) and introduce the concept of \emph{angle}. 

For any definite vector subspace $E$ of $\R^{p,q}$, let us denote $\epsilon(E) = +1$ if $E$ is positive, $\epsilon(E) = -1$ if $E$ is negative. Let $E$ and $F$ be two definite vector subspaces of $\R^{p,q}$ and $(e_i)$ an orthonormal basis of $E$. Given a morphism $f \in \Hom(E,F)$, we define its adjoint $f^* : F \to E$ by $\scal{f(u)}{v} = \scal{u}{f^*(v)}$, as well as its norm :
\[\norm{f}_2^2 = \sum_i \norm{f(e_i)}^2 = \epsilon(E) \epsilon(F) \tr(f \circ f^*). \]
The second equality shows $\norm{f}_2 = \norm{f^*}_2$. Let us denote by $\pi_{E \to F} : E \to F$ the orthogonal projection. We define the \emph{angle} $\theta(E,F) \in \R_+$ between $E$ and $F$ by
\[\theta(E,F)^2 = \norm{\pi_{E \to F}}_2^2. \]
We have $\pi_{E \to F}^* = \pi_{F \to E}$, hence $\theta(E,F) = \theta(F,E)$. Given two definite vector subspaces $E_1,E_2$ of $\R^{p,q}$ and vectors $v_1 \in E_1$, $v_2 \in E_2$, we have by definition:
\begin{equation}\label{éq:angles_scalaire_projections}
	\abs{\scal{\pi_{E_1 \to F}(v_1)}{\pi_{E_2 \to F}(v_2)}} \leq \theta(E_1,F) \theta(E_2,F) \norm{v_1} \norm{v_2}. 
\end{equation}
\begin{lemme}\label{lemme:angles_somme}
	Let $E_1, \dots, E_n$ and $F$ be definite subspaces of $\R^{p,q}$ such that the $E_i$ are pairwise orthogonal and satisfy $\bigoplus_{i=1}^n E_i = \R^{p,q}$. We have
	\begin{equation}
		\epsilon(F) \dim(F) = \sum_i \epsilon(E_i)\, \theta(E_i,F)^2. 
	\end{equation}
\end{lemme}
\begin{proof}
	Let $(f_i)$ be an orthonormal basis of $F$. We have
	\begin{align*}
		\epsilon(F) \dim(F) &= \sum_i \scal{f_i}{f_i} \\
		&= \sum_{i,j} \scal{\pi_{F \to E_j} (f_i)}{\pi_{F \to E_j} (f_i)} \tag*{since $\stackrel{\perp}{\bigoplus} E_j = \R^{p,q}$} \\
		&= \sum_j \epsilon(E_j) \theta(F,E_j)^2 \tag*{by definition,}
	\end{align*}
	whence the result given the symmetry $\theta(E_i,F) = \theta(F,E_i)$. 
\end{proof}
Let us fix from now on $F = \Vect{J A_{x_0}}$, equal to $\Vect{a_{x_0}}$ if $a_{x_0}$ is non-zero, and let us denote $\pi : \R^{2,3} \to F$ the orthogonal projection. Let $\gamma : \R \to B(x_0,\frac{R}{2})$ be a path such that $\gamma(0) = x_0$, parametrised by arc length ($\norm{\dot{\gamma}} = 1$). We will apply the notion of angle to compare $F$ with one of the definite subspaces $\Vect{\gamma(t)}$, $\Tan_{\gamma} S$, $\Vect{A_{\gamma(t)}}$ and $\Vect{J A_{\gamma(t)}}$ (equal to $\Vect{a_{\gamma(t)}}$ if $a_{\gamma(t)}$ is non-zero). This defines functions on $\R$:
\begin{equation*}
	\left\{
	\begin{aligned}
		\theta_{\gamma}(t) &= \theta(F, \Vect{\gamma(t)}) \\
		\theta_{\Tan}(t) &= \theta(F, \Tan_{\gamma(t)} S) \\
		\theta_A(t) &= \theta(F, \Vect{A_{\gamma(t)}}) \\
		\theta_a(t) &= \theta(F, \Vect{J A_{\gamma(t)}})
	\end{aligned}
	\right.
\end{equation*}

\begin{lemme}\label{lemme:inégalité_différentielle_theta}
There exist constants $R_0, C, \alpha \in \R_{>0}$ (depending only on $k$ in Definition \ref{définition:décroissance_exponentielle}) such that if $R \geq R_0$, then we have
\begin{equation*}
	\left\{
	\begin{aligned}
		& \left| \theta_\gamma' \right| \leq \theta_{\Tan} &&\text{when } \theta_\gamma > 0 \\
		& \left| \theta_{\Tan}' \right| \leq \theta_\gamma + \norm{A} \theta_{A} + \norm{a} \theta_{a} &&\text{when } \theta_{\Tan} > 0 \\
		& \left| \theta_{A}' \right| \leq \theta_{\Tan} + (\theta_a + \theta_A) C e^{-\alpha R} &&\text{when } \theta_{A} > 0
	\end{aligned}
	\right.
\end{equation*}
\end{lemme}

\begin{proof}
Recall (see section \ref{subsection:geom_projective_hpq}) that for $\sf{v} \in \Tan_x \Hc^{p,q}$ and $X$ a vector field on $\Hc^{p,q}$ in a neighbourhood of $x$, the connections $\nabla^\R$ of $\R^{p,q+1}$ and $\nabla^\H$ of $\Hc^{p,q}$ satisfy $\nabla^{\R}_\sf{v} X = \scal{\sf{v}}{X} x + \nabla^{\H}_\sf{v} X$ (equation \eqref{éq:Hpq_ombilique}). If moreover $x \in S$ and $v \in \Tan_x S$, then in the decomposition $\Tan_x \Hc^{p,q} = \Tan_x S \oplus \Nor_x S$, we have
\[ \nabla^{\H}_\sf{v} = 
\begin{pmatrix}
	\nabla^{\Tan}_\sf{v} & B(\sf{v}, \cdot)\\
	\II(\sf{v},\cdot) & \nabla^{\Nor}_\sf{v}
\end{pmatrix}. \]
By definition, we have $(\theta_{\gamma})^2 = -\scal{\pi(\gamma)}{\pi(\gamma)}$ (where $\pi$ still denotes the orthogonal projection on $F$). Differentiating this expression gives
\begin{align*}
	\dt[(\theta_{\gamma})^2] &= -2 \scal{\pi(\nabla^{\R}_{\dot \gamma} \gamma)}{\pi(\gamma)} \\
	&= -2 \scal{\pi(\dot \gamma)}{\pi(\gamma)}
\end{align*}
thus formula \eqref{éq:angles_scalaire_projections} gives the desired inequality. For $\theta_{\Tan}'$, same idea: fix an orthonormal basis of $\Tan_{\gamma(0)} S$, then transport it along $\gamma$ thanks to the connection $\nabla^{\Tan}$ on the bundle $\Tan S$. We obtain a basis $(\rm{t}_i(t))$ of $\Tan_{\gamma(t)} S$, which allows writing the definition $\theta_{\Tan}^2 = -\sum_i \scal{\pi(\rm{t}_i)}{\pi(\rm{t}_i)}$. From now on, we omit the indices : for exemple, $A_{\gamma(t)}$ will be simply denoted by $A$. Let us differentiate, taking into account $\nabla^{\R}_{\dot{\gamma}} \sf{t} = \scal{\dot{\gamma}}{\sf{t}} \gamma + \II(\sf{t},\dot{\gamma})$ :
\begin{align}
	\dt[(\theta_{\Tan}^2)] &= -2 \sum_i \scal{\pi(\nabla^{\R}_{\dot \gamma} \rm{t}_i)}{\pi(\rm{t}_i)} \nonumber \\
	&= -2 \sum_i \scal{\dot \gamma}{\rm{t}_i} \scal{\pi(\gamma)}{\pi(\rm{t}_i)} -2 \sum_i \scal{\pi(\II(\dot \gamma,\rm{t}_i))}{\pi(\rm{t}_i)} \nonumber \\
	&= -2 \scal{\pi(\gamma)}{\pi(\dot \gamma)} -2 \sum_i \scal{\pi(\II(\dot \gamma,\rm{t}_i))}{\pi(\rm{t}_i)} \label{éq:inégalité_différentielle_theta_T}
\end{align}
To control the second term, set $\II(\sf{t}_i,\dot{\gamma}) = u_i \frac{A}{\norm{A}} + v_i \frac{a}{\norm{a}}$, where $\abs{u_i} \leq \norm{A}$ and $\abs{v_i} \leq \norm{a}$, so that
\begin{align*}
	\sum_i \scal{\pi(\II(\dot \gamma,\rm{t}_i))}{\pi(\rm{t}_i)} &= \sum_i \scal{\pi(u_i \frac{A}{\norm{A}})}{\pi(\rm{t}_i)} + \sum_i \scal{\pi(v_i \frac{a}{\norm{a}})}{\pi(\rm{t}_i)} \\
	&= \sum_i u_i \scal{\pi(\frac{A}{\norm{A}})}{\pi(\rm{t}_i)} + \sum_i v_i \scal{\pi(\frac{a}{\norm{a}})}{\pi(\rm{t}_i)}. 
\end{align*}
Taking into account that $\abs{u_i} \leq \norm{A}$, $\abs{v_i} \leq \norm{a}$ and using inequality \eqref{éq:angles_scalaire_projections}, we find :
\[\abs{\sum_i \scal{\pi(\II(\dot \gamma,\rm{t}_i))}{\pi(\rm{t}_i)}} \leq \norm{A} \theta_A \theta_{\Tan} + \norm{a} \theta_a \theta_{\Tan}. \]
Rewrite equation \eqref{éq:inégalité_différentielle_theta_T} using this upper bound, controlling the first term with inequality \eqref{éq:angles_scalaire_projections} : we obtain
\[\abs{\theta_{\Tan} \theta_{\Tan}'} \leq \theta_{\Tan} \theta_{\gamma} + \norm{A} \theta_A \theta_{\Tan} + \norm{a} \theta_a \theta_{\Tan}, \]
hence the announced inequality on $\theta_{\Tan}'$. We bound $\theta_A'$ and $\theta_a'$ using the following lemma. 

\begin{lemme}\label{lemme:contrôle_direction_A}
	The function $\norm{\nabla^{\Tan} \sf{V}}$ is exponentially decaying. 
\end{lemme}
\begin{proof}
	Recall that the axes are given by $\II(\sf{V},\sf{V}) = A$ (major axis) and $\II(\sf{V},J\sf{V}) = a$ (minor axis). Let $\nabla^{\Tan} \sf{V} = \omega \otimes J(\sf{V})$ with $\omega \in \Omega^1(U)$. We have
	\begin{align*}
		0 &= \nabla \scal{\II(\sf{V},\sf{V})}{\II(\sf{V},J\sf{V})} \\
		&= \scal{(\nabla \II)(\sf{V},\sf{V})}{a} + 2\scal{\II(\sf{V},\nabla \sf{V})}{a} \\
		&+ \scal{A}{(\nabla \II)(\sf{V},J\sf{V})} + \scal{A}{\II(\nabla \sf{V},J\sf{V})} + \scal{A}{\II(\sf{V},\nabla (J\sf{V}))} \\
		&= \scal{(\nabla \II)(\sf{V},\sf{V})}{a} + 2\omega \scal{a}{a} \\
		&+ \scal{A}{(\nabla \II)(\sf{V},J\sf{V})} + \omega \scal{A}{\II(J\sf{V},J\sf{V})} + \omega \scal{A}{\II(\sf{V},J^2 \sf{V})} \tag*{since $\nabla J = 0$ and $\nabla \sf{V} = \omega \cdot J\sf{V}$} \\
		&= \scal{(\nabla \II)(\sf{V},\sf{V})}{A} + \scal{A}{(\nabla \II)(\sf{V},J\sf{V})} + 2\omega \norm{a}^2 - 2\omega \norm{A}^2 \tag*{since $J^2 = -\Id$ and $\II(J\sf{u},\sf{v}) = -\II(\sf{u},J\sf{v})$, hence}\\
		\omega &= \frac{\scal{(\nabla \II)(\sf{V},\sf{V})}{a} + \scal{A}{(\nabla \II)(\sf{V},J\sf{V})}}{\norm{A}^2 -\norm{a}^2}.
	\end{align*}
	Taking into account $\norm{A}^2 -\norm{a}^2 = 2e^{u+v}$, we obtain
	\begin{align*}
		\norm{\nabla \sf{V}} &\leq \frac{\norm{\nabla \II}_\infty (\norm{A} + \norm{a})}{2e^{u+v}}. 
	\end{align*}
	Corollary \ref{corollaire:contrôle_nablaII} asserts that $\nabla \II$ is exponentially decaying. We also control $e^{u+v}$ (which tends to $\frac14$ for large $\rho$) and by equation \eqref{éq:borne_inf_A}, we have $\frac{1}{\sqrt3} \leq \norm{A} \leq \norm{A} + \norm{a} \leq 2$, which suffices to conclude that $\nabla \sf{V}$ is exponentially decaying. 
\end{proof}

From the previous lemma, Corollary \ref{corollaire:contrôle_nablaII} (exponential decay of $\norm{\nabla \II}$) and the fact that the image of $\gamma$ is contained in the ball $B(x_0,\frac{R}{2})$ (thus remains at distance at least $\frac{R}{2}$ from $D^k$), we deduces that there exist constants $R_0, C, \alpha \in \R_{>0}$ (depending only on $k$ in Definition \ref{définition:décroissance_exponentielle}) such that if $R \geq R_0$, then we have at the point $\gamma(t)$ (for any $t$) :
\begin{align}
	\norm{\nabla^{\Nor} A} &= \norm{\nabla^{\Nor} (\II(\sf{V},\sf{V}))} \nonumber \\
	&\leq \norm{\nabla \II} + 2 \norm{\II}_\infty \norm{\nabla^{\Tan} \sf{V}} \nonumber \\
	&\leq C e^{-\alpha R} \label{éq:borne_nabla_A}. 
\end{align}
We now turn to the control of
\begin{align*}
	\dt[\scal{\pi(A)}{\pi(A)}] &= 2 \scal{\pi(B(\dot \gamma,A))}{\pi(A)} +2 \scal{\pi(\nabla^{\Nor}_{\dot \gamma} A)}{\pi(A)}. 
\end{align*}
By Theorem \ref{theoreme:Ishihara} and the correspondence between $\II$ and $B$, we have $\norm{B}_\infty \leq 1$. As moreover $\norm{\dot \gamma} = 1$, inequality \eqref{éq:angles_scalaire_projections} gives:
\[\abs{\scal{\pi(B(\dot \gamma,A))}{\pi(A)}} \leq \theta_{\Tan} \theta_A \norm{A}^2. \]
Moreover, we have
\begin{align}
	\abs{\scal{\pi(\nabla^{\Nor}_{\dot \gamma} A)}{\pi(A)}} &\leq \theta_A \theta_a \norm{\nabla^{\Nor}_{\dot \gamma} A} \tag*{thus by \eqref{éq:borne_nabla_A}} \nonumber \\
	&\leq \theta_A \theta_a C e^{-\alpha R} \tag*{thus finally} \nonumber \\
	\abs{\dt[\scal{\pi(A)}{\pi(A)}]} &\leq 2\theta_A [\theta_{\Tan} \norm{A} + \theta_a C e^{-\alpha R}]. \label{éq:angles_contrôle_A}
\end{align}
Finally control $\theta_A$: by definition, $\theta_A^2 = \frac{\scal{\pi(A)}{\pi(A)}}{\scal{A}{A}}$, thus
\begin{align*}
	\dt[(\theta_A^2)] &= \frac{1}{\norm{A}^2} \dt[\scal{\pi(A)}{\pi(A)}] - 2 \scal{\pi(A)}{\pi(A)} \frac{\scal{\nabla^{\Nor} A}{A}}{\norm{A}^4}. 
\end{align*}
On the ball $B(x_0,R)$, we have seen (equation \eqref{éq:borne_inf_A}) that we have $\frac{1}{\sqrt3} \leq \norm{A} \leq 1$, thus $\frac{\scal{\nabla^{\Nor} A}{A}}{\norm{A}^4}$ is exponentially decaying by \eqref{éq:borne_nabla_A}. Equation \eqref{éq:angles_contrôle_A} controls the first term, which gives the expected control on $\theta_A$. This concludes the proof of Lemma \ref{lemme:inégalité_différentielle_theta}. 
\end{proof}

We now deduce an upper bound on the angles $\theta_A$. Let $u = \theta_\gamma + \theta_{\Tan} + \theta_A$, $\epsilon = 2C e^{-\alpha R}$ (which should be thought of as close to $0$) and $d = \sqrt{1-\epsilon^2}$. 
\begin{lemme}\label{lemme:décroissance_angles}
	There exists a constant $R_0$ (depending only on $k$ in Definition \ref{définition:décroissance_exponentielle}) such that if $R \geq R_0$, then for any $t \geq 0$ satisfying $e^{2dt} \leq \frac{1+d}{1-d}$, we have
	\[u(t) \leq \frac{1}{\epsilon} \frac{e^{2dt} -1}{\frac{1}{1-d} - \frac{e^{2dt}}{1+d}}. \]
\end{lemme}
\begin{proof}
	Lemma \ref{lemme:longueurs_axes} (which expresses $\norm{A} \leq 1$ and $\norm{a}$ in terms of the $\mu_k$) and Proposition \ref{proposition:contrôle_mu} (which states that the $\mu_k$ decay exponentially) show that $\norm{a}$ decays exponentially: on the ball $B(x_0,\frac{R}{2})$, we have
	\[\norm{a} \leq C e^{-\alpha R}. \]
	Take $R_0$ large, so that for $R \geq R_0$ we have $\epsilon < 1$. The previous lemma (Lemma \ref{lemme:inégalité_différentielle_theta}) then allows controlling the derivative of $u$ :
	\[\abs{u'} \leq 2u + \epsilon \theta_a. \]
	Note furthermore that by Lemma \ref{lemme:angles_somme}, we have
	\begin{align*}
		(1+\theta_a) \abs{1-\theta_a} &= \abs{1-\theta_a^2} \\
		&= \abs{\theta_{\Tan}^2 - \theta_\gamma^2 - \theta_A^2} \\
		&\leq u^2 \tag*{thus as $1+\theta_a \geq 1$,} \\
		\abs{1-\theta_a} &\leq u^2. 
	\end{align*}
	We finally obtains an inequality only in terms of $u$: $\abs{u'} \leq 2u + \epsilon (1+u^2)$. Since $u(0) = 0$, $u$ is bounded, for $t \geq 0$, by the function $u_m$ satisfying:
	\[u_m(0) = 0, \quad u_m' = 2u_m + \epsilon(1+u_m^2). \]
	This is a Riccati differential equation, which we now recall how to solve. $v = \epsilon u_m$ satisfies $v' = 2v + \epsilon^2 + v^2$. Chooses a function $\phi$ such that $v = -\frac{\phi'}{\phi}$. We obtain $v' = -\frac{\phi''}{\phi} + v^2$, so that
	\[\frac{\phi''}{\phi} = v^2 - v' = -2v + \epsilon^2 = 2\frac{\phi'}{\phi} + \epsilon^2, \]
	which gives a linear differential equation:
	\[\phi'' -2\phi' + \epsilon^2 \phi = 0. \]
	The solutions of this equation are given by the functions
	\[\phi : t \mapsto \lambda e^{(1+d)t} + \mu e^{(1-d)t}, \]
	for all real $\lambda$ and $\mu$. Taking into account the condition $v(0) = 0$, one can express $v$ explicitly. The inequality $\epsilon u \leq v$ then rewrites as the inequality of Lemma \ref{lemme:décroissance_angles}. 
\end{proof}

We now deduce the exponential decay of the volume $\bar{\bf{V}}$ above a point (see Definition \ref{définition:volume_tranche}). 
\begin{proposition}\label{proposition:décroissance_volume}
	$\bar{\bf{V}}$ is exponentially decaying. 
\end{proposition}
\begin{proof}
	Given Definition \ref{définition:décroissance_exponentielle}, we must show that there exist constants $R_0,C,\alpha$ (depending only on $k$ in Definition \ref{définition:décroissance_exponentielle}) such that for any ball $B(x,R)$ not intersecting $D^k$ and satisfying $R \geq R_0$, we have $V(x) \leq C e^{-\alpha R}$ ; thus we can assume $R$ large. The vector $a^1_0 = \frac{J A_x}{\norm{A_x}}$ is unitary and parallel to the minor axis $a_x$ ; for any real $s$, let $a^1_s = \sin(s)\, x + \cos(s)\, a^1_0$. This is a point in $\Hc^{2,2}$, which defines a half-space
	\[E_s = \{y \in \Hc^{2,2} \mid \scal{y}{a^1_s} \leq 0\}. \]
	We shall show (step \ref{étape:contrôle_loin}) that for a suitably chosen $s_1>0$ (roughly equal to $e^{-\alpha R}$ for some $\alpha > 0$), all $y \in S$ at distance $d(x,y) \geq t_0 = t_0(R)$ from $x$ lie in $E_{s_1} \cap E_{\pi-s_1}$, for some $t_0$ verifying $t_0 \stackrel{R \to \infty}{\sim} \frac14 \alpha R$. 

We shall then show (step \ref{étape:contrôle_près}) that for a suitably chosen $s_2>0$ (roughly equal to $e^{-\alpha' R}$ for some $\alpha' > 0$), all $y \in S$ at distance $d(x,y) \leq t_0$ from $x$ lie in $E_{s_2} \cap E_{\pi-s_2}$. Note that the map $s \mapsto E_s \cap E_{\pi-s}$ is increasing for inclusion (for $0 \leq s < \frac{\pi}{2}$) : indeed, for all $y \in \H^{2,2}$ we have
\begin{align*}
	&y \in E_s \cap E_{\pi-s} \\
	\Leftrightarrow &\scal{y}{\pm \cos(s)\, a^1_0 + \sin(s)\, x} \leq 0 \\
	\Leftrightarrow &\abs{\scal{y}{a^1_0}} \leq -\tan(s) \scal{y}{x},
\end{align*}
and this condition indeed becomes less restrictive as $s$ increases. Combining steps \ref{étape:contrôle_loin} and \ref{étape:contrôle_près} and setting $s = \max(s_1,s_2)$, we thus obtain
\[S \subset E_s \cap E_{\pi-s}. \]
The convex hull $\Conv(S)$ is hence also contained in the intersection of half-spaces $E_s \cap E_{\pi-s}$. Let now a unit normal vector $\sf{n} \in \Nor_x S$ and $d \geq 0$ satisfying $d \sf{n} \in \Nor_x^0 S$ (hence $\exp_x(d\sf{n}) \in \Conv(S)$, see Definition \ref{définition:décroissance_exponentielle}). By definition of $s$, we have
\begin{align*}
	0 &\geq \scal{\exp_x(d\sf{n})}{a^1_s} \\
	&= \scal{\cos(d)\, x + \sin(d)\, n}{\cos(s)\, a^1_0 + \sin(s)\, x} \\
	&= -\cos(d) \sin(s) + \sin(d) \cos(s) \scal{\sf{n}}{a^1_0}, \tag*{thus} \\
	&\scal{\sf{n}}{a^1_0} \tan(d) \leq \tan(s). 
\end{align*}
Replacing $a^1_s$ by $a^1_{\pi-s}$, we similarly obtain :
\[\scal{\sf{n}}{a^1_0} \tan(d) \leq -\tan(s). \]
The two preceding inequalities show that if we write $\sf{n} \in \Nor_x^0 S$ in the form
\[\sf{n} = u_1 \frac{A_x}{\norm{A_x}} + u_2 J\frac{A_x}{\norm{A_x}}, \]
then we have $\abs{u_2} \leq s$. Moreover, since $\Nor_x^0 S$ is contained in the ball centred at $0$ of radius $\frac{\pi}{2}$ (Lemma \ref{lemme:épaisseur_pi2}), we have $\abs{u_1} \leq \frac{\pi}{2}$ : thus
\[\bar{\bf{V}}_S(x) = \Vol \Nor_x^0 S \leq \frac{\pi}{2} s. \]
Since we have an inequality of the type $s = \max(s_1,s_2) \leq Ce^{-\alpha R}$, this concludes the proof of Proposition \ref{proposition:décroissance_volume}. 

\etape[control far from $x$]\label{étape:contrôle_loin}
For any real $s$ satisfying $-\frac{\pi}{2}<s<\frac{\pi}{2}$, the set of points of $\Hc^{2,2}$ connected to $a^1_s$ by a spacelike geodesic (see section \ref{subsection:définition_Hpq}) is
\[\Omega_s = \{y \in \Hc^{2,2} \mid \scal{y}{a^1_s} < -1\}. \]
We shall show that for a suitably chosen $t_0 \leq R$ and $s>0$ (with $s$ close to $0$), all $y \in S$ at distance $t_0 = d(x,y)$ from $x$ lie in $\Omega_s$. Lemma \ref{lemme:cône_1_Lipschitz} then shows that all $y \in S$ satisfying $d(x,y) \geq t_0$ lie in this set, and thus in particular satisfy $\scal{y}{a^1_s} < 0$, which will conclude step \ref{étape:contrôle_loin}. Let us therefore find such $s$ and $t_0$. We can thus assume $R$ large. We have
\begin{align*}
	\scal{a^1_s}{y} &= \cos(s) \scal{\frac{J A_x}{\norm{A_x}}}{y} + \sin(s) \scal{x}{y}.
\end{align*}
Given two points $x$ and $y$ of $S$, we have $\scal{x}{y} \leq -\frac12 e^{d(x,y)}$ (Lemma \ref{corollaire:sous-variété_simplement_connexe}). We seek to apply Lemma \ref{lemme:décroissance_angles}: this is possible for $0 \leq d(x,y) \leq \frac{1}{2d} \ln(\frac{1+d}{1-d})$. From the definition $\epsilon = 2C e^{-\alpha R}$, we deduce the equivalents
\begin{align}
	1-d = 1-\sqrt{1-\epsilon^2} &\stackrel{R \to \infty}{\sim} \frac{\epsilon^2}{2}, \label{éq:équivalent_d} \\
	\frac{1}{2d} \ln(\frac{1+d}{1-d}) &\stackrel{R \to \infty}{\sim} \frac12 \ln(\frac{1}{\epsilon^2}) \stackrel{R \to \infty}{\sim} \alpha R. \nonumber
\end{align}
In particular, for $R$ sufficiently large, if the distance $d(x,y) = t$ between a point $y \in S$ and $x$ satisfies $t \leq \frac12 \alpha R$, then we can apply Lemma \ref{lemme:décroissance_angles}, which gives:
\begin{align*}
	\scal{a^1_s}{y} &= \cos(s) \scal{\frac{J A_x}{\norm{A_x}}}{y} + \sin(s) \scal{x}{y} \\
	&\leq \abs{\scal{\frac{J A_x}{\norm{A_x}}}{y}} - \frac{\sin(s)}{2} e^t \\
	&\leq \left( \frac{1}{\epsilon} \frac{e^{2dt} -1}{\frac{1}{1-d} - \frac{e^{2dt}}{1+d}} \right)^{\frac12} - \frac{\sin(s)}{2} e^t \tag*{by Lemma \ref{lemme:décroissance_angles}.}
\end{align*}
Let us search for the maximal distance $t = t_0$ such that the first term is less than $\epsilon^{\frac14}$:
\begin{align*}
	&0 \leq \left( \frac{1}{\epsilon} \frac{e^{2dt} -1}{\frac{1}{1-d} - \frac{e^{2dt}}{1+d}} \right)^{\frac12} \leq \epsilon^{\frac14} \\
	&\Leftrightarrow e^{2dt} \left( 1 +\frac{\epsilon^{\frac32}}{1+d} \right) \leq 1 + \frac{\epsilon^{\frac32}}{1-d} \\
	&\Leftrightarrow t \leq \frac{1}{2d} \ln \left( \frac{1 + \frac{\epsilon^{\frac32}}{1-d}}{1 +\frac{\epsilon^{\frac32}}{1+d}} \right). 
\end{align*}
We now impose that the distance $d(x,y)$ is equal to $t_0 = \frac{1}{2d} \ln \left( \frac{1 + \frac{\epsilon^{\frac32}}{1-d}}{1 +\frac{\epsilon^{\frac32}}{1+d}} \right)$. Taking into account the equivalent \eqref{éq:équivalent_d}, we have $t_0 \stackrel{R \to \infty}{\sim} \frac12 \ln(\frac{1}{\epsilon^{\frac12}}) \stackrel{R \to \infty}{\sim} \frac{\alpha}{4} R$ (which is fully compatible with the hypothesis $t \leq \frac12 \alpha R$ made previously). Given $y \in S$ at a distance from $x$, we thus have
\begin{align}
	\scal{a_s}{y} &\leq \epsilon^{\frac12} - \frac{\sin(s)}{2} e^{t_0} \label{éq:contrôle_scalaire_as} \\
	&= 1 - \frac{\sin(s)}{2} e^{t_0}. \nonumber
\end{align}
Setting $s_1 = \sin^{-1}(4 e^{-t_0})$, we indeed obtain
\[\scal{a_s}{y} \leq -1. \]
Taking into account the equivalent $t_0 \stackrel{R \to \infty}{\sim} \frac{\alpha}{4} R$, this also proves a bound of the type $s \leq C' e^{-\beta R}$ for $R$ sufficiently large, which concludes step \ref{étape:contrôle_loin}. 

\etape[control near $x$]\label{étape:contrôle_près}
Set $s_2 = \sin^{-1}(2\sqrt{2C} e^{-\frac{\alpha}{2} R})$. Given $y \in S$ at distance $d(x,y) \leq t_0$ from $x$, equation \eqref{éq:contrôle_scalaire_as} rewrites as
\begin{align*}
	\scal{a_{s_2}}{y} &\leq \sqrt{2C} e^{-\frac{\alpha}{2} R} - \frac{\sin(s_2)}{2} \\
	&\leq 0. 
\end{align*}
Moreover, for $R_0$ sufficiently large and $R \geq R_0$, we have $s_2 \leq 4\sqrt{C} e^{-\frac{\alpha}{2} R}$ as desired. This concludes step \ref{étape:contrôle_près}, and therefore the proof of Proposition \ref{proposition:décroissance_volume}. 
\end{proof}

\subsection{Integration of an exponentially decaying function}
\label{subsection:contrôle_Dt}

In this section, we use the previous estimates to conclude the proof of Theorem \ref{theoreme:borne_sup}, from which we deduce Theorem \ref{théorème:borne_sup_cocompacte}. We first recall the framework of Theorem~\ref{theoreme:borne_sup}. 

\begin{definition}
Let $\Gamma$ be a subgroup of $\PO(2,3)$ preserving a complete maximal surface $S$ of $\H^{2,2}$ and whose action on $S$ is properly discontinuous and free : $S/\Gamma$ is therefore a Riemannian manifold. Let $f$ be a continuous function from $\cal{M}_{2,2}^*$ to $\R$, invariant under the action of $\PO(2,3)$. Given a fundamental domain $\cal{F}$ of $S$ for the action of $\Gamma$, we denote
\[\int_{S/\Gamma} f = \int_{\cal{F}} f(S,x)\ \d\mu_S(x) \]
when this latter integral is absolutely convergent. We say that the sectional curvature $K_{S/\Gamma}$ tends towards $0$ at infinity (denoted $K_{S/\Gamma}~\stackrel{\infty}{\to}~0$) if every neighbourhood of $0$ contains the values of $K_{S/\Gamma}$ outside a compact of $S/\Gamma$. 
\end{definition}

\begin{theoremeLettre}\label{theoreme:borne_sup}
Let $f : \cal{M}_{2,2}^* \to \R$ be a continuous and exponentially decaying function. There exists a constant $C_f$ satisfying the following property. Let $\Gamma$ be a subgroup of $\PO(2,3)$ preserving a complete maximal surface $S$ of $\H^{2,2}$ and whose action on $S$ is properly discontinuous and free. Suppose that $K_{S/\Gamma}$ tends towards $0$ at infinity. We have
\[\int_{S/\Gamma} \abs{f} \leq C_f \int_{S/\Gamma} \abs{K}. \]
\end{theoremeLettre}

\textbf{Plan of the proof.} We begin by defining a subset $\cal{D}$ of $\Sigma$ on which the curvature is large (Definition \ref{définition:domaine_D}). We then control the number of connected components of $\cal{D}$ (Lemma \ref{lemme:zéros=courbure}). Finally, we use the exponential decay of the function $f$ to integrate it on $\Sigma$ and conclude.

\begin{proof}
Let us recall that for all $k<0$, we denote by $D^k_S$ the set of points $x$ of $S$ satisfying $K_S(x) < k$. Choose, for the remainder of this section, a real number $-\frac13 < k_0 < 0$. By Proposition \ref{proposition:courbure_petite_contrôle_uv}, for all $k < 0$ sufficiently close to $0$, if at a point $x$ of $S$ the sectional curvature $K_S$ satisfies $k \leq K_S(x)$, then we have
\[\abs{u(x)-\ln\frac12}, \abs{v(x)-\ln\frac12} \leq \epsilon = -\ln(1+k_0) \]
thus $\mu_1(x) \geq \frac12 \ln\frac12 - \frac{\epsilon}{2}$. Let us choose such a $k$, imposing moreover that $k_0 \leq k$ and that $D^k_S$ has a piecewise $\cal{C}^\infty$ boundary (this is the case for generic $k$ by Sard's lemma).

Let $C$ be a connected component of $D^k_S$ not containing any zero of $\alpha \beta$ : in other words, on $C$, the function $\mu_1 = \frac12 \ln(\norm{\alpha \beta})$ is well-defined. From the formulas \eqref{éq:Hitchin_mu} ($\frac12 \Delta \mu_1 = K_S$) and from the inequality $K_S \leq 0$ (Theorem \ref{theoreme:Ishihara}), we deduce that $\mu_1$ is superharmonic on $C$:
\[\Delta \mu_1 \leq 0. \]
We moreover have the inequality $\mu_1(x) \geq \frac12 \ln\frac12 - \frac{\epsilon}{2}$ on $\partial C$. The maximum principle therefore implies that on $C$, we have $\mu_1(x) \geq \frac12 \ln\frac12 - \frac{\epsilon}{2}$. We conclude that on $C$, we have
\[K_S = 2e^{2\mu_1} \cosh(\mu_2) -1 \geq e^{-\epsilon}-1 = k_0. \]
\begin{definition}\label{définition:domaine_D}
	Let $D$ denote the union of the connected components of $D^k_S$ containing a zero of $\alpha \beta$. The open set $D$ descends to an open set of $\Sigma$ that we denote by $\cal{D}$.
\end{definition}
We have just shown that $k_0 \leq K_S$ on $D^k_S \setminus D$. As by hypothesis we have $k_0 \leq k \leq K_S$ on $S \setminus D_S^k$, we obtain $k_0 \leq K_S$ on $S \setminus D_S^k$. This inequality descends to $\Sigma$ : on $\Sigma \setminus \cal{D}$, we have
\[k_0 \leq K_\Sigma. \]
Let us notice that if $D_S^k$ is empty, then $S$ is a Barbot surface (remark \ref{remarque:D_vide}) : in this case we have $\int_{S/\Gamma} f = 0$ and the theorem is trivial. We henceforth assume $D_S^k$ non-empty. We also assume that the holomorphic section $\alpha \beta$ is non-zero : the case $\alpha \beta = 0$, which is simpler, will be treated separately (remark \ref{remarque:alpha_beta_nul}). The function $\norm{\alpha \beta}$ descends to a function on $\Sigma$ whose zeros are isolated and contained in the compact $\bar{\cal{D}}$. It therefore has a finite non-zero number of zeros, which we denote by $Z$. 

Let us recapitulate : the set $\cal{D}$ is non-empty, has at most $Z$ connected components by definition, and the hypothesis $K_\Sigma(x) \stackrel{x \to \infty}{\to} 0$ implies that it is relatively compact. 

Let us denote by $\d^\Sigma$ the distance on $\Sigma$ and for every real $t > 0$, let us define the set $\cal{D}_t$ of points $x$ of $\Sigma$ satisfying $\d^\Sigma(x,\cal{D}) < t$. By Proposition \ref{proposition:régularité_bord_Dt}, for $t$ in an open set $U$ of $\R_{>0}$ with complement of measure zero, the boundary of $\cal{D}_t$ is piecewise $\cal{C}^\infty$. By Corollary \ref{corollaire:contrôle_bord_Dt}, for $t$ and $s$ two elements of $U$ satisfying $t \leq s$, we have
\begin{align*}
	\Vol(\partial \cal{D}_s) - \Vol(\partial \cal{D}_t) &\leq \left( 4\pi \Card \pi_0(\bar{\cal{D}}) + \int_\Sigma \abs{K_\Sigma} \right) (s-t) \\
	&\leq \left( 4\pi Z + \int_\Sigma \abs{K_\Sigma} \right) (s-t) \tag*{by definition of $\cal{D}$.}
\end{align*}
Let us set $T = 4\pi Z + \int_\Sigma \abs{K_\Sigma}$. With $s=1$ (or $s$ in $U$ arbitrarily close to $1$), we obtain for $0 \leq t \leq 1$ :
\[\Vol(\partial \cal{D}_t) \geq \Vol(\partial \cal{D}_1) - T. \]
By integrating the expression $\frac{\d \Vol(\cal{D}_s)}{\d s}\bigg|_{t} = \Vol(\partial \cal{D}_t)$ (valid on the open set $U$, with complement of measure zero), we deduce
\begin{equation}\label{éq:démo_volume_Dt}
	\Vol(\cal{D}_1) \geq \Vol(\cal{D}_0) - T + \Vol(\partial \cal{D}_1). 
\end{equation}
In order to bound above the volume of $\cal{D}_1$, we establish the following lemma. 
\begin{lemme}\label{lemme:Harnack_K}
	There exists a \emph{universal} constant $C_H \geq 1$ such that for every complete maximal surface $S$ of $\H^{2,2}$, every point $x$ of $S$ and every point $y$ in the ball $B(x,1)$ of $S$, the sectional curvature $K$ satisfies
	\[K(x) \leq C_H K(y). \]
\end{lemme}
\begin{proof}
	The function $u = -K \geq 0$ satisfies $(\Delta -4 \Id) u \leq 0$ (\cite[équation (20)]{LT22}). The Harnack inequality (Proposition \ref{prop:harnack}) applies to the solutions of $(\Delta -4 \Id) u = 0$, hence \emph{a fortiori} to super-solutions : we therefore have, for every $x \in S$, $\sup_{B(x,1)} K \leq C_H K(x)$. 
\end{proof}
We deduce from the previous lemma that we have $\abs{K_\Sigma} \geq \frac{\abs{k}}{C_H}$ on $\cal{D}_1$, therefore
\begin{equation}\label{éq:majoration_volume_D1}
	\frac{\abs{k}}{C_H} \Vol(\cal{D}_1) \leq \int_{\cal{D}_1} \abs{K_\Sigma} \leq T.
\end{equation}
We can then bound above $\Vol(\partial \cal{D}_1)$ by rewriting the formula \eqref{éq:démo_volume_Dt} :
\[\Vol(\partial \cal{D}_1) \leq \Vol(\cal{D}_1) + T \leq \frac{C_H}{\abs{k}} T + T. \]
We therefore obtain, by integrating \eqref{éq:périmètre_boule_croissance}, a bound on $\Vol(\partial \cal{D}_t)$ for $t \geq 1$ :
\begin{equation}\label{éq:croissance_linéaire_bord}
	\Vol(\partial \cal{D}_t) \leq T\ t + \frac{T C_H}{\abs{k}}. 
\end{equation}

\begin{lemme}[control of the number of connected components of $\cal{D}$]\label{lemme:zéros=courbure}
	With the previous hypotheses and notations, $Z$ is given by
	\[Z = \frac{1}{4\pi} \int_\Sigma \abs{K_\Sigma} \d\!\vol^\Sigma. \]
\end{lemme}
\begin{proof}
	If the holomorphic section $\alpha \beta$ vanishes at a point $x$ of $S$, let us denote by $\rm{ord}_x(\alpha \beta)$ the order of this zero. Since the function $\mu_1$ is (locally) the norm of a holomorphic section of the canonical bundle of $S$, we classically have
	\begin{equation}\label{éq:laplacien_mu1}
		\frac12 \Delta \mu_1 = K_S + 2\pi \sum_{(\alpha \beta)_x = 0} \rm{ord}_x(\alpha \beta)\, \delta_x.
	\end{equation}
	Let $t > 1$ be a real number such that $\cal{D}_t$ is piecewise $\cal{C}^\infty$ (this is the case for almost every $t$ by Proposition \ref{proposition:régularité_bord_Dt}). By smoothing $\cal{D}_t$, we define an open set $\cal{D}_t'$ with $\cal{C}^\infty$ boundary satisfying $\cal{D}_{t-1} \subset \cal{D}_t' \subset \cal{D}_t$ and $\abs{\Vol(\partial \cal{D}_t') - \Vol(\partial \cal{D}_t)} \leq 1$. The equation \eqref{éq:croissance_linéaire_bord} then gives :
	\[\Vol(\partial \cal{D}_t') \leq \Vol(\partial D_t) + 1 \leq  T\ t + \frac{T C_H}{\abs{k}} + 1. \]
	The function $\mu_1$ descends to $\Sigma$ as a function which we denote $\mu_1^\Sigma$. We know that the functions $\mu_1$ and $\mu_2$ are exponentially decaying (Proposition \ref{proposition:contrôle_mu}) hence that $\Delta \mu_1$ is (equations \eqref{éq:Hitchin_mu_alt}). Proposition \ref{proposition:estimées_W2p} then proves $\norm{\nabla \mu_1}$ is exponentially decaying (see the proof of Corollary \ref{corollaire:contrôle_nablaII}). We obtain that there exist universal constants $C,\alpha > 0$ such that for every real $t$ greater than a universal constant, we have on $\partial \cal{D}_t$ :
	\[\norm{\nabla \mu_1^\Sigma} \leq C e^{-\alpha t}. \]
	The equality \eqref{éq:laplacien_mu1} descends to an equality between distributions on $\Sigma$. Let us integrate it over $\cal{D}_t'$ taking into account the bounds on $\Vol(\partial \cal{D}_t')$ and $\norm{\nabla \mu_1^\Sigma}$. We obtain
	\begin{align*}
		\abs{Z + 2\pi \int_{\cal{D}_t'} K_\Sigma} &= \abs{\int_{\cal{D}_t'} \Delta \mu_1^\Sigma} \tag*{since $\cal{D}_t'$ contains the zeros of $\alpha \beta$} \\
		&\leq \int_{\partial \cal{D}_t'} \norm{\nabla \mu_1^\Sigma} \tag*{by Stokes} \\
		&\leq \Vol(\partial \cal{D}_t') C e^{\alpha (t-1)}
	\end{align*}
	by the estimate of $\norm{\nabla \mu_1^\Sigma}$ and $\cal{D}_{t-1} \subset \cal{D}_t' \subset \cal{D}_t$. Taking into account the estimate of $\Vol(\partial \cal{D}_t')$, we obtain :
	\[\abs{Z + 2\pi \int_{\cal{D}_t'} K_\Sigma} \stackrel{t \to \infty}{\longrightarrow} 0. \]
	This proves Lemma \ref{lemme:zéros=courbure}. 
\end{proof}
To finish the proof of Theorem \ref{theoreme:borne_sup}, we still have to integrate $f$ over $\Sigma$ taking into account its exponential decay of $f$ and equation \eqref{éq:croissance_linéaire_bord}. To that end, we use the following classical lemma. 
\begin{lemme}[co-area formula]
	Let $M$ and $N$ be smooth Riemannian manifolds whose measures are denoted by $\mu_M$ and $\mu_N$. Let $u : M \to N$ be a map, smooth in a neighbourhood of almost every point of $M$ and such that $\d u$ is almost everywhere surjective. Let $\phi : M \to \R_{\geq 0}$ be a measurable function. Let us denote by $\abs{\Jac^\perp(u,x)}$ the absolute value of the determinant (of the matrix in orthonormal frames) of the restriction of $\d u$ to the orthogonal of its kernel : this function is well defined as soon as $u$ is smooth in a neighbourhood of $x$ and $\d u$ surjective, hence for almost every $x$. We have:
	\[\int_M \phi\ \abs{\Jac^\perp(u,x)}\ \d\mu_M(x) = \int_{y \in N} \left( \int_{x \in u^{-1}(y)} \phi(x)\ \d\mu_{u^{-1}(y)}(x) \right) \d\mu_N(y), \]
	the integral $\int_{x \in u^{-1}(y)} \phi(x)\ \d\mu_{u^{-1}(y)}(x)$ being well defined when $y$ is a regular value, hence for almost every $y$. 
\end{lemme}
Let us apply the co-area formula to $u = \d^\Sigma(\cdot,\cal{D}) : \Sigma \setminus \cal{D}_1 \to \R_{>0}$ (with $\d^\Sigma$ the distance on $\Sigma$), which satisfies the previous assumptions with $\abs{\Jac^\perp(u,x)} = 1$ almost everywhere. We obtain that every measurable function $\phi : \Sigma \to \R_{\geq 0}$ satisfies
\begin{equation}\label{éq:co-aire}
	\int_\Sigma \phi\ \d\mu_\Sigma = \int_{\cal{D}_1} \phi\ \d\mu_\Sigma + \int_{t>1} \left( \int_{\partial \cal{D}_t} \phi(x)\ \d\mu_{\partial \cal{D}_t}(x) \right) \d t
\end{equation}
the integral $\int_{\partial D_t} \phi(x)\ \d\mu_{\partial D_t}(x)$ being well defined for $t \in U$, hence for almost every $t$. The function $f$, which is exponentially decaying, is $\Gamma$-invariant and thus descends to a function on $\Sigma$ denoted by $f_\Sigma$. The definition of exponential decay (Definition \ref{définition:décroissance_exponentielle}) and remark \ref{remarque:décroissance_exponentielle_partout} show that there exist constants $C,\alpha > 0$ (depending only on $k$ and $f$) satisfying, for all $x \in \Sigma$ :
\[\abs{f_\Sigma(x)} \leq Ce^{-\alpha \d^{\Sigma}(x,\cal{D})}. \]
Let us apply formula \eqref{éq:co-aire} to $\phi = f_\Sigma$ taking into account this inequality and inequality \eqref{éq:croissance_linéaire_bord}. We obtain
\begin{align}
	\int_\Sigma \abs{f_\Sigma} &\leq C \vol(\cal{D}_1) + \int_{t>1} (T\ t + \frac{T C_H}{\abs{k}}) Ce^{-\alpha t}\ \d t \tag*{thus by \eqref{éq:majoration_volume_D1},} \nonumber \\
	&\leq T \left[ C \frac{C_H}{\abs{k}} + \int_{t>1} (t + \frac{C_H}{\abs{k}}) Ce^{-\alpha t}\ \d t \right]. \label{éq:majoration_intégrale_VSigma}
\end{align}
In this last upper bound, the constants $C$, $C_H$ and $k$ depend only on $k_0$ (which we have chosen arbitrarily) and on $f$ : they are in particular independent of $S$ and $\Gamma$. Recall that we have 
\begin{align*}
	T &= 4\pi Z + \int_\Sigma \abs{K_\Sigma} \tag*{by definition} \\
	&=2 \int_\Sigma \abs{K_\Sigma} \d\!\vol^\Sigma \tag*{by Lemma \ref{lemme:zéros=courbure}.}
\end{align*}
Equation \eqref{éq:majoration_intégrale_VSigma} can hence be written as
\[\int_\Sigma \abs{f_\Sigma} \leq C_f \int_\Sigma \abs{K_\Sigma}, \]
with $C_f$ a constant depending only on $k_0$ and $f$. This concludes the proof of Theorem \ref{theoreme:borne_sup}.
\end{proof}

\begin{remarque}\label{remarque:alpha_beta_nul}
	If $\alpha \beta = 0$, then $\alpha = 0$ or $\beta = 0$. We have therefore $\cal{D} = \Sigma$, that is $K_\Sigma \leq -\frac13$ on $S$, which gives an upper bound for the area of $\Sigma$ : $\Vol(\Sigma) \leq 3 \int_\Sigma \abs{K_\Sigma} \d\!\vol^\Sigma$. We thus obtain, instead of the equation \eqref{éq:majoration_intégrale_VSigma} : $\int_\Sigma f_\Sigma \leq 3C T$. The rest is unchanged. 
\end{remarque}

\begin{proof}[Proof of Theorem \ref{théorème:borne_sup_cocompacte}]
Let $\rho$ be a maximal representation from $\Gamma = \pi_1(\Sigma_g)$ into $\SO_0(2,3)$, $S$ be the maximal surface associated to $\rho$ by Theorem \ref{theoreme:existence_unicité_CTT} and $\cal{F}$ be a fundamental domain of $S$ for the action of $\Gamma$. We use Proposition \ref{proposition:volume_tranche}. With the notations of this proposition, we have
\begin{align*}
	\Vol \left[ \Conv(S)/\Gamma \right] &= \int_{x \in \cal{F}} \left( \int_{\Nor_x^0 S} \frac{\nu^* \mu_\H}{\mu_{\Nor S}}\ \d\mu_{\Nor S}(\sf{n}) \right) \d\mu_S(x) \\
	&\leq C_2 \int_{\cal{F}} \bar{\bf{V}}(S,x)\ \d\mu_S(x) \tag*{by Proposition \ref{proposition:volume_tranche}}\\
	&= C_2 \int_{S/\Gamma} \bar{\bf{V}}. 
\end{align*}
As the function $f = \bar{\bf{V}}$ is exponentially decaying (Proposition \ref{proposition:décroissance_volume}), we can apply Theorem \ref{theoreme:borne_sup}. We obtain :
\[\int_{S/\Gamma} \bar{\bf{V}} \leq C_{\bar{\bf{V}}} \int_{S/\Gamma} \abs{K} \]
Since $K$ is of constant sign, the Gau\ss-Bonnet formula writes as
\[\int_{S/\Gamma} \abs{K} = 2\pi \abs{\chi(\Sigma)} = 4\pi (g-1). \]
We thus obtain
\[\Vol(\rho) = \Vol \left[ \Conv(S)/\Gamma \right] \leq C_2 C_{\bar{\bf{V}}} 4 \pi (g-1) \]
thus, taking into account $g \geq 2$, an inequality of the type $\Vol(\rho) \leq C g$ with $C$ a universal constant. This proves Theorem \ref{théorème:borne_sup_cocompacte}.
\end{proof}

Let us finally prove a consequence of Lemma \ref{lemme:Harnack_K} which gives examples to which Theorem \ref{theoreme:borne_sup} can be applied. 
\begin{lemme}\label{lemme:rayon_injectivité_gr}
	Let $S$ be a complete maximal surface of $\H^{2,2}$ and $\Gamma$ a subgroup of $\PO(2,3)$ acting properly discontinuously and freely on $S$. Let us suppose that the injectivity radius of $S/\Gamma$ is strictly positive and that the total curvature $\int_{S/\Gamma} K$ is finite. Then $S/\Gamma$ is asymptotically flat (in other words, we have $K_{S/\Gamma} \stackrel{\infty}{\to} 0$). 
\end{lemme}
\begin{proof}
	Let us denote by $R$ the injectivity radius of $S/\Gamma$. We argue by contradiction and suppose given a divergent sequence $(x_n)$ of points of $S$ satisfying $K(x_n) \leq -\epsilon$ for some $\epsilon > 0$. Up to extracting, we can suppose the balls $B(x_n,R)$ disjoint. By Lemma \ref{lemme:Harnack_K}, on each of the balls $B(x_n,R)$, we have $K \leq -\frac{\epsilon}{C_H}$. We deduce
	\[\int_S K \leq \sum_n \int_{B(x_n,R)} K \leq \sum_n -\frac{\epsilon}{C_H} \Vol(B(x_n,R)). \]
	As we have everywhere $-1 \leq K \leq 0$, the volume $\Vol(B(x_n,R))$ is bounded below by a strictly positive universal constant, which proves that the total curvature of $S$ is infinite, thus contradicts the hypothesis of the lemma. 
\end{proof}

\begin{corollaire}\label{corollaire:Moriani_volume}
	Moriani shows that a complete maximal surface $S$ of $\H^{2,2}$ has finite total curvature if, and only if, the boundary of $S$ in $\P(\R^{2,3})$ is union of a finite number $N$ of projective segments. We then have $\int_S K = -\frac{\pi}{2} (N-4)$ \cite[Main Theorem]{Mor25}. Lemma \ref{lemme:rayon_injectivité_gr} shows that we can apply Theorem \ref{theoreme:borne_sup} with $\Gamma = \{\Id\}$ and $S$ such a surface. Theorem \ref{theoreme:borne_sup} proves that there exists a constant $C_{\bar{V}}$ (independent of $S$ and $N$) satisfying
	\[\int_S \bar{V} \leq C_{\bar{V}}\, (N-4). \]
	As in the proof of Theorem \ref{théorème:borne_sup_cocompacte}, we deduce the existence of a constant $C_{\Vol}$ such that for any such surface, we have
	\[\Vol(\Conv(S)) \leq C_{\Vol}\, (N-4). \]
\end{corollaire}

\begin{corollaire}\label{corollaire:cocompact_alpha}
	By Theorem \ref{theoreme:existence_unicité_CTT}, any maximal representation $\rho$ from $\Gamma = \pi_1(\Sigma_g)$ into $\SO_0(2,3)$ defines a $\rho$-invariant maximal surface of $\H^{2,2}$ denoted $S_\rho$. The quotient $S_\rho/\rho(\Gamma)$ is then a compact Riemannian manifold whose sectional curvature is denoted $K_{S_\rho/\rho(\Gamma)}$. Let $\alpha > 0$ be a real number. Since $K$ is exponentially decaying, this is also the case of $\abs{K}^\alpha$ (see proof of Corollary \ref{corollaire:contrôle_nablaII}). Theorem \ref{theoreme:borne_sup} therefore proves the existence of a constant $C_\alpha'$ such that for any maximal representation $\rho : \pi_1(\Sigma_g) \to \SO_0(2,3)$, we have
	\[\int_{S_\rho/\rho(\Gamma)} \abs{K}^\alpha \leq C_\alpha'\, \int_{S_\rho/\rho(\Gamma)} \abs{K}. \]
	We deduce from this, as in the proof of Theorem \ref{théorème:borne_sup_cocompacte}, a bound of the type
	\[\int_{S_\rho/\rho(\Gamma)} \abs{K}^\alpha \leq C_\alpha\, g. \]
\end{corollaire}

\appendix
\newpage

\section{Volume above a point}
\label{annexe:nu}

We state two lemmas of differential geometry used in section \ref{subsection:Chern-Weil}. Given a maximal complete surface $S$ in $\H^{2,q}$ ($q \geq 1$), we denote by $\nu$ the restriction of the exponential map of $\H^{2,q}$ to $\Nor S$. For every real $s > 0$, let us denote by $\Nor^{<s} S$ the set of normal vectors 
$\sf{n} \in \Nor S$ of norm $\norm{\sf{n}} < s$. 
\begin{lemme}[{\cite[Lemma 4.3]{MV23}}]\label{lemme:nu_injective}
	Let $S$ be a complete maximal surface of $\H^{2,2}$. The restriction of $\nu$ to $\Nor^{<\frac{2}{\sqrt3}} S$ is injective. 
\end{lemme}
\begin{proof}
	This is Lemma 4.3 of \cite{MV23}, combined with remark \ref{remarque:fonctions_géométriques_axes} ($\norm{\II}_\infty \leq \frac{2}{\sqrt3}$). 
\end{proof}

The following lemmas are used to control the volume of the convex core of $S$ by an integral over $\Nor S$. The metric of $\Hc^{2,q}$ defines a measure $\mu_\H$ on $\Hc^{2,q}$. The Sasaki metric $g^{\Tan \H^{2,q}}$ (see Section \ref{subsection:geo_diff_Hpq}) restricts to $\Nor S \subset \Tan \H^{2,q}$ into a pseudo-Riemannian metric of signature $(2,q)$, which defines a measure $\mu_{\Nor S}$ on $\Nor S$. Since $\Nor S$ and $\H^{2,q}$ have the same dimension, the image of a negligible subset of $\Nor S$ by the smooth map $\nu$ is a negligible subset of $\H^{2,q}$ : this means that $\nu^* \mu_\H$ is absolutely continuous with respect to $\mu_{\Nor S}$. The Radon-Nidokym theorem defines the derivative $\frac{\nu^* \mu_\H}{\mu_{\Nor S}}$, which is in $\rm{L}^1_{loc}(\Nor S,\mu_{\Nor S})$. If we choose orientations of the manifolds $\Nor S$ and $\H^{2,q}$, we can also define this function as the absolute value of the jacobian of $\nu$. 
\begin{lemme}\label{lemme:jacobienne_minorée}
	There exists a real $c>0$ having the following property : for every complete maximal surface $S$ of $\H^{2,2}$ and every vector $\sf{n} \in \Nor S$ satisfying $\norm{\sf{n}} \leq c$, we have at the point $\sf{n}$ :
	\[\frac{\nu^* \mu_\H}{\mu_{\Nor S}} \geq \frac12. \]
\end{lemme}
\begin{proof}
	Lemma \ref{lemme:différentielle_exp_normal} states that for every unit $\sf{n} \in \Nor_x S$ and every real $s$ satisfying $0 < s < \pi$, we have
	\begin{align*}
		\frac{\nu^* \mu_\H}{\mu_{\Nor S}} &= \frac{\sin(s)}{s} [\cos(s)^2 - \sin(s)^2 (\scal{\II (e_1,e_1)}{\sf{n}}^2 + \scal{\II (e_1,e_2)}{\sf{n}}^2] \\
		&\geq \frac{\sin(s)}{s} [\cos(s)^2 - \frac83 \sin(s)^2] \tag*{because $\norm{\II}_\infty^2 \leq \frac{4}{3}$ by remark \ref{remarque:fonctions_géométriques_axes}.}
	\end{align*}
	This last expression tends to $1$ when $s$ tends to $0$, which concludes. 
\end{proof}
We deduce from this a comparison between the volume $\bar{\bf{V}}_S$ (see Definition \ref{définition:volume_tranche}) and the integral of $\frac{\nu^* \mu_\H}{\mu_{\Nor S}}$. 
\begin{proposition}\label{proposition:volume_tranche}
	There exists universal constants $C_1,C_2 > 0$ such that for every complete maximal surface $S$ of $\H^{2,2}$ and every $x \in S$, we have:
	\[C_1 \bar{\bf{V}}_S(x) \leq \int_{\Nor_x^0 S} \frac{\nu^* \mu_\H}{\mu_{\Nor S}} \d\!\vol^{\Nor S}(\sf{n}) \leq C_2 \bar{\bf{V}}_S(x). \]
\end{proposition}
\begin{proof}
	Let us remark that $\Nor_x^0 S$ is contained in the ball of centre $0$ and of radius $\frac{\pi}{2}$ (Lemma \ref{lemme:épaisseur_pi2}) and that if $s\sf{n}$ is in $\Nor_x^0 S$, then the segment between $0$ and $s\sf{n}$ is included in $\Nor_x^0 S$ (by convexity of the convex hull of $S$). This observation, combined with Lemma \ref{lemme:jacobienne_minorée}, is sufficient to bound $\int_{\Nor_x^0 S} \frac{\nu^* \mu_\H}{\mu_{\Nor S}} \d \sf{n}$ from below. The upper bound comes from the fact that $\frac{\nu^* \mu_\H}{\mu_{\Nor S}}$ is bounded above (given its expression given by Lemma \ref{lemme:différentielle_exp_normal}). 
\end{proof}
Let $x \in S$ and $(\sf{u}_i)_{1 \leq i \leq p}$ and $(\sf{n}_i)_{1 \leq i \leq q}$ be orthonormal bases of $\Tan_x S$ and $\Nor_x S$ respectively. We also denote $\sf{n}_1 = \sf{n}$. 
\begin{lemme}\label{lemme:différentielle_exp_normal}
	With the previous notations, let $\tau > 0$ be a real number. At the point $\tau \sf{n} \in \Nor S$, the ratio between the measures $\nu^* \mu_\H$ and $\mu_{\Nor S}$ is given by
	\[\frac{\nu^* \mu_\H}{\mu_{\Nor S}} = \left| \frac{\sin(\tau)}{\tau}\right|^{q-1} \left| \cos(\tau)^2 - \sin(\tau)^2 (\scal{\II (\sf{u}_1,\sf{u}_1)}{\sf{n}}^2 + \scal{\II(\sf{u}_1,\sf{u}_2)}{\sf{n}}^2) \right|. \]
\end{lemme}
\begin{proof}
With the notations of Section \ref{subsection:geo_diff_Hpq} (definition of the Sasaki metric), we obtain an orthonormal basis of $\Tan_{\tau\sf{n}}(\Nor S)$ by concatenating $(\bar{\sf{u}}_i) = (\iota_{\bf{hor}} (\sf{u}_i))$ and $(\bar{\sf{n}}_i) = (\iota_{\bf{vert}} (\sf{n}_i))$. We also define an orthonormal basis $(\sf{U}_i) \cup (\sf{V}_i)$ of $\Tan_y \H^{2,q}$ by parallel transporting the orthonormal basis $(\sf{u}_i) \cup (\sf{n}_i)$ along the geodesic $t \mapsto \nu(t\sf{n})$, from $t=0$ to $t=\tau$. To prove the lemma, let us compute the matrix of $\d \nu_{\tau \sf{n}}$ in these bases using Jacobi fields. Let $\sf{u} \in \Tan_x S$ and let $\alpha : ]-\epsilon;\epsilon[ \to S$ be a geodesic of $S$ satisfying $\alpha(0) = x$ and $\dot{\alpha}(0) = \sf{u}$. We define a vector field $\sf{n}_s \in \Nor_{\alpha(s)}$ by parallel transport of $\sf{n}$ along $\alpha$. Define the map 
\[F : \R \times ]-\epsilon;\epsilon[ \to \H^{2,q},\ (t,s) \mapsto \exp_{\alpha(s)}(t\sf{n} _s). \]
Equip the bundle $F^* \Tan \Hc^{2,q}$ with the connection $\nabla = F^*\nabla^\H$ and denote its curvature by $\courb$. We thus have bundle morphisms between $F^* \Tan \Hc^{2,q}$ and $\Tan \Hc^{2,q}$, denoted by $F_*$ and $F^*$. We extend $\sf{n}$ and $\sf{u}$ to sections of $F^* \Tan \Hc^{2,q}$, respectively by $N = F^*(\d F (\der[t]))$ and $J = F^*(\d F (\der[s]))$. $F$ is a variation of the geodesic $\gamma : t \mapsto F(t,s)$ by geodesics : $\d F (\der[s])$ therefore defines a Jacobi field along $\gamma$. The Jacobi field equation then writes as : 
	\[\nabla_{\der[t]} \nabla_{\der[t]} J + \courb(\der[s],\der[t]) N = 0. \]
	The curvature tensor $\courb^\H$ of $\Hc^{2,q}$ is given by $\courb^\H(X,\,Y) Z = -\scal Y Z X + \scal X Z Y$ (see section \ref{subsection:équations_fondamentales}). The previous equation therefore becomes $\nabla_{\der[t]} \nabla_{\der[t]} J = -J$. Let us consider this equation at $(t,0)$: it is a linear second-order differential equation in $t$, for which we have the initial conditions $J_{(0,0)} = \sf{u}$ and $(\nabla_{\der[t]} J)_{(0,0)} = B(\sf{u},\sf{n})$. Let us denote by $U$, $V$ the sections of $F^* \Tan \H^{2,q}$ along $t \mapsto (t,0)$ obtained by doing parallel transport of the vectors $\sf{u}$ and $B(\sf{u},\sf{n})$ along $t \mapsto (t,0)$. We therefore have, at the point $(t,0)$:
\begin{equation*}
	J = \cos(t) U + \sin(t) V. 
\end{equation*}
The equality $B(\sf{u},\sf{n}) = \sum_i -\scal{\sf{n}}{\II(\sf{u},\sf{u}_i)} \sf{u}_i$ shows that at $(\tau,0)$, we have $F_*(V) = \sum_i -\scal{\sf{n}}{\II(\sf{u},\sf{u}_i)} \sf{U}_i$. We deduce:
\[(\d\nu)_{\tau \sf{n}}(\bar{\sf{u}}_i) = F_*(J_{(\tau,0)}) = \cos(\tau) \sf{U}_i - \sin(\tau) [\sum_j \scal{\sf{n}}{\II(\sf{u}_i,\sf{u}_j)} \sf{U}_j]. \]
Differentiating the equality $\nu(t\sf{n}) = \cos(t) x + \sin(t) \sf{n}$ with respect to $\sf{n} \in \Nor^1_x S$ with respect to $t \in \R$, we further obtain
\begin{align*}
	(\d\nu)_{\tau \sf{n}}(\bar{\sf{n}}_i) &= \frac{\sin(\tau)}{\tau} \sf{N}_i \tag*{for $i>1$ and} \\
	(\d\nu)_{\tau \sf{n}}(\bar{\sf{n}}_1) &= (d\nu)_{\tau \sf{n}}(\sf{n}) = \sf{N}_1. 
\end{align*}
The ratio between the measures $\mu_{\Nor S}$ and $\nu^* \mu_{\H}$ is the absolute value of the determinant of the matrix of $d\nu_{\tau \sf{n}}$ in the orthonormal bases $(\bar{\sf{u}}_i) \cup (\bar{\sf{n}}_j)$ and $(\sf{U}_i) \cup (\sf{N}_j)$. We have just computed this matrix, equal to
\begin{equation*}
	\begin{pmatrix}
		\cos(\tau) -\sin(\tau) \scal{\sf{n}}{\II(\sf{u}_1,\sf{u}_1)} & -\sin(\tau) \scal{\sf{n}}{\II(\sf{u}_2,\sf{u}_1)} &  &  \\
		-\sin(\tau) \scal{\sf{n}}{\II(\sf{u}_1,\sf{u}_2)} & \cos(\tau) -\sin(\tau) \scal{\sf{n}}{\II_{2,2}} &  &  \\
		&  & 1 &  \\
		&  &  & \frac{\sin(\tau)}{\tau} \Id_{q-1}
	\end{pmatrix}.
\end{equation*}
Taking into account $\II(\sf{u}_1,\sf{u}_1) = - \II(\sf{u}_2,\sf{u}_2)$ and $\II(\sf{u}_1,\sf{u}_2) = \II(\sf{u}_2,\sf{u}_1)$, its determinant is indeed the announced quantity. 
\end{proof}

\section{Control of $\partial D_t$ }
\label{annexe:contrôle_Dt}

Let $S$ be a $\cal{C}^\infty$ surface endowed with a Riemannian metric. 
\begin{definition}\label{définition:bord_C_infini}
	An open set $D$ of $S$ is said to have piecewise $\cal{C}^\infty$ boundary if in a neighbourhood of every point $p \in \partial D$, one can parametrise $\partial D$ by a continuous path $c : ]-\epsilon;\epsilon[ \to S$ such that $c(0) = p$, of class $\cal{C}^\infty$ away from $0$, whose derivative satisfies $\norm{c'} = 1$ away from $0$ and admits non-zero limits at $0^-$ and $0^+$, denoted respectively by $c'_-$ and $c'_+$ and satisfying $-c'_- \neq c'_+$. One moreover requires that the image of $c$ separates a small neighbourhood of $p$ into two connected components: one contained in $D$ and the other not intersecting $D$. We define the exterior angle at $p$ as the angle (exterior to $D$) between the vectors $c'_-$ and $c'_+$. 
\end{definition}

From now on, we assume $S$ complete, denote its distance by $\d^S$ and consider an open set $D$ of $S$ (not necessarily connected) with compact $\cal{C}^\infty$ boundary (and not just piecewise $\cal{C}^\infty$). We shall study $D_t = \{x \in S \mid \d^S(x,D) < t\}$. At each point $x$ of $\partial D$, let us denote by $\sf{n}_x \in \Tan^1_x S$ the outward normal vector to $\partial D$. We define a map
\begin{equation*}\begin{array}{rrcl}
		\gamma : & \partial D \times \R & \longrightarrow & S \\
		\ & (x,t) & \longmapsto & \exp_x(t \sf{n}_x)
\end{array}\end{equation*}
as well as $\gamma_t = \gamma(\cdot,t) : \partial D \to S$. We thus have $D_t = D \cup \gamma(\partial D \times [0;t[)$. The remainder of this appendix is devoted to the proof, technical but elementary, of the following proposition, and to that of its corollary. 
\begin{proposition}\label{proposition:régularité_bord_Dt}
	There exists an open subset $U$ of $\R_+$ with measure zero complement, such that for $t \in U$, $D_t$ has piecewise $\cal{C}^\infty$ boundary. 
\end{proposition}
\begin{corollaire}\label{corollaire:contrôle_bord_Dt}
	Assume moreover $S$ of finite total curvature ($\int_S \abs{K} < \infty$) and $D$ relatively compact. For $t \leq s$ two elements of $U$, we have
	\begin{equation}\label{éq:périmètre_boule_croissance}
		\Vol(\partial D_s) - \Vol(\partial D_t) \leq \left( 4\pi \Card \pi_0(\bar{D}) + \int_S \abs{K} \right) (s-t).
	\end{equation}
\end{corollaire}

\begin{proof}[Proof of Proposition \ref{proposition:régularité_bord_Dt}]
Let us choose, along $\partial D$, a vector field $\partial_x$ tangent to $\partial D$ of constant norm $\norm{\partial_x} = 1$. We say that $(x,t) \in \partial D \times \R$ is a \emph{focal point}, or that $x$ is a focal point of $\gamma_t$, if $\partial_x \gamma_t (x,t) = 0$. We say that it is a \emph{degenerate} focal point if moreover $\partial_x^2 \gamma_t (x,t) = 0$. If $x \in \partial D$ is not a focal point of $\gamma_t$, then $\gamma_t$ is regular ($\gamma_t' \neq 0$) in a neighbourhood of $x$. We shall first show that if $x$ is a non-degenerate focal point of $\gamma_t$ (such $x$ are isolated on $\partial D$), then in a neighbourhood of $x$, one can parametrise $\gamma_t$ by two $\cal{C}^\infty$ arcs of speed $1$ (step \ref{étape:point_focal_non_dégénéré}), glued together in a $\cal{C}^1$ way at $x$. We shall then show that for $t$ in an open subset $U$ of $\R$ of full measure, the focal points of $\gamma_t$ are non-degenerate (step \ref{étape:points_focaux_génériques}). Finally, we shall show that the self-intersections of $\gamma_t$ involving focal points can occur only for a countable set of $t$ (step \ref{étape:intersection_points_focaux}), then that for generic $t$ the segments parametrising $\partial D_t$ intersect transversely (step \ref{étape:intersection_génériquement_transverse}). Since $\partial D_t$ is made of arcs of $\gamma_t$ bounded by self-intersection points (see Definition \ref{définition:bord_C_infini}), this will complete the proof of Proposition \ref{proposition:régularité_bord_Dt}. 

\etape[Non-degenerate focal points are $\cal{C}^1$]\label{étape:point_focal_non_dégénéré}
Let $(x,t)$ be a non-degenerate focal point. We shall show that $\gamma_t$ admits a $\cal{C}^1$ parametrisation in a neighbourhood of $x$. In particular, the angle at the point $x$ is well defined and zero, and one can reparametrise the curve $\gamma_t$ in a neighbourhood of $x$ by two $\cal{C}^\infty$ arcs of speed $1$. 

Let us define a parametrisation $\alpha : ]-\epsilon;\epsilon[ \to \partial D$ in a neighbourhood of $x$ by integrating $\partial_x$ : we have $\alpha(0) = x$ and $\der[s] \alpha(s) = \partial_x$. The curve $c : ]-\epsilon;\epsilon[ \to \R^2$ defined by $c = \exp_{\gamma(x,t)}^{-1} \circ \gamma_t \circ \alpha$ is of class $\cal{C}^\infty$, and the hypothesis that $(x,t)$ is a non-degenerate focal point means that $c'(0) = 0 \neq c''(0)$. As $c(0) = c'(0) = 0$, the parametrisation $s \mapsto c(\frac{s}{\sqrt{\abs{s}}})$ is indeed of class $\cal{C}^1$, regular (i.e. $c' \neq 0$) up to restricting $\epsilon$, and $\cal{C}^\infty$ on $]-\epsilon;0[$ and $]0;\epsilon[$. 

\etape[For generic $t$, the focal points of $\gamma_t$ are non-degenerate]\label{étape:points_focaux_génériques}
Along the geodesic $\alpha_x = \gamma(x,\cdot)$, define a Jacobi field by $J = \d \gamma(\partial_x)$ : we thus have $J(t) \in \Tan_{\alpha(t)} S$. Let us denote by $U$ the vector field along $\alpha_x$ obtained by parallel transporting $\partial_x = J(0)$, and define
\begin{equation*}
	\begin{array}{rrcl}
		\phi : & \partial D \times \R & \longrightarrow & \R \\
		\ & (x,t) & \longmapsto & \scal{J(t)}{U(t)}
	\end{array}
\end{equation*}
so that $J = \phi U$.  The equation of Jacobi fields ($\ddot{J} = -R(J,\dot{\alpha})\dot{\alpha} = -K J$ with $K$ the sectional curvature) rewrites $\ddt[\phi] = -K \phi$. By the Cauchy-Lipschitz theorem, if $\phi(x,t) = \dt \phi\big|_{(x,t)} = 0$ for some $t \in \R$, then $\phi(x,t) = 0$ for all $t \in \R$, which is impossible since $\phi(x,0) = 1$. In other words, $0$ is a regular value of $\phi(x,\cdot)$ for all $x$, hence of $\phi$. 

Since moreover $\phi$ is $\cal{C}^\infty$, $\phi^{-1}(0)$ is a submanifold of dimension $1$ of $(\partial D) \times \R$. Let us denote by $\pi : \phi^{-1}(0) \to \R,\ (x,t) \mapsto t$ the projection. One easily sees that the set of degenerate focal points $\cal{F} = \{(x,t) \mid \phi(x,t) = \der[x]\phi\big|_{(x,t)} = 0\}$ is the set of critical points of $\pi$, hence $\pi(\cal{F})$ has measure zero by Sard's lemma. Since $\partial D$ is compact, $\pi$ is proper, hence closed. As $\cal{F}$ is closed, $\pi(\cal{F})$ is still closed. We have shown that $\pi(\cal{F})$, that is the set of $t$ such that $\gamma_t$ has degenerate focal points, is a closed set of measure zero, which concludes step \ref{étape:points_focaux_génériques}. We now denote $U = \R \setminus \pi(\cal{F})$. 

\etape[Intersection with a focal point]\label{étape:intersection_points_focaux}
Let $\psi : (\partial D)^2 \times U \to \R^2 \times S^2,\ (x,y,t) \mapsto (\phi(x,t),\phi(y,t),\gamma(x,t),\gamma(y,t))$. We define a submanifold of $\R^2 \times S^2$ by $\cal{F}_1 = \{(\sf{v}_1,\sf{v}_2,p,q) \in \R^2 \times S^2 \mid \sf{v}_1 = 0,\ \sf{v}_2 \neq 0 \text{ and } p=q\}$. A point $\psi(x,y,t)$ is in $\cal{F}_1$ if, and only if, $(x,t)$ is a focal point, $(y,t)$ is not a focal point and $\gamma(x,t) = \gamma(y,t)$. Let us show that $\psi$ intersects $\cal{F}_1$ with constant rank (here zero), in the sense that for all $(x,y,t) \in \psi^{-1}(\cal{F}_1)$ (which implies $x \neq y$), the image of $\d \psi_{(x,y,t)}$ intersects $\Tan_{\psi(x,y,t)} \cal{F}_1$ in a vector subspace of constant rank (here, this is equivalent to $\psi$ transverse to $\cal{F}_1$). Indeed, given such a triple $(x,y,t)$ and denoting $p = \gamma(x,t) = \gamma(y,t)$, we have
\begin{align*}
	\d \psi_{(x,y,t)}(\partial_x,0,0) &= (\der[x]\phi\big|_{(x,t)},0,0,0) \tag*{with $\der[x]\phi\big|_{(x,t)} \neq 0$ since $(x,t)$ is a non-degenerate focal point,}\\
	\d \psi_{(x,y,t)}(0,\partial_x,0) &= (0,\cdot,0,J_y(t)) \tag*{with $J_y(t) \neq 0$ since $(y,t)$ is not a focal point,}\\
	\d \psi_{(x,y,t)}(0,0,\partial_t) &= (\cdot,\cdot,\der[s]\gamma(x,s)\big|_t,\der[s]\gamma(y,s)\big|_t) \tag*{with $\der[s]\gamma(x,s)\big|_t, \der[s]\gamma(y,s)\big|_t \in \Tan_p^1 S$ distinct since $x \neq y$.}
\end{align*}
The tangent space to $\cal{F}_1$ at $(\sf{v}_1,\sf{v}_2,p,p)$ is the set of $(0,\cdot,\sf{u},\sf{u})$ with $\sf{u} \in \Tan_p S$. We deduce from the previous description that this space is in direct sum with the image of $\d \psi$, that is to say that $\psi$ intersects $\cal{F}_1$ with constant rank $0$. Consequently, $\cal{I}_1 = \psi^{-1}(\cal{F}_1)$ is a submanifold of $(\partial D)^2 \times U$ of dimension $0$, that is to say a discrete set, hence countable. 

Similarly, with $\cal{F}_2 = \{(\sf{v}_1,\sf{v}_2,p,q) \in \R^2 \times S^2 \mid \sf{v}_1 \neq 0,\ \sf{v}_2 = 0 \text{ and } p=q\}$, we show that $\cal{I}_2 = \psi^{-1}(\cal{F}_2)$ is a discrete set hence countable of $(\partial D)^2 \times U$ (in fact equal to $\{(y,x,t) \mid (x,y,t) \in \psi^{-1}(\cal{F}_1)\}$). 

Finally, let us denote $\cal{F}_0 = \{(\sf{v}_1,\sf{v}_2,p,q) \in \R^2 \times S^2 \mid \sf{v}_1 = \sf{v}_2 = 0 \text{ and } p=q\}$ as well as $(\partial D)^{(2)} = (\partial D)^2 \setminus \{(x,x) \mid x \in \partial D\}$ and $\psi_0 = \psi_{|(\partial D)^{(2)} \times U}$. As previously, $\psi_0$ intersects $\cal{F}_0$ with constant rank zero hence $\cal{I}_0 = \psi_0^{-1}(\cal{F}_0)$ is a submanifold of $(\partial D)^{(2)} \times U$ of dimension $0$, hence is countable. 

Let us suppose that a sequence $(x_n,y_n,t)$ of points of one of the $\cal{I}_k$ converges in $(\partial D)^2 \times U$ towards $(x,x,t)$. $x$ is a non-degenerate focal point because $t \in U$, which implies (step \ref{étape:points_focaux_génériques}) that $\gamma$ is a homeomorphism between a neighbourhood of $(x,t) \in \partial D \times \R$ and a neighbourhood of $\gamma(x,t) \in S$, thus that $x_n=y_n$ from some rank onwards, which is absurd. The $\cal{I}_k \subset (\partial D)^2 \times U$ therefore has no accumulation point in the diagonal $\{(x,x,t)\}$. This shows that $\cal{I}_0$ is closed as a subset of $(\partial D)^2 \times U$. For $k = 1$ or $2$, the definition of $\cal{F}_k$ implies that the boundary of $\cal{I}_k$ in $(\partial D)^2 \times U$ satisfies $\partial \cal{I}_k \subset \cal{I}_0 \cup \{(x,x,t)\}$, thus $\partial \cal{I}_k \subset \cal{I}_0$. This shows that $\cal{I} = \cal{I}_0 \cup \cal{I}_1 \cup \cal{I}_2$ is a closed subset of $(\partial D)^2 \times U$. it corresponds to the set of $(x,y,t)$ with $x \neq y$, such that $\gamma(x,t) = \gamma(y,t)$ and that $(x,t)$ or $(y,t)$ is a focal point. 

The projection $\pi : (\partial D)^2 \times U \to U,\ (x,y,t) \mapsto t$ is proper because $\partial D$ is compact, so that $\pi(\cal{I})$ is a closed subset of $U$. The $\cal{I}_k$ being countable, $\pi(\cal{I})$ also is. This countable closed set is the set of $t \in U$ such that $\gamma_t$ has a self-intersection at a focal point. We now replace $U$ by $U \setminus \pi(\cal{I})$, which is still an open subset of $\R$ with negligible complement. 

\etape[Non-transverse intersections]\label{étape:intersection_génériquement_transverse}
It remains to show that for generic $t \in U$, $\gamma_t$ self-intersects transversely. We have already eliminated the case of a self-intersection at focal points. As $\gamma_t$ is $\cal{C}^\infty$ and the normal vector to $\gamma_t$ at $\gamma(x,t)$ is $\der[s]\exp(s \sf{n}_x)$, it suffices to show that for generic $t$, one cannot find distinct $x,y \in \partial D$ such that $\der[s]\exp(s \sf{n}_x)\big|_t = -\der[s]\exp(s \sf{n}_y)\big|_t$ (in other words, such that $\gamma_t$ self-intersects non-transversely at $\gamma(x,t)$ and $\gamma(y,t)$). Let us consider
\begin{equation*}
	\begin{array}{rrcl}
		G : & (\partial D)^{(2)} \times \R & \longrightarrow & S^2, \\
		\ & (x,y,t) & \longmapsto & (\gamma(x,t),\gamma(y,t)).
	\end{array}
\end{equation*}
As in the previous step, we consider the submanifold $\Delta = \{(x,x) \mid x \in S\} \subset S^2$ (of codimension 2) and we verify that for $(x,y,t) \in G^{-1}(\Delta)$, the image of $\d G_{(x,y,t)}$ intersects $\Tan \Delta$ transversely. Indeed, thanks to the previous step, neither $(x,t)$ nor $(y,t)$ is a focal point, so that $G^{-1}(\Delta)$ is a submanifold $V$ (of dimension $1$) of $(\partial D)^{(2)} \times \R$, which does not accumulate on $\Delta \times U$ (as in the previous step) thus is even a submanifold of $(\partial D)^2 \times U$. The map $\pi_t : (\partial D)^2 \times \R \to \R, (x,y,t) \mapsto t$ restricts to $V$ and we verify that its critical points are the $(x,y,t)$ such that $t \in U$ and $\gamma_t$ has a non-transverse intersection at $\gamma(x,t) = \gamma(y,t)$. The set of its critical values, as previously, is a measure zero closed subset of $U$, which we subtract from $U$.

We obtain that for $t \in U$, $\gamma_t$ has only non-degenerate focal points, no intersections at focal points and only transverse intersections. One can moreover show (as previously) that the set of $t \in U$ such that $\gamma_t$ self-intersects at three points is discrete ; we subtract this set from $U$. 

For all $t$ in $U$, $\partial D_t$ therefore satisfies the properties of Definition \ref{définition:bord_C_infini}. This shows that $D_t$ has piecewise $\cal{C}^\infty$ boundary, thus completes the proof of Proposition \ref{proposition:régularité_bord_Dt}. 
\end{proof}

\begin{proof}[Proof of Corollary \ref{corollaire:contrôle_bord_Dt}]
Given a relatively compact open subset $D$ of $S$ with $\cal{C}^\infty$ boundary, we have defined the family $(\theta_i)_{i \in I}$, $-\pi < \theta_i < \pi$, of its exterior angles (see Definition \ref{définition:bord_C_infini}). At every smooth point of a regular path used to parametrise the boundary of $D$, we can define the geodesic curvature $k$ of $\partial D$. 
\begin{fait}[Gau\ss-bonnet formula with angles, see \cite{GHL04}]
	We have $\sum_I \theta_i + \int_{\partial D} k = 2\pi \chi(D) - \int_D K$. 
\end{fait}

\begin{remarque}
	Note that $t \mapsto \liminf_{\epsilon \to 0} \frac{\Vol \{x \in S \mid \d^S(x,\partial D_t) \leq \epsilon\}}{2\epsilon}$, called Minkowski content of $\partial D_t$, extends $t \mapsto \Vol(\partial D_t)$ to $\R$. 
\end{remarque}

The function which to $t$ in $U$ associates $\Vol(\partial D_t)$ is well-defined. It is even differentiable, of derivative
\begin{align*}
	\frac{\d \Vol(\partial D_s)}{\d s}\bigg|_{t} &= \sum_{I_t} \theta_i + \int_{\partial D_t} k \tag*{with $(\theta_i)_{i \in I_t}$ the exterior angles of $D_t$}\\
	&= 2\pi \chi(\bar{D_t}) - \int_D K \tag*{by Gau\ss-Bonnet formula}\\
	&\leq 2\pi \Card \pi_0(\bar{D}) + \int_S \abs{K} \tag*{since $\chi(\bar{D_t}) \leq 2 \Card \pi_0(\bar{D_t}) \leq 2 \Card \pi_0(\bar{D})$.}
\end{align*}
Moreover, the function $t \mapsto \Vol(\partial D_t)$ can only decrease in its points of discontinuity. Indeed, we have $\Vol(\partial D_t) = \int_{A_t} \norm{\gamma_t'}$ with $A_t = \{x \in \partial D \mid \gamma_t(x) \notin D_t\}$ a measurable subset of $\partial D$ and $t \mapsto A_t$ decreasing, hence $\Vol(\partial D_s)$ is bounded above, for $s$ in a right neighborhood of $t$, by the continuous function $s \mapsto \int_{A_t} \norm{\gamma_s'}$, which is equal to $\Vol(\partial D_t)$ en $t$. 

The bound from above of the derivative of $t \mapsto \Vol(\partial D_t)$ hence concludes the proof of Corollary \ref{corollaire:contrôle_bord_Dt}. 
\end{proof}

\bibliographystyle{alpha}
\bibliography{bibli}

\end{document}